\documentclass{amsart}

\usepackage[utf8]{inputenc}
\usepackage{xcolor}
\usepackage{amssymb}
\usepackage{tikz-cd}
\usepackage{tikz}
\usetikzlibrary{arrows.meta,decorations.pathreplacing, positioning}
\usepackage{graphicx}
\usepackage{float}
\usepackage{booktabs}
\usepackage{array}
\usepackage{caption}
\usepackage{comment}
\usepackage{placeins}
\usepackage{amsthm}

\usepackage{aliascnt}

\usepackage{hyperref}
\usepackage{xcolor}
\usepackage{soul}

\usepackage[capitalise]{cleveref}

\newcounter{marialauracomments}

\theoremstyle{plain}

\newtheorem{theorem}{Theorem}[section]

\newaliascnt{proposition}{theorem}
\newtheorem{proposition}[proposition]{Proposition}
\aliascntresetthe{proposition}

\newaliascnt{corollary}{theorem}
\newtheorem{corollary}[corollary]{Corollary}
\aliascntresetthe{corollary}

\newaliascnt{lemma}{theorem}
\newtheorem{lemma}[lemma]{Lemma}
\aliascntresetthe{lemma}

\newaliascnt{conjecture}{theorem}

\aliascntresetthe{conjecture}

\newaliascnt{claim}{theorem}

\aliascntresetthe{claim}

\theoremstyle{definition}

\newaliascnt{definition}{theorem}
\newtheorem{definition}[definition]{Definition}
\aliascntresetthe{definition}

\newaliascnt{question}{theorem}

\aliascntresetthe{question}

\newaliascnt{example}{theorem}

\aliascntresetthe{example}

\theoremstyle{remark}

\newaliascnt{remark}{theorem}
\newtheorem{remark}[remark]{Remark}
\aliascntresetthe{remark}

\theoremstyle{plain}

\newtheorem{theoremA}{Theorem}

\newtheorem{corollaryB}{Corollary}

\newtheorem{theoremC}{Theorem}

\newtheorem{corollaryD}{Corollary}

\Crefname{theorem}{theorem}{theorems}
\Crefname{theorem}{Theorem}{Theorems}

\Crefname{proposition}{proposition}{propositions}
\Crefname{proposition}{Proposition}{Propositions}

\Crefname{corollary}{corollary}{corollaries}
\Crefname{corollary}{Corollary}{Corollaries}

\Crefname{lemma}{lemma}{lemmas}
\Crefname{lemma}{Lemma}{Lemmas}

\Crefname{definition}{definition}{definitions}
\Crefname{definition}{Definition}{Definitions}

\Crefname{conjecture}{conjecture}{conjectures}
\Crefname{conjecture}{Conjecture}{Conjectures}

\Crefname{question}{question}{questions}
\Crefname{question}{Question}{Questions}

\Crefname{example}{example}{examples}
\Crefname{example}{Example}{Examples}

\Crefname{claim}{claim}{claims}
\Crefname{claim}{Claim}{Claims}

\Crefname{remark}{remark}{remarks}
\Crefname{remark}{Remark}{Remarks}

\Crefname{theoremA}{theorem}{theorems}
\Crefname{theoremA}{Theorem}{Theorems}

\Crefname{corollaryB}{corollary}{corollaries}
\Crefname{corollaryB}{Corollary}{Corollaries}

\Crefname{theoremC}{theorem}{theorems}
\Crefname{theoremC}{Theorem}{Theorems}

\Crefname{corollaryD}{corollary}{corollaries}
\Crefname{corollaryD}{Corollary}{Corollaries}

\newcommand{\F}{\mathbb{F}}
\newcommand{\N}{\mathbb{N}}

\newcommand{\Q}{\mathbb{Q}}
\newcommand{\Z}{\mathbb{Z}}

\DeclareMathOperator{\Sym}{Sym}
\DeclareMathOperator{\Aut}{Aut}
\DeclareMathOperator{\End}{End}
\DeclareMathOperator{\St}{St}

\newcommand{\TT}{\mathcal{T}}
\newcommand{\BB}{\mathcal{B}}
\newcommand{\LL}{\mathcal{L}}
\newcommand{\MM}{\mathcal{M}}

\numberwithin{equation}{section}

\newcommand{\DiamondBetween}[3]{% #1=yTop, #2=yBottom, #3=half width
\pgfmathsetmacro{\ymid}{(#1+#2)/2}
\draw[thick]
(0,#1) -- (#3,\ymid) -- (0,#2) -- (-#3,\ymid) -- cycle;
}

\calclayout

\title[Lower central series of a family of non-periodic GGS-groups]%
{On the lower central series of a large family of non-periodic GGS-groups}

\author[G.\,A.\ Fern\'{a}ndez-Alcober]{Gustavo\,A.\ Fern\'{a}ndez-Alcober}
\address{Gustavo\,A.\ Fern\'{a}ndez-Alcober: Department of Mathematics, University of the Basque Country UPV/EHU, 48080 Bilbao, Spain}
\email{gustavo.fernandez@ehu.eus}
 
\author[M.\,E.\ Garciarena]{Mikel Eguzki Garciarena}
\address{Mikel Eguzki Garciarena: Heinrich Heine University Düsseldorf, Faculty of Mathematics and Natural Sciences, Mathematical Institute,
Universitätsstr. 1, 40225 Düsseldorf, Germany}
\email{mikel.eguzki.garciarena.perez@hhu.de}
 
\author[M.\ Noce]{Marialaura Noce}
\address{Marialaura Noce: Dipartimento di Informatica, Universit\`a di Salerno, 84084 Fisciano, Italy}
\email{mnoce@unisa.it}

\thanks{\\[-5pt] \indent
The first author is supported by the Basque Government grant IT1913-26, and by grant PID2025-171325NB-I00 funded by MICIU/AEI/10.13039/501100011033 and FEDER/EU.
The second author was supported by the ``National Group for Algebraic and Geometric Structures, and their Applications"\break
(GNSAGA -- INdAM).
The third author is partially supported by project SERICS (PE00000014) under the MUR National Recovery and Resilience Plan funded by the European Union -- NextGenerationEU, and partially by the European Union -- Next Generation EU, Missione 4 Componente 1 CUP B53D23009410006, PRIN2022-2022PSTWLB -- Group Theory and Applications.}

\keywords{Groups acting on trees, GGS-groups, lower central series, lower central width}
\subjclass[2020]{20E08, 20F14}

\begin{document}

\begin{abstract}
For an odd prime $p$, we fully determine the lower central series of a large family of non-periodic GGS-groups, which has a density of roughly $(\frac{p-1}{p})^2$ within all GGS-groups.
We also describe the structure of the $\F_p$-Lie algebra associated to the lower central series of these
groups.
These results represent a major step forward in the knowledge of the lower central series of distinguished classes of branch groups, which to date was basically restricted to the Grigorchuk group.
As part of our results, we obtain the exact values of the indices between consecutive terms of the lower central series, and we show that these groups, as well as their profinite completions, have lower central width equal to $2$.
In particular, this confirms a conjecture of Bartholdi, Eick, and Hartung about the generalised Fabrykowski-Gupta groups.
\end{abstract}

\maketitle
\enlargethispage{4pt}%remove if something weird happens

\section{Introduction}

For a given integer $d\ge 2$, the \emph{$d$-regular rooted tree} or simply the \emph{$d$-adic tree} is a tree
with a distinguished vertex (the root) of degree $d$ and with every other vertex of degree $d+1$.
Groups of automorphisms of regular rooted trees are an important source of examples of groups with
very interesting properties.
Prominent among these are the (first) \emph{Grigorchuk group} and the \emph{Grigorchuk-Gupta-Sidki groups}
(GGS-groups, for short).
The Grigorchuk group was constructed by R.I.\ Grigorchuk \cite{Grigorchuk80} in 1980 and has been extensively studied.
It was the first example of a group of intermediate word growth, thus providing an answer to Milnor's Problem, 
as well as the first example of a group that is amenable but not elementary amenable, answering Day's Problem; cf.\ 
\cite{Grigorchuk85} for both these properties.
Also it is a $3$-generator infinite $2$-group, and so it gives a negative answer to the
General Burnside Problem.
For an arbitrary prime $p$, examples of $2$-generator infinite $p$-groups can be found within the class of GGS-groups,
such as the second Grigorchuk group \cite{Grigorchuk80} and the Gupta-Sidki groups \cite{Gupta-Sidki}.
The GGS-groups can be defined by specifying a non-zero \emph{defining vector} $\mathbf{e}\in (\Z/d\Z)^{d-1}$;
see \Cref{sec:uniseriality} for details.
They have received much attention in the past few decades, particularly in the case of the $p$-adic tree
for a prime $p$: see \cite{Gustavo} for their basic theory and Hausdorff dimension, \cite{Gustavo-Alejandra-Jone}
for the position of their normal subgroups, \cite{Francoeur-Thillaisundaram,Pervova} for their maximal subgroups,
or \cite{Petschick} for the isomorphism problem.
It is worth mentioning that GGS-groups have also found applications outside the realm of group theory; as shown in
\cite{Gul-Uria}, they can be used for the determination of Beauville surfaces.

\vspace{8pt}

The full group of automorphisms $\Aut(\TT)$ of a regular rooted tree $\TT$ is a profinite group.
From now onwards we will always assume that $\TT$ is $p$-adic for a prime $p$.
In particular, when we refer to GGS-groups, we do it in this setting.
Then the set of all automorphisms whose local action at every vertex of $\TT$ is a power of the cycle $(1\ 2\ \ldots \ p)$
is a Sylow pro-$p$ subgroup of $\Aut(\TT)$, which we denote by $\Gamma(p)$.
Note that both the Grigorchuk group (for $p=2$) and the GGS-groups are subgroups of $\Gamma(p)$.
A subgroup $G$ of $\Gamma(p)$ is residually a finite $p$-group, and its topological closure in $\Aut(\TT)$ is
a pro-$p$ group, which is isomorphic to the profinite completion of $G$ provided that $G$ satisfies the
\emph{congruence subgroup property} (see \Cref{sec:uniseriality} for the definition).
It has been proved that the Grigorchuk group and all GGS-groups satisfy the congruence subgroup property, with the
only exception of the GGS-group with constant defining vector; also, these groups are just infinite, i.e.\ every
non-trivial normal subgroup is of finite index (see \cite[Section 12]{Grigorchuk-New Horizons} and
\cite[Theorem 2.7]{Gustavo-Alejandra-Jone} for all these properties).

\vspace{8pt}

Recall that a pro-$p$ group $G$ is said to have \emph{finite coclass} if for some positive integer $i_0$ we have
$|\gamma_i(G):\gamma_{i+1}(G)|=p$ for all $i\ge i_0$.
Then
\[
\log_p |G:\gamma_i(G)| - i
\]
takes a constant value for $i\ge i_0$, which is called the \emph{coclass} of $G$.
Leedham-Green \cite{LeedhamGreen} and Shalev \cite{Shalev} independently proved that, for every prime $p$, there are
only finitely many infinite pro-$p$ groups of each fixed coclass, and that these groups are soluble.
After that, attention was turned towards pro-$p$ groups of finite lower central width, a natural generalisation of
finite coclass.
A residually nilpotent group $G$ is said to have \emph{finite lower central width} if there exists a constant $C$ such
that $|\gamma_i(G):\gamma_{i+1}(G)|\le C$ for all $i\ge 1$.
If $G$ has finite lower central width and all indices $|\gamma_i(G):\gamma_{i+1}(G)|$ are powers of a prime $p$
(in particular, if $G$ is a pro-$p$ group), then one usually defines the \emph{lower central width} of $G$ as
\[
\max_{i\ge 1} \, \log_p |\gamma_i(G):\gamma_{i+1}(G)|.
\]
A wealth of information about linear pro-$p$ groups of finite width can be found in the book \cite{Klaas-LeedhamGreen-Plesken}
by Klaas, Leedham-Green and Plesken, including insoluble examples.
In 1996, Zelmanov \cite{Zelmanov} conjectured that a just infinite pro-$p$ group of finite lower central width is
either soluble, $p$-adic analytic (so linear over $\Q_p$), or commensurable to a positive part of a loop group or to the Nottingham group.
This conjecture was settled in the negative by Rozhkov \cite{Rozhkov}, by proving that the Grigorchuk group (and as a consequence also its profinite completion) has width $3$.
Indeed, the Grigorchuk group is a branch group, and so is its profinite completion.
Now as shown by \'Abert, weakly branch groups do not satisfy any laws and are not linear over any field
(see \cite{Abert05} and \cite{Abert06}, respectively).
Rozhkov's proof contained some gaps that were partially fixed in \cite{Rozhkov-2}, and eventually the result was fully confirmed by Bartholdi and Grigorchuk \cite{Bartholdi-Grigorchuk} in 2000.

\vspace{8pt}

The results by Rozhkov, Bartholdi, and Grigorchuk do not only determine the lower central width of the Grigorchuk group $G$, but also provide a detailed description of the terms of its lower central series.
In particular, they show that $|G:G'|=2^3$, that $|G':\gamma_3(G)|=2^2$, and that
\[
|\gamma_i(G):\gamma_{i+1}(G)|
=
\begin{cases}
2^2, & \text{if $2^m+1\le n \le 3\cdot 2^{m-1}$,}
\\
2, & \text{if $3\cdot 2^{m-1}+1\le n\le 2^{m+1}$,}
\end{cases}
\]
where $m$ runs over the positive integers.

\vspace{8pt}

Information about the lower central series of other distinguished (weakly) branch subgroups of $\Aut(\TT)$ is very scarce.
Regarding GGS-groups, all results are limited to the original Gupta-Sidki examples and to the so-called
\emph{generalised Fabrykowski-Gupta groups}.
As already mentioned, the Gupta-Sidki groups are examples of $2$-generator infinite periodic groups
(actually, $p$-groups); they correspond to the defining vector $(1,-1,0,\overset{p-3}{\ldots},0)$ in $\F_p^{p-1}$.
In turn, the Fabrykowski-Gupta group was originally defined in \cite{Fabrykowski-Gupta85} for $p=3$, and then generalised to an arbitrary
odd prime in \cite{Bartholdi-Eick-Hartung}, by considering the defining vector $(1,0,\overset{p-2}{\ldots},0)$.
Contrary to the Gupta-Sidki examples, the generalised Fabrykowski-Gupta groups are not periodic.
For $p=3$, the Fabrykowski-Gupta group is known to have intermediate growth,
see \cite{Bartholdi-Pochon}.
However, the growth of the generalised Fabrykowski-Gupta groups is not known for $p\ge 5$.
In \cite[Corollary 3.9]{Bartholdi}, Bartholdi showed that the Gupta-Sidki $3$-group has infinite lower central width,
completing work of Vieira \cite{Vieira}, who had previously determined the indices between the first few terms
of the lower central series.
In the same paper (Corollary 3.14), Bartholdi also showed that the Fabrykowski-Gupta group (over the $3$-adic tree)
has lower central width $2$.
In \cite{Bartholdi-Eick-Hartung}, Bartholdi, Eick, and Hartung developed a nilpotent quotient algorithm that they applied
to obtain by computer the indices between the first terms of the lower central series of both the Gupta-Sidki groups and
the generalised Fabrykowski-Gupta groups for small primes.
They performed similar calculations for other weakly branch groups not in the class of GGS-groups, namely for the
Brunner-Sidki-Vieira group, the Grigorchuk supergroup, and the Basilica group.
Based on this information, they made several conjectures about the behaviour of the lower central series of these groups.
In particular, they conjectured (Conjecture 17) that the lower central width of the generalised Fabrykowski-Gupta group
over the $p$-adic tree is $2$ for every prime $p$.

\vspace{8pt}

The goal of this paper is to describe the lower central series of the groups in a wide class of non-periodic GGS-groups,
containing the generalised Fabrykowski-Gupta groups.
In particular, we determine the values of all indices $|\gamma_i(G):\gamma_{i+1}(G)|$ for these groups, and as a
consequence we show that they all have lower central width $2$.
Thus we confirm Conjecture 17 from \cite{Bartholdi-Eick-Hartung} and extend its validity to many other GGS-groups.
The class of GGS-groups to which our result applies is the following.

\begin{definition}
\label{definition FG-type}
Let $G$ be a GGS-group with defining vector $\mathbf{e}=(e_1,\ldots,e_{p-1})\in \F_p^{p-1}$, where $p$ is an odd prime.
Then we define the following two elements of $\F_p$:
\[
\varepsilon(\mathbf{e})=\sum_{i=1}^{p-1} \, e_i
\qquad
\text{and}
\qquad
\delta(\mathbf{e})=\sum_{i=1}^{p-1} \, ie_i,
\]
and we say that $G$ is of \emph{FG-type} if $\varepsilon(\mathbf{e})\ne 0$ and $\delta(\mathbf{e})\ne 0$.
\end{definition}

Since we will always be working with a fixed GGS-group given by a fixed defining vector, in the remainder we
will simply write $\varepsilon$ and $\delta$ instead of $\varepsilon(\mathbf{e})$ and $\delta(\mathbf{e})$.
Obviously, the generalised Fabrykowski-Gupta groups are of FG-type, having $\mathbf{e}=(1,0,\overset{p-2}{\ldots},0)$, while
the GGS-group with constant defining vector is not of FG-type, since $\delta=0$ in this case.
It is well-known \cite[Theorem 1]{Vovkivsky} that a GGS-group is periodic if and only if $\varepsilon=0$;
thus all groups of FG-type are non-periodic.
One can easily see that, out of the possible $p^{\,p-1}-1$ defining vectors of GGS-groups, the number of
vectors for which $\varepsilon\ne 0$ and $\delta\ne 0$ is $p^{\,p-1}-2p^{\,p-2}+p^{\,p-3}$.
As a consequence, the proportion of groups of FG-type among all GGS-groups is roughly $(\frac{p-1}{p})^2$, and a majority
of GGS-groups are of FG-type for $p\ge 5$.
Similarly, the proportion of groups of FG-type among non-periodic GGS-groups is $\frac{p-1}{p}$.
On the other hand, in the case $p=3$, even if there are four different choices for the defining
vector of a GGS-group of FG-type, the corresponding groups are all isomorphic to the Fabrykowski-Gupta group by
Corollary 3.6 of \cite{Petschick}.

\vspace{8pt}

Before stating the main theorem of the paper, which describes the lower central series of GGS-groups of FG-type,
we need to introduce the following notation.
We define the sequence of integers $\{\ell(m)\}_{m\ge 0}$ and $\{r(m)\}_{m\ge 0}$ recursively by means of
\begin{align*}
\ell(m)
=
\begin{cases}
1,
&
\text{if $m=0$,}
\\
p,
&
\text{if $m=1$,}
\\
p (\ell(m-1)-1)-1,
&
\text{if $m\ge 2$,}
\end{cases}
\end{align*}
and
\begin{align*}
r(m)
=
\begin{cases}
2,
&
\text{if $m=0$,}
\\
p+1,
&
\text{if $m=1$,}
\\
p (r(m-1)-1)-1,
&
\text{if $m\ge 2$.}
\end{cases}
\end{align*}
Thus for every $m\ge 2$ we have
\vspace{-5pt}
\[
\ell(m) = p^m - p^{m-1} - 2 \sum_{i=1}^{m-2} p^i -1
\vspace{-8pt}
\]
\vspace{-5pt}
and
\[
r(m) = p^m - 2 \sum_{i=1}^{m-2} p^i -1 = \ell(m) + p^{m-1}.
\]
On the other hand, we define the following elements of a GGS-group $G$ of FG-type.
If $a$ and $b$ are the canonical generators of $G$ (see \Cref{sec:uniseriality}), we set $x(1)=b$
and
\[
x(i) = [b,a^{\varepsilon}b,\overset{i-1}{\ldots},a^{\varepsilon}b],
\]
for every $i\ge 2$, where $\varepsilon$ is as above.
We also define $y_1(1)=a^{\varepsilon}b$ and
\[
y_j(i) = [x(j-1),b,a^\varepsilon b,\overset{i-j}{\dots},a^\varepsilon b], 
\]
for $2\le j\le i$.
(See our convention for iterated commutators in the Notation section at the end of this introduction.)
As it turns out, we are only going to need the elements $y_j(i)$ when $j$ is of the form $\ell(m)$ for some
$m$, and then only for $i\in [\ell(m),r(m)-1]$.
Since these intervals are disjoint, we can simplify the notation and write simply $y(i)$ instead of
$y_{\ell(m)}(i)$ when $i$ lies in the corresponding interval.
Observe that $x(i),y(i)\in \gamma_i(G)$ in every case.

\begin{theoremA}
\label{main}
Let $p$ be an odd prime and let $G$ be a GGS-group of FG-type defined over the $p$-adic tree.
Then the following hold:
\begin{enumerate}
\item 
If $\ell(m)\le i < r(m)$ for some $m\ge 0$, then 
\[
|\gamma_i(G):\gamma_{i+1}(G)|=p^2.
\]
Also $\gamma_i(G)=\langle x(i), y(i) \rangle \, \gamma_{i+1}(G)=\langle x(i), y(i) \rangle^G$.
\vspace{5pt}
\item 
If $r(m)\le i< \ell(m+1)$ for some $m\ge 0$ then
\[
|\gamma_i(G):\gamma_{i+1}(G)|=p.
\]
Also $\gamma_i(G)=\langle x(i)\rangle \, \gamma_{i+1}(G)=\langle x(i) \rangle^G$.
\end{enumerate}
In particular, $G$ is a group of lower central width $2$.
\end{theoremA}

We remark that \Cref{main} does not hold in this form for non-periodic GGS-groups that are regular branch over $G'$ but are not of FG-type
(see \Cref{indices not FG-type}).
On the other hand, a standard argument in profinite group theory (see \Cref{index of gammai equal}) yields the following corollary.

\begin{corollaryB}
Let $p$ be an odd prime and let $G$ be a GGS-group of FG-type defined over the $p$-adic tree.
Then the profinite completion of $G$ is a pro-$p$ group of lower central width $2$.
\end{corollaryB}

It has been brought to our attention by B.\ Klopsch and A.\ Thillaisundaram that they have more generally proved that non-periodic multi-EGS groups with the congruence subgroup property
have finite lower central width (see \cite{Klopsch-Thillaisundaram}).
However, their methods do not provide the description of the lower central series nor the exact
values of the indices $|\gamma_i(G):\gamma_{i+1}(G)|$ for every $i$.

\vspace{8pt}

Our methods can also be used to describe the structure constants of the Lie algebra associated to the lower central series of a GGS-group of FG-type.
Recall that we can associate a Lie ring $L(G)$ to an arbitrary group $G$ as follows.
We consider the abelian group
\[
L(G) = \oplus_{i\ge 1} \, \gamma_i(G)/\gamma_{i+1}(G),
\]
with componentwise operation inherited from the group multiplication of $G$, but written additively.
The Lie bracket operation of $L(G)$ is first defined on homogeneous elements by means of
\begin{equation}
\label{dfn lie bracket}
[g\gamma_{i+1}(G),h\gamma_{j+1}(G)] = [g,h] \gamma_{i+j+1}(G),
\end{equation}
where $g\in\gamma_i(G)$ and $h\in\gamma_j(G)$, and is then extended by linearity to the whole of $L(G)$.
For simplicity, given an element $g\in \gamma_i(G)\smallsetminus \gamma_{i+1}(G)$, we still write as $g$ its image in its natural homogeneous component $\gamma_i(G)/\gamma_{i+1}(G)$, and for
another element $h\in \gamma_j(G)\smallsetminus \gamma_{j+1}(G)$, we write $[g,h]_L$ to denote
their Lie bracket in $L(G)$, to distinguish it from the group commutator $[g,h]$.
By \eqref{dfn lie bracket}, we have $[g,h]_L=[g,h]\gamma_{i+j+1}(G)$.
The addition and Lie bracket defined above make $L(G)$ into a Lie ring, and if furthermore
$\gamma_i(G)/\gamma_{i+1}(G)$ is an elementary abelian $p$-group for all $i\ge 1$, where $p$ is a fixed prime, then $L(G)$ is actually an $\F_p$-Lie algebra.
This is the case for GGS-groups, and also for other subgroups of the group of automorphisms of a rooted tree, such as the Grigorchuk group (for $p=2$).
These, and other similar Lie algebras derived from branch groups, were studied by Bartholdi and Grigorchuk in \cite{Bartholdi-Grigorchuk} in 2000, and in two later papers by Bartholdi \cite{Bartholdi,Bartholdi2}.

\vspace{8pt}

The results in \Cref{main} provide an $\F_p$-basis for $L(G)$ when $G$ is a GGS-group of FG-type, consisting of the images of the elements $x(i)$ and $y(i)$.
Then the structure of the Lie algebra $L(G)$ is fully determined by giving the values of the Lie brackets between any two basis elements, expressed as a linear combination of the basis.
Actually, due to the Jacobi identity, and to the fact that all elements $x(i)$ and $y(i)$ are defined as long commutators all of whose components are either $a^{\varepsilon}b$ or $b$, all Lie brackets between basis elements can be (at least theoretically) determined from the special values
$[x(i),a^{\varepsilon}b]_L$, $[x(i),b]_L$, $[y(i),a^{\varepsilon}b]_L$, and $[y(i),b]_L$.
In Theorem C below, we list all these values under the assumption that $\varepsilon=1$.
This is not really a restriction, since two proportional non-zero vectors define the same GGS-group; thus, we can always divide the defining vector $\mathbf{e}$ by $\varepsilon(\mathbf{e})$ in order to represent the same GGS-group in such a way that $\varepsilon=1$.
Following the terminology in \cite{Bartholdi-Grigorchuk}, Theorem C amounts to describing the Cayley graph of the Lie algebra $L(G)$.
We note that such a description in the case of branch groups was only known so far for the Grigorchuk group and an overgroup of it (see \cite{Bartholdi-Grigorchuk}), and for the Gupta-Sidki $3$-group and the Fabrykowski-Gupta group (see \cite{Bartholdi}).

\begin{theoremC}
\label{theoremC}
Let $G$ be a GGS-group of FG-type with $\varepsilon=1$, and let $L(G)$ be the Lie algebra associated to the lower central series of $G$.
Then we have the following values of Lie brackets in $L(G)$:
\begin{enumerate}
\item
$[x(i),ab]_L=x(i+1)$ for every $i\ge 1$.
\item
$[y(1),ab]_L=0$, $[y(p),ab]_L=2x(p+1)$ and, for every $m\ge 2$, $[y(i),ab]_L=y(i+1)$
for $\ell(m)\le i\le r(m)-2$ and $[y(r(m)-1),ab]_L=0$.
\item
If $r(m)\le i\le \ell(m+1)-1$ with $m\ge 0$, then
\[
[x(i),b]_L
=
\begin{cases}
x(i+1),
&
\text{if $i\equiv -1 \hspace{-8pt} \mod p$, $i\ne p-1$,}
\\
y(i+1),
&
\text{if $i=\ell(m+1)-1$,}
\\
0,
&
\text{otherwise.}
\end{cases}
\]
\item
$[x(1),b]_L=0$, $[x(p),b]_L=-x(p+1)$.
\item
For every $m\ge 2$, $[x(r(m)-1),b]_L=0$.
\item
Let $m\ge 2$ and $\ell(m)\le i\le r(m)-2$ with $i\not\equiv -2,-1 \mod p$.
Then $[x(i),b]_L=0$.
\item
Let $m\ge 2$ and $\ell(m)\le i\le r(m)-2$, with $i\equiv -2 \mod p$.
If we write $i=\ell(m)-1+s$ with $s\equiv 0\mod p$, then
\[
[x(i),b]_L
=
\eta y(i+1),
\]
where $\eta=(-1)^t$ if $s=p^{i_1}+\cdots+p^{i_t}$, with $i_1<\cdots<i_t$, is a sum of $t$ different powers of $p$, and $\eta=0$ otherwise.
\item
Let $m\ge 2$ and $\ell(m)\le i\le r(m)-2$, with $i\equiv -1 \mod p$.
If we write $i=\ell(m)-1+s$ with $s\equiv 1\mod p$, then
\[
[x(i),b]_L
=
x(i+1) + \eta y(i+1),
\]
where $\eta=(-1)^t$ if $s=p^{i_1}+\cdots+p^{i_t}$, with $0=i_1<\cdots<i_t$, is a sum of $t$ different powers of $p$, and $\eta=0$ otherwise.
\item
$[y(1),b]_L=-x(2)$.
\item
Let $m\ge 1$ and $\ell(m)\le i\le r(m)-1$.
Then
\[
[y(i),b]_L
=
\begin{cases}
y(i+1),
&
\text{if $i\equiv -2 \hspace{-6pt} \mod p$, $i\ne r(m)-1$,}
\\
0,
&
\text{otherwise.}
\end{cases}
\]
\end{enumerate}
\end{theoremC}

In \Cref{sec:lie algebra} we depict the so-called Cayley graph of $L(G)$, which encodes all the
information in Theorem C in a graphical way, easier to understand.
Note that the structure constants do not depend on the defining vector $\mathbf{e}$ of the group $G$, and that
the subspace $\oplus_{m\ge 2} \, \langle y(r(m)-1) \rangle$ is central in $L(G)$.
Hence we obtain the following corollary.

\begin{corollaryD}
The $\F_p$-Lie algebras associated to the lower central series of the GGS-groups of FG-type are all
isomorphic for a fixed prime $p$, and have infinite dimensional centre.
\end{corollaryD}

We conclude this introduction by sketching the strategy that we follow for the proof of \Cref{main}, which is reflected in the structure of the paper:
\begin{enumerate}
\item
If $G$ is a GGS-group then $|G:G'|=p^2$ and so $|G:\gamma_i(G)|$ is finite for every $i\ge 1$.
Now if the defining vector of $G$ is not constant then $G$ satisfies the congruence subgroup property
and the $n$th level stabiliser $\St_G(n)$ is contained in $\gamma_{i+1}(G)$ for some $n$.
Hence the quotient $\gamma_i(G)/\gamma_{i+1}(G)$ is isomorphic to
$\gamma_i(G_n)/\gamma_{i+1}(G_n)$, where $G_n=G/\St_n(G)$ is a finite $p$-group
(the \emph{$n$th congruence quotient of $G$}).
\item
Thus our goal is to determine the lower central series of $G_n$ for every $n\in\N$.
We accomplish this by induction on $n$.
The cases $n=1$ and $n=2$ are trivial, and we devote \Cref{sec:n at most 3} to the study of the case $n=3$,
which is critical for the induction.
This requires a very careful analysis of the position of some commutators in the generators of $G_3$, as well as
the fact that $G_3$ acts uniserially on the level stabiliser $\St_{G_3}(2)$.
This last property is proved more generally in \Cref{sec:uniseriality} for $G_n$ acting on $\St_{G_n}(n-1)$,
and for an arbitrary non-periodic GGS-group, i.e.\ provided that $\varepsilon\ne 0$.
The condition $\delta\ne 0$ of the definition of groups of FG-type is also essential to settle the case $n=3$.
If $\delta=0$ then we get into further technical difficulties and the structure of the lower central
series of $G_3$ is far from uniform, as shown in \Cref{examples of LCS G3 in general} with some examples worked out with the
help of GAP.
Even more, if $G$ is regular branch over $G'$ with $\varepsilon\ne 0$ and $\delta=0$ then the structure of the
lower central series of $G_3$ is \emph{always different} from that of the groups of FG-type
(see \Cref{indices not FG-type}).
\item
The induction step to pass from $G_{n-1}$ to $G_n$ is completed in \Cref{sec:proof theorem A}.
The approach here is to compare the lower central series of $G_n$ with that of the wreath product
$W(G_{n-1})=G_{n-1}\wr C_p$, by using the map $\psi$ that embeds the group $G$ into $W(G)=G\wr C_p$
(see \Cref{sec:uniseriality}).
As we will see, some terms of the lower central series of $\psi(G_n)$ coincide with terms of the lower central
series of $W(G_{n-1})$, while some other terms appear as ``sandwiches" of the series of $W(G_{n-1})$.
At this point, it is essential to have a very detailed knowledge of the lower central series of groups of the
form $W(P)$, where $P$ is a finite $p$-group with lower central factors of exponent $p$.
A formula yielding the terms of the lower central series of $W(P)$ was already known (see \cite{Eberhard}), but
we need a characterisation of the elements lying in each subgroup $\gamma_i(W(P))$ via explicit bases, word equations,
and normal forms. 
This problem is addressed in \Cref{sec:LCS W(G)}.
We remark that the condition $\delta\ne 0$ plays also a fundamental role here in proving the sandwich property.
\end{enumerate}
On the other hand, \Cref{sec:lie algebra} is devoted to the proof of Theorem C.
Once again, the condition $\delta\ne 0$ plays a role: see for example the proof of \Cref{brackets x(i) with b}.

\vspace{5pt}

We want to close this introduction by proposing the following problem: to determine the behaviour of the lower central series for
all remaining non-periodic regular branch GGS-groups.
Asking for the groups to be non-periodic is a natural restriction, in order to be able to make use of the uniserial action of each
quotient $G_n$ on $\St_{G_n}(n-1)$.
On the other hand, the only GGS-group that is not regular branch is the one with a constant defining vector.
This group is harder to study and behaves very differently from the rest of GGS-groups (see Section 4 of \cite{Gustavo}).
So it is also natural to restrict to regular branch groups.

As a first step, one should deal with the groups that are regular branch over $G'$ but have $\delta=0$.
In this case, we would need to complement \Cref{position commutators}, which already gives some partial information about the position of the
relevant commutators needed to determine where $\gamma_{p+1}(G_3)$ enters the uniserial chain of $\St_{G_3}(2)$.
Knowing this would allow us to establish how many values $i\in\{3,\ldots,p-1\}$ satisfy that $|\gamma_i(G_3):\gamma_{i+1}(G_3)|=p^2$; see the proof
of \Cref{indices not FG-type}.
More detailed calculations should be carried out in order to determine for which indices $i$ this happens.
As shown in \Cref{examples of LCS G3 in general}, there is a wide range of possibilities for this.
After that, one should consider the extension of these results to a general quotient $G_n$, by analysing how the image of the lower central series
of $G_n$ under $\psi$ intertwines with the lower central series of $W(G_{n-1})$.
We already mentioned that the condition $\delta\ne 0$ plays an essential role in proving the sandwich property that holds for groups of FG-type.

The last step of this programme (as is also usual with other questions about GGS-groups, like the study of Hausdorff dimension, see \cite{Gustavo})
would be to deal with the groups that are only regular branch over $\gamma_3(G)$ and not over $G'$.
First of all, we should adapt \Cref{position in uniserial series} to get information about the position of elements in the uniserial chain of
$\St_{G_3}(2)$, and then proceed in a similar fashion to what we have indicated for groups that are regular branch over $G'$.

\vspace{5pt}

\noindent
\textit{Notation.}
In the remainder, we consider only right group actions.
Accordingly, the symmetric group $\Sym(X)$ on a set $X$ (and consequently also the automorphism group of an algebraic or combinatorial structure) is endowed with the composition operation $\sigma\tau$, which means that we apply first the map $\sigma$ and then $\tau$.
In the same vein, we use the notation $g^h=h^{-1}gh$ and $[g,h]=g^{-1}h^{-1}gh$ for conjugates and commutators.
Iterated commutators are defined recursively via $[g_1,\ldots,g_n]=[[g_1,\ldots,g_{n-1}],g_n]$.
Also, if a group $G$ is the semidirect product of a subgroup $H$ and a normal subgroup $N$, we write $G=H\ltimes N$.
Note that, with our conjugation convention, if we represent its elements in the form $hn$ then we have the rule $(h_1n_1)(h_2n_2)=h_1h_2n_1^{h_2}n_2$.

Throughout the paper, $p$ stands for an odd prime number, and $\sigma$ denotes the $p$-cycle
$(1\ 2\ \ldots \ p)$.
If $G$ is a group then we write $B(G)=G\times \overset{p}{\cdots} \times G$ and
$W(G)=G\wr \langle \sigma \rangle=\langle \sigma \rangle \ltimes B(G)$.
As usual, $H^G$ stands for the normal closure of $H$ in $G$.
Also if $H$ and $J$ are subgroups of $G$ and $N\trianglelefteq G$, we write $H\equiv J \mod N$ to mean that
$H$ and $J$ have the same image in the quotient $G/N$, in other words, that $HN=JN$.
In order to avoid cumbersome notation, we will frequently write the image of an element $g\in G$ in a quotient
$G/N$ still as $g$ rather than $gN$ or $\overline g$: it will be clear from the context whether we are considering
$g$ in $G$ or in $G/N$.
This will apply basically to the image of an element $g$ of a GGS-group $G$ in a congruence quotient $G_n$.
On the other hand, if a map $\varphi:K\rightarrow L$ between two groups induces naturally a map between two
quotients of $K$ and $L$, we usually still denote the induced map by the same letter $\varphi$.

\vspace{8pt}

\noindent
\textit{Acknowledgements.}
We thank Rostislav I.\ Grigorchuk for information about the state of the art regarding branch groups
of finite lower central width.
We acknowledge the use of the GAP system \cite{GAP} and its package \cite{automgrp} to explore the lower central series of small congruence quotients of GGS-groups.
We are also grateful to Jan Moritz Petschick for help with calculations with GAP.
Finally, we thank the anonymous referees for their valuable suggestions to improve the presentation of the paper.
In particular, for encouraging us to determine the structure constants of the Lie algebra associated to the lower central series, for providing us with some references regarding the lower central series of wreath products that we had missed, and for drawing our attention to a mistake in our previous formulation of \Cref{indices not FG-type}.

\section{Preliminaries and uniseriality in non-periodic GGS-groups}
\label{sec:uniseriality}

In this section we first introduce all preliminaries about GGS-groups, and more specifically about groups
of FG-type, that we will need throughout the paper.
Our main reference for the basic theory of GGS-groups is the article \cite{Gustavo}.
Then we move on to study the action of the congruence quotient $G_n$ of a GGS-group $G$ on the level
stabiliser $\St_{G_n}(n-1)$, and we prove that this action is uniserial if $G$ is not periodic.

\vspace{8pt}

Let $X=\{1,\ldots,p\}$.
Our model of the $p$-adic tree $\TT$ consists in taking as vertices the elements of the free monoid
$X^*$, and connecting $v$ to $vx$ with an edge for every $v\in X^*$ and every $x\in X$.
The root of $\TT$ is then the empty word $\varnothing$, and the words of length $n$ constitute the $n$th
level of the tree.
Then $\Aut(\TT)$, the set of automorphisms of $\TT$ as a graph, is a group with respect to the operation of
composition, which we write by juxtaposition: if $f,g\in\Aut(\TT)$ then $fg$ is the automorphism that results
from applying first $f$ and then $g$.
This is consistent with considering the action of $\Aut(\TT)$ on $\TT$ as a right action: we denote the image of
$f\in\Aut(\TT)$ on $v\in X^*$ as $vf$ or $(v)f$.
Also we write $\St(n)$ for the normal subgroup of $\Aut(\TT)$ consisting of all automorphisms acting trivially on
the $n$th level.

\vspace{8pt}

Every automorphism $f$ of $\TT$ can be described by giving the permutation $\sigma(f)\in \Sym(p)$ that it induces
on the vertices of the first level, together with the so-called sections $f|_x$ at every vertex $x$ of the first
level.
These sections are defined via the formula
\[
(xv)f = xf \cdot vf|_x,
\]
for every $v\in X^*$ and $x\in X$.
Then we have a group isomorphism
\[
\begin{matrix}
\psi & \colon & \Aut(\TT) & \longrightarrow & \Aut(\TT) \wr \Sym(p)
\\[5pt]
& & f & \longmapsto & (f|_1,\ldots,f|_p) \, \sigma(f).
\end{matrix}
\]
If $f\in\St(1)$ then we simply write $\psi(f)=(f|_1,\ldots,f|_p)$.

\vspace{8pt}

The map $\psi$ can be used to define automorphisms of $\TT$.
Let $a$ be given by
\[
\psi(a) = (1,\overset{p}{\ldots},1) \, \sigma,
\]
where we recall that $\sigma=(1\ 2\ \ldots \ p)$.
Thus $a$ permutes cyclically the $p$ subtrees hanging from the root, and it is of order $p$.
As a consequence, it makes sense to write $a^e$ for an element $e\in\F_p$.
On the other hand, for a fixed tuple $\mathbf{e}=(e_1,\ldots,e_{p-1})\in \F_p^{p-1}$,
we define the automorphism $b\in\St(1)$ recursively via
\[
\psi(b) = (a^{e_1},\ldots,a^{e_{p-1}},b).
\]
Note that $b$ is also of order $p$.
Then the group $G=\langle a,b \rangle$ is called the \emph{GGS-group with defining vector $\mathbf{e}$}.

For a given $f\in\St(1)$, we have
\begin{equation}
\label{f conjugate a}
\psi(f^a) = (f|_p,f|_1,\ldots,f|_{p-1}).
\end{equation}
As a consequence,
\begin{equation}
\label{commutator of f and a}
\psi([f,a]) = (f|_1^{-1}f|_p,f|_2^{-1}f|_1,\ldots,f|_p^{-1}f|_{p-1}),
\end{equation}
a fact that we will freely use multiple times throughout the paper.
From \eqref{f conjugate a} and the definition of $a$ and $b$, it readily follows that, for a GGS-group $G$, we have
\begin{equation}
\label{psi(G) and W(G)}
G\cong \psi(G) \subseteq (G\times \overset{p}{\cdots} \times G) \langle \sigma \rangle = W(G).
\end{equation}
Note that $\psi(G)$ is not equal to $W(G)$ in this case.
By \cite[Lemma 3.2]{Gustavo}, if the defining vector $\mathbf{e}$ is not constant then $G$ is
\emph{regular branch} over $\gamma_3(G)$.
More precisely, we have
\begin{equation}
\label{regular branch gamma3}
\gamma_3(G) \times \overset{p}{\cdots} \times \gamma_3(G)
=
\psi(\gamma_3(\St_G(1)))
\subseteq 
\psi(\gamma_3(G)).
\end{equation}
If furthermore the vector $\mathbf{e}$ is not symmetric then $G$ is regular branch over $G'$,
and we actually have \cite[Lemma 3.4]{Gustavo}
\begin{equation}
\label{regular branch G'}
G' \times \overset{p}{\cdots} \times G'
=
\psi(\St_G(1)')
\subseteq 
\psi(G').
\end{equation}
Be also aware that $|G:\St_G(1)'|=p^{\,p+1}$ by Theorem 2.14 of \cite{Gustavo}.

\vspace{8pt}

For a GGS-group $G$, we define $\St_G(n)=G\cap \St(n)$, the $n$th level stabiliser of $G$.
As mentioned in the introduction, if $\mathbf{e}$ is not constant then $G$ satisfies the congruence subgroup property,
i.e.\ every (normal) subgroup of $G$ of finite index contains $\St_G(n)$ for some $n\in\N$.
Also $G$ is just infinite: every non-trivial normal subgroup of $G$ is of finite index.
It follows from \eqref{regular branch gamma3} that
\begin{equation}
\label{psi stG(n)} 
\psi(\St_G(n))
=
\St_G(n-1) \times \overset{p}{\cdots} \times \St_G(n-1)
\end{equation}
for every $n\ge 3$; see \cite[Lemma 3.3]{Gustavo}.
Now let $G_n=G/\St_G(n)$ be the $n$th congruence quotient of $G$.
Then $G_n$ acts faithfully on the tree $\TT$ truncated at level $n$, and consequently $G_n$ embeds in $\Sym(p^n)$.
Also $\St_{G_n}(n-1)=\St_G(n-1)/\St_G(n)$ is an elementary abelian $p$-group, and in particular $\St_G(n-1)'\le \St_G(n)$.
From \eqref{psi(G) and W(G)} and \eqref{psi stG(n)}, $\psi$ induces an injective map (that we denote by the same letter)
from $G_n$ to $W(G_{n-1})$.
In \Cref{sec:proof theorem A} we will obtain the lower central series of $G_n$ by comparing that of $\psi(G_n)$ with
that of $W(G_{n-1})$.

\vspace{8pt}

Now assume that $G$ is a GGS-group of FG-type, as in \Cref{definition FG-type}.
If its defining vector $\mathbf{e}$ is symmetric then $e_i=e_{p-i}$ for $i=1,\ldots,(p-1)/2$, and so
\[
\delta = \sum_{i=1}^{p-1} \, ie_i = \sum_{i=1}^{(p-1)/2} \, ie_i + \sum_{i=1}^{(p-1)/2} \, (p-i)e_{p-i} = 0
\]
in $\F_p$, which is a contradiction.
Consequently $\mathbf{e}$ is not symmetric and $G$ satisfies \eqref{regular branch G'}.
On the other hand, since $G$ is not periodic, it follows from Theorem 2.4 and Lemma 2.7 of \cite{Gustavo} that
$|G:\St_G(2)|=p^{\,p+1}$.
Since $\St_G(1)'\le \St_G(2)$ and $|G:\St_G(1)'|=p^{\,p+1}$, we conclude that $\St_G(2)=\St_G(1)'$.
Thus \eqref{regular branch G'} implies that
\begin{equation}
\label{psi stG(2)} 
\psi(\St_G(2))
=
G' \times \overset{p}{\cdots} \times G'.
\end{equation}
Another consequence of $G$ not being periodic and $\mathbf{e}$ not being symmetric is that the order of the
congruence quotient $G_n$ is given by
\begin{equation}
\label{order Gn}
|G_n| = p^{\,p^{n-1}+1}
\end{equation}
for every $n\ge 2$; see \cite[Theorem A and Lemma 2.7]{Gustavo}.

After this quick introduction to the basic properties of GGS-groups and groups of FG-type that we are going
to need, our next purpose is to show that, for a non-periodic GGS-group $G$, the $n$th congruence quotient $G_n$
acts uniserially on $\St_{G_n}(n-1)$.
We start by recalling the concept of uniserial action.

\begin{definition}
Let $G$ be a finite $p$-group acting on another finite $p$-group $K$.
We say that $G$ \emph{acts uniserially} on $K$, or that the action of $G$ on $K$ is \emph{uniserial} if
for every non-trivial $G$-invariant subgroup $L$ of $K$ we have $|L:[L,G]|=p$.
\end{definition}

If $G$ acts uniserially on $K$, it is well-known that the only $G$-invariant subgroups of $K$ are precisely those of the
form $[K,G,\overset{i}{\ldots},G]$ with $i\ge 0$; see \cite[Lemma 4.1.3]{LeedhamGreen-McKay}.
A typical example of a uniserial action is given by a finite $p$-group of maximal class acting on one of its maximal subgroups
by conjugation.
For an arbitrary GGS-group $G$, we know that $G_2=\langle a \rangle \ltimes \St_{G_2}(1)$ is a $p$-group of maximal class
\cite[Theorem 2.4]{Gustavo}, and so $G_2$ acts uniserially on $\St_{G_2}(1)$.
Actually, since $\St_{G_2}(1)$ is abelian, we can further say that $\langle g \rangle$ acts uniserially on $\St_{G_2}(1)$
for every $g\in G\smallsetminus \St_{G_2}(1)$.
In the next theorem we generalise this property to other congruence quotients $G_n$, provided that $G$ is non-periodic.
We are going to rely on the following fact for its proof: for a subgroup $H$ of a group $G$ and an element $g\in G$,
if $H$ is $g$-invariant then
\begin{equation}
\label{comm H with <g>}
[H,\langle g \rangle] = [H,g].
\end{equation}
This is an immediate consequence of the identities
\[
[h,g^i]
=
[h,g] [h,g]^g \cdots [h,g]^{g^{i-1}}
=
[h,g] [h^g,g] \cdots [h^{g^{i-1}},g]
\]
for $i\ge 1$, and
\[
[h,g^{-1}]
=
([h,g]^{g^{-1}})^{-1}
=
[h^{g^{-1}},g]^{-1},
\]
which hold for every $h\in H$.

\begin{theorem}
\label{uniserial}
Let $G$ be a non-periodic GGS-group.
Then the following hold for every $n\in\N$:
\begin{enumerate}
\item 
The element $a^{\varepsilon}b$ has order $p^n$ in $G_n$.
\item
If $g\in G_n$ is an element of order $p^{n-1}$ in $G_{n-1}$ then the cyclic subgroup $\langle g \rangle$ acts
uniserially on $\St_{G_n}(n-1)$.
\end{enumerate}
In particular, $G_n$ acts uniserially on $\St_{G_n}(n-1)$, and the only subgroups of $\St_{G_n}(n-1)$ that are
normal in $G_n$ are those of the form
\[
[\St_{G_n}(n-1),G_n,\overset{i}{\ldots},G_n].
\]
Also if $h$ is a generator of $\St_{G_n}(n-1)$ modulo $[\St_{G_n}(n-1),G_n]$ then, for every $i\ge 1$,
$[h,a^{\varepsilon}b,\overset{i}{\ldots},a^{\varepsilon}b]$ is a generator of
$[\St_{G_n}(n-1),G_n,\overset{i}{\ldots},G_n]$ modulo $[\St_{G_n}(n-1),G_n,\overset{i+1}{\ldots},G_n]$.
\end{theorem}

\begin{proof}
(i)
We use induction on $n$, the case $n=1$ being obvious, since $\varepsilon\ne 0$ in $\F_p$.
Now assume that $n>1$ and set $t=a^{\varepsilon}$.
Then
\[
(tb)^p= b^{t^{p-1}} \cdots b^t b.
\]
We are going to obtain the order of $tb$ in $G_n$ by applying $\psi$ to the equality above and by looking at
the orders of the corresponding components in $G_{n-1}$.
The argument we present is essentially the same as the one used to prove that a GGS-group is
non-periodic provided that $\varepsilon\ne 0$; in the literature, this can be found in Theorem 1 of Vovkivsky's paper \cite{Vovkivsky},
where the proof is given in the more general framework of GGS-groups over $p^n$-adic trees.
Since the case of the $p$-adic tree is considerably easier and we have not found it explicitly developed elsewhere,
we have decided to include it for the convenience of the reader.

Since $\varepsilon\ne 0$, we have $\langle t \rangle=\langle a \rangle$ and consequently $\psi((tb)^p)$ is
the product of the tuples $\psi(b),\psi(b^a),\ldots,\psi(b^{a^{p-1}})$ in some unspecified order.
Taking into account \eqref{f conjugate a}, it follows that every entry of $\psi((tb)^p)$ is of the form
\begin{equation}
\label{entry pth power}
a^{f_1}\ldots a^{f_k} b a^{f_{k+1}} \ldots a^{f_{p-1}} = a^{f_1+\cdots+f_{p-1}} b^{a^{f_{k+1}+\cdots+f_{p-1}}}
= \big( a^{f_1+\cdots+f_{p-1}} b \big)^{a^{f_{k+1}+\cdots+f_{p-1}}},
\end{equation}
where $\{f_1,\ldots,f_{p-1}\}=\{e_1,\ldots,e_{p-1}\}$.
Thus \eqref{entry pth power} is equal to
\[
(tb)^{a^{f_{k+1}+\cdots+f_{p-1}}}.
\]
Thus all components of $\psi((tb)^p)$ have the same order in $G_{n-1}$, namely the same as the order of $tb$ in $G_{n-1}$.
By induction, this implies that the order of $(tb)^p$ in $G_n$ is $p^{n-1}$, and the result follows.

(ii)
Let $S=\St_{G_n}(n-1)$ and let $L\ne 1$ be a $\langle g \rangle$-invariant subgroup of $S$.
The action of $g$ on the $p^{n-1}$ vertices of the $(n-1)$st level induces a permutation $\tau\in\Sym(p^{n-1})$,
and the order of $g$ in $G_{n-1}$ coincides with the order of $\tau$.
Since, by hypothesis, $g$ has order $p^{n-1}$ in $G_{n-1}$, it follows that $\tau$ is a cycle of length $p^{n-1}$.
In particular, $\tau$ acts transitively on the vertices of the $(n-1)$st level.

Now consider the map
\[
\begin{matrix}
\kappa & \colon & L & \longrightarrow & L
\\
& & \ell & \longmapsto & [\ell,g],
\end{matrix}
\]
which is a group homomorphism, since $L$ is abelian (recall that $S$ is elementary abelian).
By \eqref{comm H with <g>}, we have $[L,\langle g \rangle]=\{[\ell,g] \mid \ell\in L\}=\kappa(L)$.
Hence the first isomorphism theorem yields
\[
|L:[L,\langle g \rangle]|=|\ker\kappa|=|C_L(g)|.
\]

Now since every element $\ell\in L$ fixes the vertices of the $(n-1)$st level, it permutes the $p$ descendants
of each of these vertices.
Thus $\ell$ defines a tuple of $p^{n-1}$ permutations in $\Sym(p)$, which we denote by $\psi_{n-1}(\ell)$. 
Actually, since $G$ is a subgroup of $\Gamma(p)$, each of these permutations lies in the subgroup
$\langle \sigma \rangle$ of order $p$.
On the other hand, one can easily see that $\psi_{n-1}(\ell^g)$ is obtained from $\psi_{n-1}(\ell)$ by permuting
its components as indicated by $\tau$.
Since $\tau$ acts transitively on the vertices of the $(n-1)$st level, it follows that $\ell\in C_L(g)$
if and only if the tuple $\psi_{n-1}(\ell)$ is constant (with values in $\langle \sigma \rangle$).
Thus $|L:[L,\langle g \rangle]|=|C_L(g)|\le p$.
Finally, since $G_n$ is a finite $p$-group and $L\ne 1$, the commutator $[L,\langle g \rangle]$ is properly
contained in $L$, and we conclude that $|L:[L,\langle g \rangle]|=p$.

The rest of the assertions follow immediately from (i) and (ii).
\end{proof}

\section{The lower central series of some wreath products}
\label{sec:LCS W(G)}

The study of nilpotency of wreath products and, in case they are nilpotent, the determination of their
nilpotency class are two problems that deserved considerable attention in the second half of the 20th century.
Given two non-trivial groups $A$ and $B$, Baumslag \cite{Baumslag} proved that the restricted wreath product $A\wr B$
is nilpotent if and only if both $A$ and $B$ are nilpotent $p$-groups (for the same prime $p$), with $A$
of finite exponent and $B$ finite.
Liebeck \cite[Theorem 5.1]{Liebeck} obtained the value of the nilpotency class of a nilpotent wreath
product, provided that both $A$ and $B$ are abelian.
Liebeck's formula expresses this value in terms of the prime $p$, the exponent of $A$, and the invariant factors
of $B$.
A special case of this is Kaloujnine's earlier result \cite{Kaloujnine} that the wreath product $C_p\wr C_p$ is nilpotent
of class $p$, i.e.\ a $p$-group of maximal class; he also described the terms of its lower central series.
Subsequently, Buckley \cite{Buckley} found the nilpotency class in the case where $A$ is abelian of exponent $p$ and $B$ is
an arbitrary finite $p$-group.
He shows that the class of such a wreath product coincides with the class, as a nilpotent ideal, of the augmentation ideal in the
group ring of $B$ over the integers modulo the exponent of $A$, and also describes the lower central series by using this method
(see the remark right after \Cref{LCS of W(F)} for more details).
Finally, Shield \cite{Shield} determined the class of a nilpotent wreath product in full generality.

Our goal in this section is to give a description, in a way that will be suitable for our later work, of the terms of the lower central series of a wreath product $W(G)=G\wr \langle \sigma \rangle \cong G\wr C_p$ under the following assumption: $G$ is a finite $p$-group and the factors of its lower central series are all elementary abelian
(equivalently, the abelianisation $G/G'$ is elementary abelian; see \cite[Theorem 2.26]{Robinson-Finiteness1}).
A description of the lower central series of these groups is essentially known.
The aforementioned papers of Kaloujnine and Buckley address the special cases where $G$ is of order $p$ or elementary abelian, respectively,
while the general case was solved by Eberhard \cite{Eberhard} in a reply to a MathOverflow question.
Our contribution here is to provide a self-contained formulation that contains explicit bases, word equations, and normal forms describing the tuples
$\mathbf{g}$ in the base group of $W(G)$ that lie in an arbitrary term $\gamma_i(W(G))$ of the lower central series.
This information will play an essential role in the remaining sections of the paper when studying the lower central series
of GGS-groups of FG-type.
As a corollary, we obtain that the nilpotency class of $W(G)$ is $cp$, where $c$ is the class of $G$.
Note that this coincides with the value given by Shield in \cite{Shield} and by Eberhard in \cite{Eberhard}.
All our results in this section have been obtained independently from the references given above.

Before proceeding, let us fix some notation.
Let $B(G)$ be the base group of $W(G)$, and let $\mathbf{g}$ be an arbitrary element of $B(G)$.
For the purpose of this section, it is more convenient to write $\mathbf{g}=(g_0,\ldots,g_{p-1})$, with $g_i\in G$ for all
$i=0,\ldots,p-1$, rather than in the more natural form $(g_1,\ldots,g_p)$.
Note that
\[
\mathbf{g}^{\sigma} = (g_0,g_1,\ldots,g_{p-1})^{\sigma} = (g_{p-1},g_0,\ldots,g_{p-2}).
\]
In this setting, we define the map
\[
\begin{matrix}
\Delta & \colon & B(G) & \longrightarrow & B(G)
\\
& & \mathbf{g} & \longmapsto & [\mathbf{g},\sigma].
\end{matrix}
\]
For every $i\ge 1$, we write $\Delta^i$ for the composition of $\Delta$ with itself $i$ times and, as usual,
we adopt the convention that $\Delta^{0} = \mathrm{id}_{B(G)}$.
Note that $\Delta$ is a group homomorphism if $G$ is abelian.
Clearly, if $N\trianglelefteq G$ and $N\le H\le G$ and we identify $B(H/N)$ with $B(H)/B(N)$, we have
\[
\Delta(B(H/N)) = \Delta(B(H))B(N)/B(N).
\]

We start by studying the case of $W(A)$, where $A$ is an elementary abelian finite $p$-group.
We first assume that $A$ is a finite field $\F$ of characteristic $p$, with the group structure
given by addition.
This has the advantage that the base group $B(\F)$ is also an $\F$-vector space.
We will see that the terms $\gamma_i(W(\F))$ of the lower central series of $W(\F)$ are vector subspaces of $B(\F)$.
As usual in linear algebra, we will describe these subspaces by giving both a basis and
a set of defining linear equations.
Note that $\Delta:B(\F)\rightarrow B(\F)$ is an $\F$-linear map in this case, and that for an arbitrary
$\sigma$-invariant subgroup $H$ of $B(\F)$, we have
\begin{equation}
\label{comm H and W(F)}
[H,W(\F)] = [H,\sigma] = \{ [\boldsymbol{\lambda},\sigma] \mid \boldsymbol{\lambda}\in H \} = \Delta(H).
\end{equation}

Now we consider the factor algebra $\F[X]/(X^p-1)$.
Since $\F$ is of characteristic $p$, this is a local $\F$-algebra whose only maximal ideal $\mathfrak{m}$
is generated by the element $\overline{X-1}$.
All the ideals of $\F[X]/(X^p-1)$ are then the powers $\mathfrak{m}^i$ for $0\le i\le p$.
The quotient $\mathfrak{m}^i/\mathfrak{m}^{i+1}$ has $\F$-dimension $1$ for $0\le i\le p-1$, and is generated
by the image of $(\overline{X-1})^i$.
We have the following isomorphism of $\F$-vector spaces:
\[
\begin{matrix}
\Theta & \colon & B(\F) & \longrightarrow & \F[X]/(X^p-1)
\\[5pt]
& & \boldsymbol{\lambda}=(\lambda_0,\lambda_1,\dots,\lambda_{p-1}) & \longmapsto
& \overline{\lambda_0+\lambda_1X+\cdots+\lambda_{p-1}X^{p-1}}.
\end{matrix}
\]
Under this isomorphism, the action of $\sigma$ on $B(\F)$ corresponds to multiplication by $\overline X$ in
$\F[X]/(X^p-1)$, and taking commutators with $\sigma$ corresponds to multiplication by $\overline{X-1}$.
As a consequence, $\Theta$ extends to an isomorphism of $W(\F)$ with the semidirect product
$\langle \overline X \rangle \ltimes \F[X]/(X^p-1)$, where the corresponding action is given by multiplication
in the algebra
(note that $\overline X$ is a unit of multiplicative order $p$ in $\F[X]/(X^p-1)$).

By \eqref{comm H and W(F)}, we have
\[
\gamma_i(W(\F))
=
[B(\F),\sigma,\overset{i-1}{\ldots},\sigma]
=
\Delta^{i-1}(B(\F))
\]
for every $i\ge 2$.
Consequently $\Theta(\gamma_i(W(\F)))=\mathfrak{m}^{i-1}$.
This proves, in particular, Liebeck's result that $W(\F)$ has nilpotency class $p$.
Also the subgroups $\gamma_i(W(\F))$ are $\F$-subspaces of $B(\F)$.
Note that
\[
|\gamma_i(W(\F)):\gamma_{i+1}(W(\F))| = |\mathfrak{m}^{i-1}:\mathfrak{m}^i| = |\F|
\]
for $2\le i\le p$.

On the other hand, an $\F$-subspace $H$ of $B(\F)$ is normal in $W(\F)$, i.e.\ invariant under $\sigma$, if and only
if $\Theta(H)$ is an $\F$-subspace of $\F[X]/(X^p-1)$ invariant under multiplication by $\overline X$ or, equivalently,
an ideal of $\F[X]/(X^p-1)$.
Hence $\Theta(H)$ coincides with one of the powers $\mathfrak{m}^i$, and then
either $H=B(\F)$ or $H=\gamma_i(W(\F))$ with $2\le i\le p+1$.

We collect all the information obtained so far in the following theorem.
In order to make the statement more uniform, it is convenient to introduce a slight modification of notation,
replacing $W(\F)$ with $B(\F)$ in the lower central series.
In general, for any group of the form $W(G)$, we set
\[
\gamma_1^*(W(G))=B(G)
\]
and
\[
\gamma_i^*(W(G))=[\gamma_{i-1}^*(W(G)),W(G)],
\quad
\text{for every $i\ge 2$.}
\]
Since $W(G)/B(G)$ is cyclic, we have $W(G)'=[B(G),W(G)]$, and consequently
$\gamma_i(W(G))=\gamma_i^*(W(G))$ for $i\ge 2$.

\begin{theorem}
\label{LCS of W(F)}
Let $\F$ be a finite field of characteristic $p$.
Then the following hold:
\begin{enumerate}
\item 
For every $i\ge 1$ we have $\gamma_i^*(W(\F))=\Delta^{i-1}(B(\F))$, and $\Theta(\gamma_i^*(W(\F)))=\mathfrak{m}^{i-1}$, where
$\mathfrak{m}$ is the ideal of $\F[X]/(X^p-1)$ generated by $\overline{X-1}$.
\item
The only $\sigma$-invariant $\F$-subspaces of $B(\F)$ are those in the chain
\[
1=\gamma_{p+1}^*(W(\F)) \le \cdots \le \gamma_i^*(W(\F)) \le \cdots \le \gamma_1^*(W(\F)) = B(\F).
\]
\item
We have
\[
|\gamma_i^*(W(\F)):\gamma_{i+1}^*(W(\F))| = |\F|
\]
and
\[
|\gamma_i^*(W(\F))| = |\F|^{p-i+1}
\]
for $1\le i\le p+1$.
\end{enumerate}
In particular, the nilpotency class of $W(\F)$ is $p$.
\end{theorem}

We observe that part (i) of \Cref{LCS of W(F)} is a special case of Buckley's result in \cite{Buckley}.
He deals more generally with a wreath product of the form $W=R\wr H$, where $R$ is the additive group of a commutative ring and $H$ is finite.
The base group can then be identified with the group ring $R[H]$, and he shows that $\gamma_{n+1}(W) \cap R[H] = \omega^n$, where $\omega$ is the
augmentation ideal of $R[H]$.
Note that the ideal $\mathfrak{m}$ of $\F[X]/(X^p-1)$ generated by $\overline{X-1}$ that we use in our proof corresponds to the augmentation ideal
of $\F[C_p] \cong \F[X]/(X^p-1)$.

\vspace{8pt}

Now it is easy to give bases for the subspaces $\gamma_i^*(W(\F))$.

\begin{theorem}
\label{gens for gammaiW(F)}
Let $\F$ be a finite field of characteristic $p$.
For $1\le i\le p$, set
\[
\boldsymbol{\lambda}_i
=
\Delta^{i-1}(1,0,\overset{p-1}{\ldots},0) \in B(\F).
\]
Then the following hold:
\begin{enumerate}
\item
We have
\[
\boldsymbol{\lambda}_i
=
(\lambda_{i,0},\ldots,\lambda_{i,i-1},0,\overset{p-i}{\ldots},0),
\]
in $B(\F)$, where
\[
\lambda_{i,r} = (-1)^{i-r-1} \dbinom{i-1}{r}
\]
for $0\le r\le i-1$.
\item
The image of $\boldsymbol{\lambda}_i$ generates the quotient $\gamma_i^*(W(\F))/\gamma_{i+1}^*(W(\F))$
as an $\F$-vector space.
\item
The set $\{\boldsymbol{\lambda}_i,\ldots,\boldsymbol{\lambda}_p\}$ is an $\F$-basis of $\gamma_i^*(W(\F))$.
\end{enumerate}
\end{theorem}

\begin{proof}
By using \Cref{LCS of W(F)}, the $\F$-isomorphism $\Theta$ induces an $\F$-isomorphism between
$\gamma_i^*(W(\F))/\gamma_{i+1}^*(W(\F))$ and $\mathfrak{m}^{i-1}/\mathfrak{m}^i$.
The latter has $\F$-dimension $1$ and is generated by the image of
\[
(\overline{X-1})^{i-1}
=
\sum_{r=0}^{i-1} \, (-1)^{i-r-1} \dbinom{i-1}{r} \overline X^r.
\]
Since $\Theta$ sends $\boldsymbol{\lambda}_i$ to $(\overline{X-1})^{i-1}$, the result follows.
\end{proof}

Since $\F$ is of characteristic $p$ and
\begin{equation}
\label{lambdapr mod p}
\lambda_{p,r} = (-1)^{p-r-1} \frac{(p-1)(p-2)\ldots(p-r)}{r!}
\equiv (-1)^{p-r-1} (-1)^r = 1 \mod p
\end{equation}
for $0\le r\le p-1$, we immediately obtain the following description of the last non-trivial term
of the lower central series of $W(\F)$.

\begin{corollary}
\label{gammapW(F)}
Let $\F$ be a finite field of characteristic $p$.
Then
\[
\gamma_p(W(\F)) = \{ (\lambda,\overset{p}{\ldots},\lambda) \mid \lambda\in\F \}
\]
is the set of constant tuples in $B(\F)$.
\end{corollary}

Thus $\gamma_p(W(\F))$ can be described as the subspace of $B(\F)$ given by the equations $X_0=X_1=\cdots=X_{p-1}$.
Our goal now is to describe all terms of the lower central series of $W(\F)$ via linear homogeneous equations.
As it turns out, these equations will be induced by certain polynomials in $\F[X]$.
To this purpose, we introduce the following notation: to every polynomial $f(X)=f_0+f_1X+\cdots+f_{p-1}X^{p-1}\in \F[X]$ of degree
less than $p$, we associate the homogeneous linear polynomial $f_0X_0+f_1X_1+\cdots+f_{p-1}X_{p-1}$, which we denote by
$\LL(f(X))$.
Here $X_0,X_1,\ldots,X_{p-1}$ are new variables.

\begin{theorem}
\label{equations for gammaiW(F)}
Let $\F$ be a finite field of characteristic $p$.
Then for every $2\le i\le p$, the subgroup $\gamma_i(W(\F))$ consists of all solutions $\boldsymbol\lambda\in B(\F)$
of either of the following systems of homogeneous linear equations:
\begin{enumerate}
\item
$\LL((X-1)^{\ell})=0$, for $p-i+1\le \ell \le p-1$.
\item
$\LL(f^{(\ell)}(X))=0$, for $0\le \ell\le i-2$, where $f(X)=1+X+\cdots+X^{p-1}$.
\end{enumerate}
\end{theorem}

\begin{proof}
The standard bilinear product
\[
(\lambda_0,\ldots,\lambda_{p-1})\cdot (\mu_0,\ldots,\mu_{p-1})=\lambda_0\mu_0+\cdots+\lambda_{p-1}\mu_{p-1}
\]
on $B(\F)$ is non-degenerate.
Consequently, for every $\F$-subspace $U$ of $B(\F)$ we have $(U^{\perp})^{\perp}=U$, and so a basis of
$U^{\perp}$ provides a system of homogeneous linear equations defining $U$.
We will apply this to $U=\gamma_i(W(\F))$.

By non-degeneracy and (iii) of \Cref{LCS of W(F)}, we have
\[
\dim_{\F} \gamma_i(W(\F))^{\perp} = p - \dim_{\F} \gamma_i(W(\F)) = i-1.
\]
Also $\gamma_i(W(\F))^{\perp}$ is $\sigma$-invariant, since $\gamma_i(W(\F))$ is.
By (ii) and (iii) of \Cref{LCS of W(F)}, it follows that $\gamma_i(W(\F))^{\perp}=\gamma_{p-i+2}(W(\F))$.
Now (i) of the same theorem yields $\Theta(\gamma_i(W(\F))^{\perp})=\mathfrak{m}^{p-i+1}$.
Since
\[
\{ (\overline{X-1})^{\ell} \mid p-i+1\le \ell\le p-1 \}
\]
is an $\F$-basis of $\mathfrak{m}^{p-i+1}$,
it follows that $\gamma_i(W(\F))$ can be defined by the equations in (i).

As for the equations in (ii), observe that
\begin{equation}
\label{X-1 to p-1}
f(X) = \frac{X^p-1}{X-1} = \frac{(X-1)^p}{X-1} = (X-1)^{p-1}
\end{equation}
in $\F[X]$.
Consequently the derivatives $f^{(\ell)}(X)=(p-1)\ldots (p-\ell)(X-1)^{p-\ell-1}$ for $0\le \ell\le i-2$
also provide an $\F$-basis of $\mathfrak{m}^{p-i+1}$.
\end{proof}

As a consequence, we can easily determine whether a tuple $\boldsymbol\lambda$ belongs to $\gamma_2(W(\F))$
or $\gamma_3(W(\F))$.
These cases will be particularly useful in the following sections.

\begin{corollary}
\label{equations for gamma2W(F) and gamma3W(F)}
Let $\F$ be a finite field of characteristic $p$.
Then, for a given tuple $\boldsymbol{\lambda}=(\lambda_0,\ldots,\lambda_{p-1})\in B(\F)$, we have
\begin{equation}
\label{equation for gamma2W(F)}
\boldsymbol{\lambda} \in \gamma_2(W(\F))
\ \
\Longleftrightarrow
\ \
\lambda_0+\cdots+\lambda_{p-1}=0,
\end{equation}
and
\begin{equation}
\label{equation for gamma3W(F)}
\boldsymbol{\lambda} \in \gamma_3(W(\F))
\ \
\Longleftrightarrow
\ \
\lambda_0+\cdots+\lambda_{p-1}=0,
\ \
\lambda_0+2\lambda_1+\cdots+(p-1)\lambda_{p-2}=0.
\end{equation}
\end{corollary}

For later use, it is convenient to translate the information obtained in the previous theorems from the
special case of a finite field to the general case of an arbitrary elementary abelian
finite $p$-group $A$, written multiplicatively.
Assume $A$ is of order $p^r$ and let $\rho:(A,\cdot)\rightarrow (\F,+)$ be a group isomorphism, where $\F$
is a field with $p^r$ elements.
Then $\rho$ obviously induces an isomorphism $\tilde\rho$ between $W(A)$ and $W(\F)$, which is defined
componentwise on $B(A)$ by using $\rho$ and is the identity on $\langle \sigma \rangle$.
As a consequence,
\begin{equation}
\label{rho and delta commute}
\tilde\rho(\Delta(a_0,\ldots,a_{p-1})) = \Delta(\tilde\rho(a_0,\ldots,a_{p-1}))
\end{equation}
for all $a_0,\ldots,a_{p-1}\in A$.
Since $\tilde\rho$ is an isomorphism, the results about the nilpotency class, and the orders and indices of
terms of the lower central series in \Cref{LCS of W(F)} remain exactly the same, replacing $\F$ with $A$.
Also
\begin{equation}
\label{gammaiW(A) and Delta}
\gamma_i^*(W(A)) = \Delta^{i-1} (B(A)),
\quad
\text{for every $i\ge 1$,}
\end{equation}
and \Cref{gammapW(F)} translates into
\begin{equation}
\label{gammapW(A)}
\gamma_p(W(A)) = \{ (a,\overset{p}{\ldots},a) \mid a\in A \}.
\end{equation}

Let us see how we can rewrite the result regarding generators of terms of the lower central series.
Consider a vector $\boldsymbol{\lambda}=(\lambda_0,\ldots,\lambda_{p-1}) \in B(\F_p)$ (note that
$\F_p$ is the prime subfield of $\F$).
Now since $\rho$ is bijective, an arbitrary element $\mu\in \F$ can be written in the form $\rho(a)$ for a unique
$a\in A$, and then the scalar product $\mu \boldsymbol{\lambda}$ is equal to
\begin{equation}
\label{rho(a)lambda}
\begin{split}
\rho(a) \boldsymbol\lambda
&=
( \lambda_0 \rho(a), \ldots, \lambda_{p-1} \rho(a) )
\\
&=
(\rho(a^{\lambda_0}), \ldots, \rho(a^{\lambda_{p-1}}) )
\\
&=
\tilde\rho \big( a^{\lambda_0}, \ldots, a^{\lambda_{p-1}} \big),
\end{split}
\end{equation}
and as a consequence
\begin{equation}
\label{tilde rho inverse}
\tilde{\rho}^{-1}(\F \boldsymbol{\lambda}) = \{ ( a^{\lambda_0}, \ldots, a^{\lambda_{p-1}} ) \mid a\in A \}.
\end{equation}
From this remark and \Cref{gens for gammaiW(F)} we can easily obtain the following theorem about the lower
central series of $W(A)$.
Before proceeding, we need to introduce some general notation: if $G$ is a group and $g\in G$, we set
\[
\boldsymbol{\lambda}_i(g) = \Delta^{i-1}(g,1,\overset{p-1}{\ldots},1) \in B(G).
\]
Then one can easily adapt the proof of \Cref{gens for gammaiW(F)} to show the following lemma.

\begin{lemma}
\label{lambdai(g)}
For every group $G$ and every $g\in G$, we have
\[
\boldsymbol{\lambda}_i(g)
=
(g^{\lambda_{i,0}},\ldots,g^{\lambda_{i,i-1}},1,\overset{p-i}{\ldots},1),
\]
where the elements $\lambda_{i,r}$ are as in \Cref{gens for gammaiW(F)}.
\end{lemma}

\begin{theorem}
\label{gens for gammaiW(A)}
Let $A$ be an elementary abelian finite $p$-group, and let $1\le i\le p$.
Then the following hold:
\begin{enumerate}
\item
The images of the tuples $\boldsymbol\lambda_i(a)$, with $a\in A$, yield all
elements of the quotient $\gamma_i^*(W(A))/\gamma_{i+1}^*(W(A))$, without repetitions.
Furthermore, the map
\[
\begin{matrix}
A & \longrightarrow & \gamma_i^*(W(A))/\gamma_{i+1}^*(W(A))
\\[5pt]
a & \longmapsto & \boldsymbol\lambda_i(a)
\end{matrix}
\]
is a group isomorphism.
\item
The elements of $\gamma_i^*(W(A))$ can be written in the form
\[
\boldsymbol\lambda_i(a_i) \cdots \boldsymbol\lambda_p(a_p),
\]
with $a_i,\ldots,a_p\in A$, without repetitions.
\item
For $1\le i\le p-1$, the map
\[
\begin{matrix}
\gamma_i^*(W(A))/\gamma_{i+1}^*(W(A)) & \longrightarrow & \gamma_{i+1}^*(W(A))/\gamma_{i+2}^*(W(A))
\\[5pt]
\mathbf{a}\gamma_{i+1}^*(W(A)) & \longmapsto & \Delta(\mathbf{a})\gamma_{i+2}^*(W(A))
\end{matrix}
\]
is a group isomorphism.
\end{enumerate}
\end{theorem}

\begin{proof}
Items (i) and (ii) are an immediate consequence of \eqref{tilde rho inverse} and the
corresponding results in \Cref{gens for gammaiW(F)}.
Finally, (iii) follows from the isomorphism in (i), applied to both $i$ and $i+1$.
\end{proof}

Now we rewrite \Cref{equations for gammaiW(F)} for an arbitrary multiplicative elementary abelian
finite $p$-group $A$.
To start with, for a given polynomial
\[
f(X)=f_0+f_1X+\cdots+f_{p-1}X^{p-1}\in \F_p[X],
\]
we associate the word $X_0^{f_0} X_1^{f_1} \cdots X_{p-1}^{f_{p-1}}$, which we denote by $\MM(f(X))$.
Since $A$ is elementary abelian, this word induces a well-defined homomorphism from $B(A)$ to $A$ by substitution.
Also we can speak of the word equation $\MM(f(X))=1$ in $A$ and of its set of solutions in $B(A)$.
Now for a tuple $\mathbf{a}=(a_0,\ldots,a_{p-1})\in B(A)$, we have the following equivalent conditions:
\begin{align*}
\tilde{\rho}(\mathbf{a})
\text{ is a solution of $\LL(f(X))=0$}
\
&\Longleftrightarrow
\
\rho(a_0)f_0 + \rho(a_1) f_1 + \cdots + \rho(a_{p-1}) f_{p-1} = 0
\\
&\Longleftrightarrow
\
\rho(a_0^{f_0} a_1^{f_1} \cdots a_{p-1}^{f_{p-1}}) = 0
\\
&\Longleftrightarrow
\
\mathbf{a}
\text{ is a solution of $\MM(f(X))=1$.}
\end{align*}
Since all the polynomials appearing in \Cref{equations for gammaiW(F)} have coefficients in $\F_p$,
we immediately get the following result.

\begin{theorem}
\label{equations for gammaiW(A)}
Let $A$ be an elementary abelian finite $p$-group.
Then for every $2\le i\le p$, the subgroup $\gamma_i(W(A))$ consists of all solutions in $B(A)$ of either of the following systems of word equations:
\begin{enumerate}
\item
$\MM((X-1)^{\ell})=1$, for $p-i+1\le \ell \le p-1$.
\item
$\MM(f^{(\ell)}(X))=1$, for $0\le \ell \le i-2$, where $f(X)=1+X+\cdots+X^{p-1}$.
\end{enumerate}
\end{theorem}

Finally, we record the result corresponding to \Cref{equations for gamma2W(F) and gamma3W(F)}.

\begin{corollary}
\label{equations for gamma2W(A) and gamma3W(A)}
Let $A$ be an elementary abelian finite $p$-group.
Then, for a given $\mathbf{a}=(a_0,\ldots,a_{p-1})\in B(A)$, we have
\begin{equation}
\label{equation for gamma2W(A)}
\mathbf{a} \in \gamma_2(W(A))
\ \
\Longleftrightarrow
\ \
a_0 \cdots a_{p-1}=1,
\end{equation}
and
\begin{equation}
\label{equation for gamma3W(A)}
\mathbf{a} \in \gamma_3(W(A))
\ \
\Longleftrightarrow
\ \
a_0 \cdots a_{p-1}=1,
\ \
a_0 a_1^2 \cdots a_{p-2}^{p-1}=1.
\end{equation}
\end{corollary}

After having completed a thorough analysis of the lower central series of $W(A)$, where $A$ is an elementary
abelian finite $p$-group, now we proceed to the general case of $W(G)$, where $G$ is a finite $p$-group all of
whose lower central factors are elementary abelian.
Before describing the lower central series of $W(G)$, we need a lemma.

\begin{lemma}
\label{comm with W(G)}
Let $G$ be a finite $p$-group and let $T$ be a $\sigma$-invariant subgroup of $B(\gamma_j(G))$ for some $j\ge 1$.
Then
\begin{equation}
\label{comm T with W(G)}
[T,W(G)] \equiv [T,\langle \sigma \rangle] \equiv \Delta(T) \mod B(\gamma_{j+1}(G)).
\end{equation}
In particular, if $J\le \gamma_j(G)$ then
\begin{equation}
\label{comm Deltak-1J with W(G)}
[\Delta^{k-1}(B(J)),W(G)] \equiv [\Delta^{k-1}(B(J)),\langle \sigma \rangle]
\equiv \Delta^k(B(J)) \mod B(\gamma_{j+1}(G))
\end{equation}
for every $k\ge 1$.
\end{lemma}

\begin{proof}
We may assume that $B(\gamma_{j+1}(G))=1$.
Since $T\le B(\gamma_j(G))$, it follows that $B(G)$ centralises $T$ and so
$[T,W(G)]=[T,\langle \sigma \rangle]$.
On the other hand, since $T$ is $\sigma$-invariant, we have $[T,\langle \sigma \rangle]=[T,\sigma]$
by \eqref{comm H with <g>}.
Also, since $T$ is abelian, the restriction of $\Delta$ to $T$ is a group homomorphism and
consequently $\Delta(T)=\{[\mathbf{g},\sigma] \mid \mathbf{g}\in T\}$ is a subgroup of $W(G)$.
Thus $[T,\sigma]=\Delta(T)$, which proves \eqref{comm T with W(G)}.

Now \eqref{comm Deltak-1J with W(G)} readily follows by induction on $k$, by observing that
the subgroup $B(J)$ is $\sigma$-invariant.
\end{proof}

Now we can give a formula for the terms of the lower central series of $W(G)$.
This is essentially contained in Eberhard's answer in \cite{Eberhard}.

\begin{theorem}
\label{LCS of W(G)}
Let $G$ be a finite $p$-group of nilpotency class $c$ whose abelianisation is elementary abelian, and let $i\ge 1$.
If we write $i=(j-1)p+k$ with $j\ge 1$ and $1\le k\le p$, then we have
\begin{equation}
\label{gamma_i W(G)}
\gamma_i^* \big( W(G) \big)
=
\Delta^{k-1} \big( B(\gamma_j(G)) \big) \cdot B(\gamma_{j+1}(G)).
\end{equation}
In particular,
\begin{equation}
\label{gamma_jp+1 W(G)}
\gamma_{jp+1}^* \big( W(G) \big)
=
B(\gamma_{j+1}(G))
\end{equation}
for every $j\ge 0$, and the nilpotency class of $W(G)$ is $cp$.
\end{theorem}

\begin{proof}
We use induction on $i$.
The basis of the induction is obvious, so let $i\ge 2$.
We first assume that $2\le k\le p$.
(Note that the basis of the induction corresponds to $k=1$, so in the induction step we can deal
with $k=1$ after having proved the result for $2\le k\le p$.)
By the induction hypothesis and \eqref{comm Deltak-1J with W(G)}, we have
\begin{equation}
\label{i general}
\begin{split}
\gamma_i^*(W(G)) \cdot B(\gamma_{j+1}(G))
&=
[ \gamma_{i-1}^*(W(G)), W(G) ] \cdot B(\gamma_{j+1}(G))
\\
&=
[\Delta^{k-2} \big( B(\gamma_j(G)) \big) \cdot B(\gamma_{j+1}(G)),W(G)] \cdot B(\gamma_{j+1}(G))
\\
&=
\Delta^{k-1} (B(\gamma_j(G))) \cdot B(\gamma_{j+1}(G)).
\end{split}
\end{equation}
This proves in particular, for $k=p$, that
\begin{equation}
\label{i=jp}
\gamma_{jp}^*(W(G)) \cdot B(\gamma_{j+1}(G)) = \Delta^{p-1} (B(\gamma_j(G))) \cdot B(\gamma_{j+1}(G)).
\end{equation}
Now the factor group $\gamma_j(G)/\gamma_{j+1}(G)$ is elementary abelian, and then
\[
\Delta^{p-1} (B(\gamma_j(G)/\gamma_{j+1}(G)))
=
\gamma_p (W(\gamma_j(G)/\gamma_{j+1}(G)))
\]
coincides with the set of constant tuples with entries in $\gamma_j(G)/\gamma_{j+1}(G)$, by \eqref{gammapW(A)}.
Thus we get
\begin{equation}
\label{k=p}
\Delta^{p-1} (B(\gamma_j(G))) \cdot B(\gamma_{j+1}(G)) = C_j \cdot B(\gamma_{j+1}(G)),
\end{equation}
where
\[
C_j = \{ (g,\overset{p}{\ldots},g) \mid g\in\gamma_j(G) \}.
\]
Observe that
\begin{equation}
\label{comm Cj and W(G)}
[C_j,W(G)] = [C_j,B(G)] = B(\gamma_{j+1}(G)),
\end{equation}
where the first equality holds because constant tuples commute with $\sigma$, and the second
because
\[
[(g,\overset{p}{\ldots},g),(h,1,\overset{p-1}{\ldots},1)]
=
([g,h],1,\overset{p-1}{\ldots},1)
\]
for every $g\in\gamma_j(G)$ and every $h\in G$.
Hence taking the commutator with $W(G)$ in \eqref{i=jp} and using \eqref{k=p} and \eqref{comm Cj and W(G)},
we get
\begin{equation}
\label{i=jp+1 partial}
\gamma_{jp+1}^*(W(G)) \cdot [B(\gamma_{j+1}(G)),W(G)] = B(\gamma_{j+1}(G)).
\end{equation}
In other words, the normal subgroup $B(\gamma_{j+1}(G))$ of $W(G)$ coincides with the commutator
$[B(\gamma_{j+1}(G)),W(G)]$ modulo $\gamma_{jp+1}^*(W(G))$.
Since $W(G)/\gamma_{jp+1}^*(W(G))$ is a finite $p$-group, this can only happen if $B(\gamma_{j+1}(G))$
is trivial in this quotient group.
Hence \eqref{i=jp+1 partial} implies that
\begin{equation}
\label{i=jp+1}
\gamma_{jp+1}^*(W(G))
=
B(\gamma_{j+1}(G)).
\end{equation}
By taking this value to \eqref{i general}, we obtain \eqref{gamma_i W(G)} for $k=2,\ldots,p$.

Finally, assume that $k=1$.
Then since $i\ge 2$, we have $j\ge 2$, and the result follows from \eqref{i=jp+1}, applied with $j-1$
in place of $j$.
\end{proof}

\begin{remark}
\label{k=p+1}
Note that \eqref{i=jp+1} can also be written as
\[
\gamma_{jp+1}^*(W(G)) = \Delta^p (B(\gamma_j(G))) \cdot B(\gamma_{j+1}(G)).
\]
In other words, \eqref{gamma_i W(G)} is also valid if $i=(j-1)p+k$ and $k=p+1$.
\end{remark}

Equality \eqref{gamma_i W(G)} above says that for $i=(j-1)p+k$ with $1\le k\le p$, the subgroup $\gamma_i^*(W(G))$
coincides with the lift to $W(G)$ of the subgroup $\gamma_k^*(W(\gamma_j(G)/\gamma_{j+1}(G)))$, taking into
account that $W(\gamma_j(G)/\gamma_{j+1}(G))$ is naturally isomorphic to the quotient
$W(\gamma_j(G))/B(\gamma_{j+1}(G))$.
By \Cref{k=p+1}, this also holds for $k=p+1$.
Then the following corollary is an immediate consequence of the results that we have proved for $W(A)$, where
$A$ is elementary abelian.

\begin{corollary}
\label{LCS of W(G) detail}
Let $G$ be a finite $p$-group of nilpotency class $c$ whose abelianisation is elementary abelian, and let
$1\le i\le cp$.
If we write $i=(j-1)p+k$ with $1\le j\le c$ and $1\le k\le p$, then the following hold:
\vspace{5pt}
\begin{enumerate}
\setlength\itemsep{10pt}
\item
The subgroup $\gamma_i^*(W(G))$ consists of all tuples in $B(\gamma_j(G))$ that are solutions of either of the
following sets of equations:
\vspace{3pt}
\begin{enumerate}
\setlength\itemsep{5pt}
\item[(a)]
$\MM((X-1)^{\ell})\equiv 1 \mod \gamma_{j+1}(G)$, for $p-k+1\le \ell \le p-1$.
\item[(b)]
$\MM(f^{(\ell)}(X))\equiv 1 \mod \gamma_{j+1}(G)$, for $0\le \ell \le k-2$, where $f(X)=1+X+\cdots+X^{p-1}$.
\end{enumerate}
\vspace{3pt}
In particular, if $\mathbf{g}=(g_0,\ldots,g_{p-1})\in B(\gamma_j(G))$ then
\[
\mathbf{g} \in \gamma_{(j-1)p+2}(W(G))
\
\Longleftrightarrow
\
g_0 \cdots g_{p-1} \in \gamma_{j+1}(G),
\]
and
\[
\mathbf{g} \in \gamma_{(j-1)p+3}(W(G))
\
\Longleftrightarrow
\
g_0 \cdots g_{p-1} \in \gamma_{j+1}(G), \ g_0 g_1^2 \cdots g_{p-2}^{p-1} \in \gamma_{j+1}(G).
\]
\item
$|\gamma_i^*(W(G)):\gamma_{i+1}^*(W(G))|=|\gamma_j(G):\gamma_{j+1}(G)|$.
\item
The map
\[
\begin{matrix}
\gamma_j(G)/\gamma_{j+1}(G) & \longrightarrow & \gamma_i^*(W(G))/\gamma_{i+1}^*(W(G))
\\[5pt]
g\gamma_{j+1}(G) & \longmapsto
& \boldsymbol{\lambda}_k(g) \, \gamma_{i+1}^*(W(G))
\end{matrix}
\]
is a group isomorphism.
\item
For $1\le k\le p-1$, the map
\[
\begin{matrix}
\gamma_i^*(W(G))/\gamma_{i+1}^*(W(G)) & \longrightarrow & \gamma_{i+1}^*(W(G))/\gamma_{i+2}^*(W(G))
\\[5pt]
\mathbf{g}\gamma_{i+1}^*(W(G)) & \longmapsto
& \Delta(\mathbf{g})\gamma_{i+2}^*(W(G))
\end{matrix}
\]
is a group isomorphism.
\end{enumerate}
\end{corollary}

As it turns out, in order to determine the lower central series of GGS-groups of FG-type, we are not going
to need the characterisation of the tuples lying in $\gamma_i^*(W(G))$ given in (i) above
for all $i=(j-1)p+k$ with $1\le k\le p$, but rather in the two special cases $k=2$ and $k=3$ that appear
explicitly developed therein.
However, our experience with computations carried out with GAP hints that this result will be needed for higher
values of $k$ if one wants to study the lower central series of arbitrary GGS-groups.
For this reason, and for completeness, we have decided to state \Cref{LCS of W(G) detail} in all generality.

\vspace{8pt}

By (iii) of the previous corollary, for a given $i=(j-1)p+k$ with $1\le k\le p$, every element
$\mathbf{g}\in \gamma_i^*(W(G))$ is congruent modulo $\gamma_{i+1}^*(W(G))$ to a tuple of the form $\boldsymbol{\lambda}_k(g) $ with $g\in\gamma_j(G)$.
In the following result, we make this explicit in the cases $k=1$ and $k=2$.

\begin{corollary}
\label{reduction to first components}
Let $G$ be a finite $p$-group of nilpotency class $c$ whose abelianisation is elementary abelian.
Assume that $\mathbf{g}=(g_0,\ldots,g_{p-1})\in \gamma_i^*(W(G))$, where $i=(j-1)p+k$ with
$1\le j\le c$ and $1\le k\le p$.
Then the following hold:
\begin{enumerate}
\item
If $k=1$, i.e.\ if $\mathbf{g}\in B(\gamma_j(G))$, then
\[
\mathbf{g} \equiv (g,1,\overset{p-1}{\ldots},1)
\mod
\gamma_{i+1}^*(W(G)),
\]
where $g=g_0\cdots g_{p-1}$.
\item
If $k=2$ then
\[
\mathbf{g} \equiv (g^{-1},g,1,\overset{p-2}{\ldots},1)
\mod
\gamma_{i+1}^*(W(G)),
\]
where $g=g_0g_1^2\cdots g_{p-2}^{p-1}$.
\end{enumerate}
\end{corollary}

\begin{proof}
First of all, observe that all components $g_i$ of $\mathbf{g}$ lie in $\gamma_j(G)$.
If $k=1$ then we have
\[
\mathbf{g}^{-1} \cdot (g_0\cdots g_{p-1},1,\overset{p-1}{\ldots},1)
=
(g_1\cdots g_{p-1},g_1^{-1},\ldots,g_{p-1}^{-1}),
\]
and by (i) of \Cref{LCS of W(G) detail}, this last tuple belongs to $\gamma_{(j-1)p+2}(W(G))$.
This proves (i).
The proof of (ii) goes exactly in the same way.
\end{proof}

\section{\texorpdfstring{Lower central series of $G_n$ for $n\le 3$}
{Lower central series of G\_n for n <= 3}}\label{sec:n at most 3}

In this section, we will determine the lower central series of the congruence quotient $G_n=G/\St_G(n)$
for $n\le 3$, where $G$ is a GGS-group of FG-type.
For $n\le 2$ the result is straightforward, and valid for all non-periodic GGS-groups.
Thus the bulk of the section is devoted to studying the lower central series of $G_3$.
Actually, we approach this problem in the more general cases in which $G$ is only assumed to be non-periodic, or non-periodic and regular branch over $G'$.
We obtain some partial results that are valid in this setting, but \Cref{position commutators} below shows
that the position of some relevant commutators is clearer in the case that $G$ is of FG-type (which under the
assumption of non-periodicity is equivalent to the condition that $\delta\ne 0$).
As a consequence, we determine the lower central series of $G_3$ only for groups of FG-type.
Finally, we prove that the description of Theorem A fails for every non-periodic GGS-group that is regular branch
over $G'$ and not of FG-type, and we give computational examples showing that in that case the behaviour of the
lower central series is far from uniform.

The description of the lower central series is trivial for $G_1$, which is cyclic of order $p$.
In the case of $G_2$, by Theorem 2.4 of \cite{Gustavo} we know that, for an arbitrary GGS-group $G$,
the quotient $G_2$ is a $p$-group of maximal class of order $p^{t+1}$, where $t$ is the rank of the
circulant matrix whose first row is $(e_1,\dots,e_{p-1},0)$.
By Lemma 2.7 of \cite{Gustavo}, if $G$ is non-periodic then $t=p$ and $G_2$ has nilpotency class $p$
(actually $G_2$ is isomorphic to $C_p\wr C_p$).
As already mentioned at the beginning of \Cref{sec:uniseriality}, if $g\in G_2\smallsetminus \St_{G_2}(1)$
then $\langle g \rangle$ acts uniserially on $\St_{G_2}(1)$, and consequently the quotient
$\gamma_i(G_2)/\gamma_{i+1}(G_2)$ can be generated by the image of $[b,g,\overset{i-1}{\ldots},g]$
for every $i=2,\ldots,p$.
We summarise all this information in the following result.

\begin{theorem}
\label{LCS G2}
Let $G$ be a non-periodic GGS-group.
Then the following hold:
\begin{enumerate}
\item
We have $|G_2:G_2'|=p^2$ and $|\gamma_i(G_2):\gamma_{i+1}(G_2)|=p$ for every $i=2,\ldots,p$.
Also $\gamma_{p+1}(G_2)=1$ and $G_2$ has nilpotency class $p$.
\item 
For every $i\ge 2$ and every $g\in G\smallsetminus \St_G(1)$, we have
\[
\gamma_i(G_2) = \langle [b,g,\overset{i-1}{\ldots},g] \rangle \gamma_{i+1}(G_2).
\]
In particular,
$\gamma_i(G_2) = \langle [b,a^{\varepsilon}b,\overset{i-1}{\ldots},a^{\varepsilon}b] \rangle \gamma_{i+1}(G_2)$.
\end{enumerate}
\end{theorem}

In the remainder of this section we deal with the case of $G_3$.
For the moment, let us only assume that $G$ is non-periodic.
Recall that the action of $G_3$ on $\St_{G_3}(2)$ is then uniserial, as we proved in \Cref{uniserial}.
Uniseriality implies that the normal subgroups of $G_3$ inside $\St_{G_3}(2)$ are just those in the chain
$\{[\St_{G_3}(2),G_3,\overset{i}{\ldots},G_3]\}_{i\ge 0}$, each of index $p$ in its predecessor until the
identity subgroup is reached.
We can exploit this fact to prove the following general result about the indices between terms of the lower central
series of $G_3$.

\begin{lemma}
\label{indices at most p2}
Let $G$ be a non-periodic GGS-group.
Then the following hold:
\begin{enumerate}
\item
$|G_3:G_3'|=p^2$ and $|G_3':\gamma_3(G_3)|=p$.
\item
$|\gamma_i(G_3):\gamma_{i+1}(G_3)|\le p^2$ for $i=3,\ldots,p$.
\item
$|\gamma_i(G_3):\gamma_{i+1}(G_3)|=p$ for $p+1\le i\le c(3)$, where $c(3)$ is the nilpotency class of $G_3$.
\end{enumerate}
\end{lemma}

\begin{proof}
Since $G_3$ is a non-abelian finite $p$-group and $|G_3:G_3'|\le |G:G'|=p^2$, (i) readily follows.
On the other hand, we have $\gamma_{p+1}(G_3)\le \St_{G_3}(2)$ by \Cref{LCS G2}, and then uniseriality of the
action of $G_3$ on $\St_{G_3}(2)$ implies that $|\gamma_i(G_3):\gamma_{i+1}(G_3)|=p$ for $p+1\le i\le c(3)$.
So we are only left with proving (ii).

It is a well-known fact in group theory that, if $K$, $L$, and $N$ are normal subgroups of a finite group $G$ with
$L\le K$ then $|K:L|=|KN:LN|\, |K\cap N:L\cap N|$.
This implies that
\begin{align*}
|\gamma_i(G_3):\gamma_{i+1}(G_3)|
&=
|\gamma_i(G_2):\gamma_{i+1}(G_2)| \, |\gamma_i(G_3)\cap \St_{G_3}(2):\gamma_{i+1}(G_3)\cap \St_{G_3}(2)|
\\
&=
p \, |\gamma_i(G_3)\cap \St_{G_3}(2):\gamma_{i+1}(G_3)\cap \St_{G_3}(2)|,
\end{align*}
for $3\le i\le p$.
If we set $S=\gamma_i(G_3)\cap \St_{G_3}(2)$ then $[S,G_3]\le \gamma_{i+1}(G_3)\cap \St_{G_3}(2)$.
Since $|S:[S,G_3]|\le p$ by uniseriality, the result follows.
\end{proof}

Uniseriality also simplifies the identification of normal subgroups of $G_3$ inside $\St_{G_3}(2)$.
Assume that $N$ is a subgroup of $\St_{G_3}(2)$ of index $p^i$ that is normal in $G_3$.
Since
\[
|\St_{G_3}(2)| = \frac{|G_3|}{|G_2|} = \frac{p^{\,p^2+1}}{p^{\,p+1}} = p^{\,p^2-p},
\]
if we have $i\le p^2-p$ then $N$ coincides with
$[\St_{G_3}(2),G_3,\overset{i}{\ldots},G_3]$.
On other occasions, we will need to establish in which difference
$[\St_{G_3}(2),G_3,\overset{i}{\ldots},G_3]\smallsetminus [\St_{G_3}(2),G_3,\overset{i+1}{\ldots},G_3]$ an element
$g\in\St_{G_3}(2)$ lies.
To this purpose, we will use the map $\psi$ as indicated in the next lemma, which we prove under
the additional assumption that $G$ is also regular branch over $G'$.
A similar result, but not the same, can be stated if $G$ is only regular branch over $\gamma_3(G)$ and not over $G'$.
However, since later results (starting with \Cref{position commutators}) do not have sufficiently strong counterparts
in this other setting, we have decided to restrict ourselves to the case where $G$ is regular branch over $G'$.

\begin{lemma}
\label{position in uniserial series}
Let $G$ be a non-periodic GGS-group that is regular branch over $G'$.
Then, for a given element $g\in\St_G(2)$, the following hold:
\begin{enumerate}
\item[(i)]
$\psi(g) \equiv ([b,a]^{\lambda_0},\ldots,[b,a]^{\lambda_{p-1}})
\mod (\gamma_3(G)\times \overset{p}{\cdots} \times \gamma_3(G))$,
where $\lambda_0,\ldots,\lambda_{p-1}\in\F_p$.
\item[(ii)]
If we set $\boldsymbol{\lambda}=(\lambda_0,\ldots,\lambda_{p-1})\in B(\F_p)$ then, for every $i=1,\ldots,p$, 
we have $g\in [\St_{G_3}(2),G_3,\overset{i}{\ldots},G_3]$ if and only if $\boldsymbol{\lambda}\in\gamma_{i+1}(W(\F_p))$.
\end{enumerate}
\end{lemma}

\begin{proof}
Since $G$ is non-periodic, we know from \Cref{sec:uniseriality} that $\St_G(1)'=\St_G(2)$.
On the other hand, from Lemma 3.4 of \cite{Gustavo}, we have
\begin{equation}\label{psi stG(2)tow}
\psi(\St_G(2)) = G'\times \overset{p}{\cdots} \times G'.
\end{equation}
This already implies (i), since the quotient $G'/\gamma_3(G)$ is of order $p$, generated by $[b,a]$.

On the other hand, we also have
\[
\psi(\gamma_3(\St_G(1)))=\gamma_3(G)\times \overset{p}{\cdots} \times \gamma_3(G).
\]
From this equality and \eqref{psi stG(2)tow}, it follows that
\begin{equation}
\label{psi stG3(2)/gamma3(stG3(1))}
\psi \left( \frac{\St_{G_3}(2)}{\gamma_3(\St_{G_3}(1))} \right)
=
B(G_2'/\gamma_3(G_2)).
\end{equation}
(Recall that we still denote by $\psi$ the maps induced by $\psi$ in quotients.)

In particular,
\[
|\St_{G_3}(2):\gamma_3(\St_{G_3}(1))| = |G_2':\gamma_3(G_2)|^p = p^{\,p},
\]
and consequently
\begin{equation}
\label{psi gamma3 stG3(1)}
\gamma_3(\St_{G_3}(1)) = [\St_{G_3}(2),G_3,\overset{p}{\ldots},G_3],
\end{equation}
by the discussion before the lemma.

Let us see how we can use \eqref{psi stG3(2)/gamma3(stG3(1))} in order to identify the elements lying
in the subgroups $[\St_{G_3}(2),G_3,\overset{i}{\ldots},G_3]$ for $0\le i\le p$.
Since $\St_{G_3}(2)=\St_{G_3}(1)'$ is central in $\St_{G_3}(1)$ modulo $\gamma_3(\St_{G_3}(1))$, we have
\[
[\St_{G_3}(2),G_3,\overset{i}{\ldots},G_3]
=
[\St_{G_3}(2),a,\overset{i}{\ldots},a] \gamma_3(\St_{G_3}(1))
\]
for $0\le i\le p$.
It follows that
\begin{align*}
\psi \left( \frac{[\St_{G_3}(2),G_3,\overset{i}{\ldots},G_3]}{\gamma_3(\St_{G_3}(1))} \right)
&=
[B(G_2'/\gamma_3(G_2)),\sigma,\overset{i}{\ldots},\sigma]
\\
&=
\Delta^{i} (B(G_2'/\gamma_3(G_2))).
\end{align*}
By \eqref{gammaiW(A) and Delta}, this subgroup coincides with $\gamma_{i+1}(W(G_2'/\gamma_3(G_2)))$ for
$i=1,\ldots,p$.
Since $G_2'/\gamma_3(G_2)$ is cyclic of order $p$ generated by $[b,a]$, an element $g\in\St_{G_3}(2)$ lies in
$[\St_{G_3}(2),G_3,\overset{i}{\ldots},G_3]$ if and only if the tuple $\boldsymbol{\lambda}\in B(\F_p)$ defined in (i)
belongs to $\gamma_{i+1}(W(\F_p))$.
This proves (ii).
\end{proof}

Notice that the condition $\boldsymbol{\lambda}\in\gamma_{i+1}(W(\F_p))$ in part (ii) above can be determined by checking whether $\boldsymbol{\lambda}$ satisfies the corresponding equations given in \Cref{equations for gammaiW(F)}.

\vspace{10pt}

We will use \Cref{position in uniserial series} in the proof of \Cref{position commutators} below, where we establish the position of some iterated commutators in $a$ and $b$ of lengths $p$ and $p+1$, which will be key to the determination of the lower central series of $G_3$.
We need the following two preliminary results.

\begin{lemma}
\label{comm a with sigma p-2 times}
Let $\boldsymbol{\lambda}=(\lambda_0,\ldots,\lambda_{p-1})\in B(\F_p)$.
Then in $W(\F_p)$ we have
\[
[\boldsymbol{\lambda},\sigma,\overset{p-2}{\ldots},\sigma]
=
(-\alpha+\beta,-2\alpha+\beta,\ldots,-(p-1)\alpha+\beta,\beta),
\]
where $\alpha=\sum_{i=0}^{p-1}\, \lambda_i$ and $\beta=\sum_{i=1}^{p-1}\, i\lambda_i$.
\end{lemma}

\begin{proof}
We use the map $\Theta:B(\F_p)\rightarrow \F_p[X]/(X^p-1)$ introduced in \Cref{sec:LCS W(G)}.
Let $f(X)=\sum_{i=0}^{p-1} \, \lambda_iX^i$.
Then taking the commutator of $\boldsymbol{\lambda}$ with $\sigma$, $p-2$ times, corresponds under $\Theta$ to multiplying
$\overline{f(X)}$ with $\overline{X-1}^{\,p-2}$.
Since $f(1)=\alpha$ and $f'(1)=\beta$, we can write $f(X)=\alpha+\beta(X-1)+g(X)(X-1)^2$ in $\F_p[X]$
and then
\[
\overline{f(X)}\cdot \overline{X-1}^{\,p-2} = \alpha \cdot \overline{X-1}^{\,p-2} + \beta \cdot \overline{X-1}^{\,p-1}
\]
in $\F_p[X]/(X^p-1)$.
Now we have
\[
(X-1)^{p-2} = - \sum_{i=0}^{p-2} \, (i+1)X^i
\]
in $\F_p[X]$, since
\begin{equation}
\label{binomial p-2}
\binom{p-2}{i}
=
\frac{(p-2)(p-3)\cdots (p-i-1)}{i!}
\equiv
(-1)^i (i+1)
\mod p,
\end{equation}
and we already know from \eqref{X-1 to p-1} that
\[
(X-1)^{p-1} = \sum_{i=0}^{p-1} \, X^i.
\]
Hence
\[
\overline{f(X)}\cdot \overline{X-1}^{\,p-2}
=
\sum_{i=0}^{p-2} \, \left(-(i+1)\alpha+\beta \right) \cdot \overline{X}^i + \beta \cdot \overline{X}^{p-1},
\]
and the result follows.
\end{proof}

\begin{lemma}
\label{comm bap-2}
Let $G$ be a GGS-group with defining vector $\mathbf{e}=(e_1,\ldots,e_{p-1})$.
Then
\begin{multline*}
\psi([b,a,\overset{p-2}{\ldots},a])
\equiv
(a^{\delta-2\varepsilon}b^{-2},a^{\delta-3\varepsilon}b^{-3},\ldots,a^{\delta-(p-1)\varepsilon}b^{-(p-1)},a^{\delta},
a^{\delta-\varepsilon}b^{-1})
\\
\mod (G' \times \cdots \times G').
\end{multline*}
\end{lemma}

\begin{proof}
Since $G/G'=\langle aG' \rangle \times \langle bG' \rangle$, we can determine separately
the components of the tuple $\psi([b,a,\overset{p-2}{\ldots},a])$ modulo $\langle a \rangle G'$ as powers
of $b$, and its components modulo $\langle b \rangle G'$ as powers of $a$, and then bring these two
results together.

We first determine the components modulo $\langle a \rangle G'$.
We have
\[
\psi(b) \equiv (1,\ldots,1,b) \mod (\langle a \rangle G' \times \cdots\times \langle a \rangle G').
\]
Now taking the commutator with $a$ in $G$ corresponds to taking the commutator with $\sigma$ after applying $\psi$.
Since we are interested in the components modulo $\langle a \rangle G'$, we want the value of the commutator
of $(1,\ldots,1,b)$ with $\sigma$, $p-2$ times, in the group $W(G/\langle a \rangle G')$.
Since $G/\langle a \rangle G'$ is cyclic of order $p$, generated by the image of $b$, we can use
\Cref{comm a with sigma p-2 times}, applied to the vector $\boldsymbol{\lambda}=(0,\ldots,0,1)$, and get
\[
\psi([b,a,\overset{p-2}{\ldots},a])
\equiv
(b^{-2},b^{-3},\ldots,b^{-(p-1)},1,b^{-1})
\mod (\langle a \rangle G' \times \cdots \times \langle a \rangle G').
\]

We can similarly obtain the components modulo $\langle b \rangle G'$, again by using \Cref{comm a with sigma p-2 times},
this time with $\boldsymbol{\lambda}=(e_1,\ldots,e_{p-1},0)$.
Note that we have $\alpha=\varepsilon$ and
\[
\beta = e_2 + 2e_3 + \cdots + (p-2)e_{p-1} = \delta - \varepsilon.
\]
Hence
\[
\psi([b,a,\overset{p-2}{\ldots},a])
\equiv
(a^{\delta-2\varepsilon},a^{\delta-3\varepsilon},\ldots,a^{\delta-(p-1)\varepsilon},a^{\delta},a^{\delta-\varepsilon})
\mod (\langle b \rangle G' \times \cdots \times \langle b \rangle G').
\]

From the calculations above, it follows that
\begin{multline*}
\psi([b,a,\overset{p-2}{\ldots},a])
\equiv
(a^{\delta-2\varepsilon}b^{-2},a^{\delta-3\varepsilon}b^{-3},\ldots,a^{\delta-(p-1)\varepsilon}b^{-(p-1)},a^{\delta},
a^{\delta-\varepsilon}b^{-1})
\\
\mod (G' \times \cdots \times G'),
\end{multline*}
as desired.
\end{proof}

After these two lemmas, we can now proceed to study the position of some relevant commutators in $a$ and $b$
of lengths $p$ and $p+1$.
Observe that the result is more precise when $\delta\ne 0$, i.e.\ when $G$ is of FG-type, in the sense that it
determines the difference in which the commutator lies when we consider the uniserial chain
$\{[\St_{G_3}(2),G_3,\overset{i}{\ldots},G_3]\}_{i\ge 0}$ of $\St_{G_3}(2)$.

\begin{lemma}
\label{position commutators}
Let $G$ be a non-periodic GGS-group that is regular branch over $G'$.
Then the following hold:
\begin{enumerate}
\item 
$[b,a,\overset{p-1}{\ldots},a]\in \gamma_p(G_3) \smallsetminus \St_{G_3}(2)$.
\item 
$[b,a,\overset{p-2}{\ldots},a,b]\in \St_{G_3}(2)$, and
$[b,a,\overset{p-2}{\ldots},a,b]\in [\St_{G_3}(2),G_3]$ if and only if $\delta=0$.
\item 
$[b,a,\overset{p}{\ldots},a]\in [\St_{G_3}(2),G_3,G_3]$.
\item 
$[b,a,\overset{p-1}{\ldots},a,b]\in [\St_{G_3}(2),G_3]$, and
$[b,a,\overset{p-1}{\ldots},a,b]\in [\St_{G_3}(2),G_3,G_3]$ if and only if $\delta=0$.
\item 
$[b,a,\overset{p-2}{\ldots},a,b,a][b,a,\overset{p-1}{\ldots},a,b]\in [\St_{G_3}(2),G_3]$, and this product lies in
$[\St_{G_3}(2),G_3,G_3]$ if and only if $\delta=0$.
\end{enumerate}
\end{lemma}

\begin{proof}
(i)
This follows directly from \Cref{LCS G2}.

In the remainder of the proof, we let $g=[b,a,\overset{p-2}{\ldots},a]$.
Recall from \Cref{comm bap-2} that
\begin{equation}
\label{psi g}
\psi(g)
\equiv
(a^{\delta-2\varepsilon}b^{-2},a^{\delta-3\varepsilon}b^{-3},\ldots,a^{\delta-(p-1)\varepsilon}b^{-(p-1)},a^{\delta},
a^{\delta-\varepsilon}b^{-1})
\hspace{-5pt}
\mod (G' \times \cdots \times G').
\end{equation}
Recall that
\begin{equation}
\label{psi(stG(2))}
G' \times \cdots \times G' = \psi(\St_G(1)') = \psi(\St_G(2)),
\end{equation}
since $G$ is regular branch over $G'$ and non-periodic.
In the proof we will use part (ii) of \Cref{position in uniserial series}, which allows us to determine whether an element lies in a subgroup of the form $[\St_{G_3}(2),G_3,\overset{i}{\ldots},G_3]$ by working in the wreath product $W(G_2'/\gamma_3(G_2))$.

(ii)
Clearly $[g,b]\in \St_G(1)'=\St_G(2)$.
Let us now determine when this commutator belongs to $[\St_{G_3}(2),G_3]$.
To this purpose, we take the commutator of \eqref{psi g} with $\psi(b)$, getting
\begin{equation}
\label{psi bap-2a}
\begin{split}
\psi([g,b])
\equiv
([b,a]^{-2e_1},[b,a]^{-3e_2},\ldots,[b,a]^{-(p-1)e_{p-2}},1,[b,a]^{\varepsilon-\delta})
\\
\mod (\gamma_3(G)\times \overset{p}{\cdots} \times \gamma_3(G)).
\end{split}
\end{equation}
By \Cref{position in uniserial series}, the condition $[g,b]\in [\St_{G_3}(2),G_3]$ is equivalent to
\[
(-2e_1,-3e_2,\ldots,-(p-1)e_{p-2},0,\varepsilon-\delta)\in \gamma_2(W(\F_p)).
\]
By \Cref{equations for gamma2W(F) and gamma3W(F)}, we need to check whether the sum of the components
of this tuple is $0$ in $\F_p$.
Since
\[
-\sum_{i=1}^{p-2} \, (i+1)e_i
=
-\sum_{i=1}^{p-1} \, (i+1)e_i
=
-\varepsilon-\delta,
\]
the sum that we want is $-2\delta$.
Since $2\ne 0$ in $\F_p$, the result follows.

(iii)
This time we take the commutator of \eqref{psi g} with $\psi(a)$, getting
\begin{equation}
\label{comm bap-1}
\begin{split}
\psi([g,a])
&\equiv
(ba^{\varepsilon}[a^{\varepsilon},b^{-1}],ba^{\varepsilon}[a^{\varepsilon},b^{-2}],\ldots,ba^{\varepsilon}[a^{\varepsilon},
b^{-(p-1)}],ba^{\varepsilon})
\\
&\equiv
(ba^{\varepsilon}[b,a]^{\varepsilon},ba^{\varepsilon}[b,a]^{2\varepsilon},\ldots,ba^{\varepsilon}[b,a]^{(p-1)\varepsilon},ba^{\varepsilon})
\mod{\psi([\St_{G_3}(2),G_3])}.
\end{split}
\end{equation}
For the modulus, we have taken into account that \eqref{psi gamma3 stG3(1)} yields
\[
\gamma_3(G_2) \times \overset{p}{\cdots} \times \gamma_3(G_2)
=
\psi(\gamma_3(\St_{G_3}(1)))
=
\psi([\St_{G_3}(2),G_3,\overset{p}{\ldots},G_3]).
\]

Taking the commutator with $\psi(a)$ again, we have
\begin{equation}
\label{psi gaa}
\psi([g,a,a])\equiv ([b,a]^{-\varepsilon},\overset{p}{\ldots},[b,a]^{-\varepsilon})
\mod{\psi([\St_{G_3}(2),G_3,G_3])}.
\end{equation}
By working in $W(G_2'/\gamma_3(G_2))$ and observing that
\[
(-\varepsilon,\overset{p}{\ldots},-\varepsilon)
\in
\gamma_p( W(\F_p) )
\]
by \Cref{gammapW(F)}, it follows that
\[
([b,a]^{-\varepsilon},\overset{p}{\ldots},[b,a]^{-\varepsilon})
\in\psi([\St_{G_3}(2),G_3,\overset{p-1}{\cdots},G_3]).
\]
The result follows.

(iv)
By taking the commutator of \eqref{comm bap-1} with $\psi(b)$, we get
\begin{equation}
\label{psi gab}
\psi([g,a,b])
\equiv
([b,a]^{e_1},[b,a]^{e_2},\ldots,[b,a]^{e_{p-1}},[b,a]^{-\varepsilon})
\mod (\gamma_3(G)\times \overset{p}{\cdots} \times \gamma_3(G)).
\end{equation}
Working again in $W(G_2'/\gamma_3(G_2))$, we immediately get $[g,a,b]\in [\St_{G_3}(2),G_3]$, since the
sum of the exponents in \eqref{psi gab} is $0$.
On the other hand,
\[
[g,a,b]\in [\St_{G_3}(2),G_3,G_3]
\]
is equivalent to
\[
(e_1,e_2,\ldots,e_{p-1},-\varepsilon) \in \gamma_3(W(\F_p)).
\]
From \Cref{equations for gamma2W(F) and gamma3W(F)}, this amounts to
the condition $\delta=0$.

(v)
By (ii) and (iv), we have $[g,b,a][g,a,b]\in [\St_{G_3}(2),G_3]$.
Let us see that this product lies in $[\St_{G_3}(2),G_3,G_3]$ if and only if $\delta=0$.
To this end, we take the commutator of (\ref{psi bap-2a}) with $\psi(a)$, obtaining
\begin{multline*}
\psi([g,b,a])
\equiv
(
[b,a]^{2e_1+\varepsilon-\delta},
[b,a]^{3e_2-2e_1},\ldots,[b,a]^{(p-1)e_{p-2}-(p-2)e_{p-3}},
\\
[b,a]^{-(p-1)e_{p-2}},
[b,a]^{\delta-\varepsilon})
\mod (\gamma_3(G)\times \overset{p}{\cdots} \times \gamma_3(G)).
\end{multline*}
By combining this congruence with \eqref{psi gab}, we get that the exponents of $[b,a]$
in the tuple $\psi([g,b,a][g,a,b])$ are given by the vector
\begin{equation}
\label{vector gba gab}
\ \ \ \ 
(3e_1+\varepsilon-\delta,4e_2-2e_1,5e_3-3e_2,\ldots,-(p-2)e_{p-3}, e_{p-1}-(p-1)e_{p-2},\delta-2\varepsilon).
\end{equation}
Working in $W(G_2'/\gamma_3(G_2))$, it follows that $[g,b,a][g,a,b]\in [\St_{G_3}(2),G_3,G_3]$
if and only if the vector in \eqref{vector gba gab} lies in $\gamma_{3}(W(\F_p))$.
By \Cref{equations for gamma2W(F) and gamma3W(F)}, this is equivalent to the sum
\[
3e_1+\varepsilon-\delta + \sum_{i=2}^{p-1} \, i \big( (i+2)e_i - ie_{i-1} \big)
=
\varepsilon-\delta - \sum_{i=1}^{p-1} \, e_i
=
-\delta
\]
being zero in $\F_p$.
This completes the proof.
\end{proof}

At this point, we can start describing the lower central series of $G_3$ in the specific case
when $G$ is a GGS-group of FG-type.
We begin with the terms $\gamma_i(G_3)$ with $i\ge p+1$.

\begin{lemma}
\label{gamma_p+i}
Let $G$ be a GGS-group of FG-type.
Then for every $i\ge 1$, we have
\begin{equation}
\label{gamma p+i}
\gamma_{p+i}(G_3)
=
[\St_{G_3}(2),G_3,\overset{i}{\ldots},G_3]
=
\langle [b,a^{\varepsilon}b,\overset{p+i-1}{\ldots},a^{\varepsilon}b] \rangle \gamma_{p+i+1}(G_3).
\end{equation}
Furthermore,
\begin{equation}
\label{st(2)}
\St_{G_3}(2) = \langle [b,a^{\varepsilon}b,\overset{p-2}{\ldots},a^{\varepsilon}b,b] \rangle \gamma_{p+1}(G_3).
\end{equation}
\end{lemma}

\begin{proof}
By \Cref{LCS G2} we have
\[
\gamma_p(G_3) \equiv \langle [b,a,\overset{p-1}{\ldots},a] \rangle \mod \St_{G_3}(2).
\]
It readily follows that
\[
\gamma_{p+1}(G_3) \equiv \langle [b,a,\overset{p}{\ldots},a], [b,a,\overset{p-1}{\ldots},a,b] \rangle^{G_3}
\mod [\St_{G_3}(2),G_3].
\]
From (iii) and (iv) of \Cref{position commutators}, $\gamma_{p+1}(G_3)$ is contained in
$[\St_{G_3}(2),G_3]$ and contains the element
\[
[b,a,\overset{p-1}{\ldots},a,b] \in [\St_{G_3}(2),G_3]\smallsetminus [\St_{G_3}(2),G_3,G_3].
\]
The uniseriality of the action of $G_3$ on $\St_{G_3}(2)$ then implies that
\begin{equation}
\label{gamma p+1}
\gamma_{p+1}(G_3) = [\St_{G_3}(2),G_3] = \langle [b,a,\overset{p-1}{\ldots},a,b] \rangle [\St_{G_3}(2),G_3,G_3],
\end{equation}
and consequently also that
\[
\gamma_{p+i}(G_3) = [\St_{G_3}(2),G_3,\overset{i}{\ldots},G_3]
\]
for every $i\ge 1$.
Since $|\St_{G_3}(2):\gamma_3(\St_{G_3}(1))|=p^{\,p}$, it follows in particular that $\gamma_3(\St_{G_3}(1))=\gamma_{2p}(G_3)$.

On the other hand,
\begin{equation}
\label{comm with a epsilon b}
[b,a^{\varepsilon}b,\overset{p}{\ldots},a^{\varepsilon}b]
\equiv
[b,a,\overset{p}{\ldots},a]^{\varepsilon^p} \,
\prod_{i=1}^{p-1} \, [b,a,\overset{i}{\ldots},a,b,a,\overset{p-i-1}{\ldots},a]^{\varepsilon^{p-1}}
\mod \gamma_{p+2}(G_3),
\end{equation}
since commutators of length $p+1$ are multilinear modulo $\gamma_{p+2}(G_3)$ and any commutator with at least three
occurrences of $b$ lies in $\gamma_3(\St_G(1))$.
Also all elements $[b,a,\overset{p}{\ldots},a]$ and
$[b,a,\overset{i}{\ldots},a,b,a,\overset{p-i-1}{\ldots},a]$ for $1\le i\le p-3$
belong to $[\St_{G_3}(2),G_3,G_3] = \gamma_{p+2}(G_3)$.
Thus \eqref{comm with a epsilon b} reduces to
\[
[b,a^{\varepsilon}b,\overset{p}{\ldots},a^{\varepsilon}b]
\equiv
\big(
[b,a,\overset{p-2}{\ldots},a,b,a]
\,
[b,a,\overset{p-1}{\ldots},a,b]
\big)^{\varepsilon^{p-1}}
\mod \gamma_{p+2}(G_3).
\]
Since $\varepsilon\ne 0$ in $\F_p$, we conclude from (v) in \Cref{position commutators} that
\[
\gamma_{p+1}(G_3) = \langle [b,a^{\varepsilon}b,\overset{p}{\ldots},a^{\varepsilon}b] \rangle \gamma_{p+2}(G_3).
\]
Now \eqref{gamma p+i} follows from uniseriality and from \Cref{uniserial}.

Finally, since $|\St_{G_3}(2):\gamma_{p+1}(G_3)|=p$, we obtain \eqref{st(2)} from (ii) of
\Cref{position commutators} arguing similarly to the last paragraph.
\end{proof}

In the previous lemma, we could have given more easily a generator of each of the sections
$\gamma_{p+i}(G_3)/\gamma_{p+i+1}(G_3)$ by using commutators of the form
\[
[b,a,\overset{p-2}{\ldots},a,b,a^{\varepsilon}b,\overset{i}{\ldots},a^{\varepsilon}b].
\]
However, we have taken the extra effort of replacing this $(p+i)$-fold commutator with
\[
[b,a^{\varepsilon}b,\overset{p+i-1}{\ldots},a^{\varepsilon}b]
\]
for the sake of a more uniform writing of these generators, not only for $G_3$ but also for an
arbitrary congruence quotient $G_n$, as we will prove in \Cref{sec:proof theorem A}.

\vspace{8pt}

Now it is easy to give the final result about the lower central series of $G_3$.

\begin{theorem}
\label{LCS G3}
Let $G$ be a GGS-group of FG-type.
Then the following hold:
\begin{enumerate}
\item
The indices between consecutive terms of the lower central series of $G_3$ are given as follows:
\begin{align*}
|\gamma_i(G_3):\gamma_{i+1}(G_3)|
=
\begin{cases}
p^2 \quad &\text{if $i\in \{1,p\}$},\\
p \quad &\text{if $i\in \{2,\dots,p-1,p+1,\dots,p^2-1\}$},\\
1 \quad &\text{if $i\geq p^2$}.
\end{cases}
\end{align*}
In particular, $G_3$ has nilpotency class $p^2-1$.
\item
Furthermore, for $i\in\{2,\dots,p-1,p+1,\dots,p^2-1\}$, we have
\[
\gamma_i(G_3) = \langle [b,a^{\varepsilon}b,\overset{i-1}{\ldots},a^{\varepsilon}b] \rangle \gamma_{i+1}(G_3),
\]
and, on the other hand,
\[
\gamma_p(G_3) =
\langle [b,a^{\varepsilon}b,\overset{p-1}{\ldots},a^{\varepsilon}b], [b,a^{\varepsilon}b,\overset{p-2}{\ldots},a^{\varepsilon}b,b] \rangle
\gamma_{p+1}(G_3).
\]
\item
For all $i\ge p+1$ we have $\psi(\gamma_{i}(G_3))=\gamma_{i+1}(W(G_2))$.
\end{enumerate}
\end{theorem}

\begin{proof}
We start by proving (i) and (ii).
These immediately follow from \Cref{gamma_p+i} for $i\ge p+1$, using the uniseriality of the action of
$G_3$ on $\St_{G_3}(2)$ to ensure that the indices between consecutive terms are equal to $p$.
Also, since $|\St_{G_3}(2)|=p^{\,p^2-p}$, the first trivial term of the lower central series
of $G_3$ is $\gamma_{p^2}(G_3)$, and $G_3$ is nilpotent of class $p^2-1$.

Regarding the sections $\gamma_i(G_3)/\gamma_{i+1}(G_3)$ with $i\le p-1$, observe that \eqref{st(2)}
implies that $\St_{G_3}(2)\le \gamma_p(G_3)$.
Hence (i) and (ii) follow immediately from \Cref{LCS G2} in this case.
Now the result for $\gamma_p(G_3)/\gamma_{p+1}(G_3)$ is a direct consequence of \Cref{LCS G2} and
\eqref{st(2)}.

Let us finally prove (iii).
We have $\gamma_{p+1}(W(G_2))=B(G_2')=\psi(\St_{G_3}(2))$, where the first equality follows from
\Cref{LCS of W(G)}.
On the other hand, since $|\gamma_j(G_2):\gamma_{j+1}(G_2)|=p$ for $j=2,\ldots,p$, it follows from (ii)
of \Cref{LCS of W(G) detail} that $|\gamma_i(W(G_2)):\gamma_{i+1}(W(G_2))|=p$ for $i=p+1,\ldots,p^2$.
Also $\psi^{-1}(\gamma_{i+1}(W(G_2)))$ is normal in $G_3$ and properly contained in $\St_{G_3}(2)$
for $i\ge p+1$.
Now \eqref{gamma p+i} and the uniseriality of the action of $G_3$ on $\St_{G_3}(2)$ yield that
$\psi^{-1}(\gamma_{i+1}(W(G_2)))=\gamma_i(G_3)$ for $i\ge p+1$.
This completes the proof.
\end{proof}

We end this section by showing that the conclusion of \Cref{LCS G3}, and hence also that of Theorem A,
are not valid for non-periodic GGS-groups that are regular branch over $G'$ and not of FG-type.
Note that such groups cannot exist for $p=3$.
Indeed, if $p=3$ and $G$ is non-periodic and not of FG-type then $\delta=e_1+2e_2=0$
and so $e_1=-2e_2=e_2$.
Hence the defining vector $\mathbf{e}$ of $G$ is symmetric and $G$ is not regular branch over $G'$.

\begin{proposition}
\label{indices not FG-type}
Let $G$ be a non-periodic GGS-group that is regular branch over $G'$ but not of FG-type.
Then $|\gamma_i(G_3):\gamma_{i+1}(G_3)|=p^2$ for at least one $i\in\{3,\ldots,p-1\}$.
\end{proposition}

\begin{proof}
We argue by way of contradiction.
Since $G$ is non-periodic, \Cref{indices at most p2} implies that
\begin{equation}
\label{assumption indices p}
|\gamma_i(G_3):\gamma_{i+1}(G_3)| = p
\qquad
\text{for all $i=2,\ldots,p-1$,}
\end{equation}
and, on the other hand, that
\begin{equation}
\label{index gammap at most p^2}
|\gamma_p(G_3):\gamma_{p+1}(G_3)| \le p^2.
\end{equation}

Now set $N=[\St_{G_3}(2),G_3,G_3]$.
Under the standing assumption that $G$ is non-periodic and regular branch over $G'$, \Cref{position commutators}
applies.
Since $G$ is not of FG-type, we have $\delta=0$ and consequently $[b,a,\overset{p-2}{\ldots},a,b]\in [\St_{G_3}(2),G_3]$.
Thus $[b,a,\overset{p-2}{\ldots},a,b]$ is central in $G_3$ modulo $N$.
As a consequence,
\[
\gamma_{p+1}(G_3) \equiv \langle [b,a,\overset{p}{\ldots},a], [b,a,\overset{p-1}{\ldots},a,b] \rangle \gamma_{p+2}(G_3)
\mod N.
\]
Now, by (iii) and (iv) of \Cref{position commutators}, both 
$[b,a,\overset{p}{\ldots},a]$ and $[b,a,\overset{p-1}{\ldots},a,b]$ belong to $N$.
Thus $\gamma_{p+1}(G_3) \equiv \gamma_{p+2}(G_3) \mod N$ and, since $G_3$ is a finite $p$-group, this implies that
$\gamma_{p+1}(G_3)\le N$.
Now, on the one hand,
\[
|G_3:\gamma_{p+1}(G_3)| \ge |G_3:N| = |G_3:\St_{G_3}(2)| \cdot |\St_{G_3}(2):[\St_{G_3}(2),G_3,G_3]|
= p^{\,p+3}
\]
while, on the other hand, by the standing assumption \eqref{assumption indices p},
\[
|G_3:\gamma_{p+1}(G_3)| = \prod_{i=1}^p \, |\gamma_i(G_3):\gamma_{i+1}(G_3)|
= p^{\,p} \, |\gamma_p(G_3):\gamma_{p+1}(G_3)|.
\]
Hence $|\gamma_p(G_3):\gamma_{p+1}(G_3)|\ge p^3$, which is a contradiction by \eqref{index gammap at most p^2}.
The result follows.
\end{proof}

\begin{remark}
The underlying reason for the different behaviour of the lower central series for the groups of FG-type and for those considered in
\Cref{indices not FG-type} has to do with the place where the subgroup $\gamma_{p+1}(G_3)$ enters the uniserial chain of
$S=\St_{G_3}(2)$.
In the former case, we have $\gamma_{p+1}(G_3)=[S,G_3]$, as stated in \Cref{gamma_p+i}.
In the latter case, the proof of \Cref{indices not FG-type} shows that $\gamma_{p+1}(G_3)\le [S,G_3,G_3]$, skipping a step of the chain
and forcing an extra index $p^2$ between consecutive terms of the lower central series.
Thus uniseriality alone is not enough to determine the lower central series of $G_3$: it determines the normal subgroups of $G_3$
inside $S$, but not where the lower central series enters the chain.
\end{remark}

\begin{remark}
\label{examples of LCS G3 in general}
Let $G$ be a non-periodic GGS-group that is regular branch over $G'$.
Calculations performed with the computer algebra system GAP \cite{GAP} and the package AutomGrp \cite{automgrp} show that the
number of positions $i$ for which $|\gamma_i(G_3):\gamma_{i+1}(G_3)|=p^2$ can be arbitrary for $p=5$ and $7$.
Actually, even all values $i$ in the range $\{3,\ldots,p\}$ can yield an index equal to $p^2$.
We give some examples in the tables below: 
\vspace{5pt}
\begin{table}[!h]
\centering
\begin{tabular}{>{\centering\arraybackslash}m{4.0cm} >{\centering\arraybackslash}m{3.2cm} >{\centering\arraybackslash}m{5.0cm}}
\toprule
Defining vector $\mathbf{e}$ & Number of factors of order $p^2$ & Positions of these factors \tabularnewline
\midrule
$(1,0,0,0)$ & $2$ & $1,5$ \\
$(4,0,2,0)$ & $3$ & $1,3,5$ \\
$(0,1,2,3)$ & $4$ & $1,3,4,5$ \\
\bottomrule
\end{tabular}
\vspace{5pt}
\caption{Examples of lower central series behaviour for $p=5$.}
\end{table}

\begin{table}[ht]
\centering
\begin{tabular}{>{\centering\arraybackslash}m{4.8cm} >{\centering\arraybackslash}m{3.2cm} >{\centering\arraybackslash}m{5.0cm}}
\toprule
Defining vector $\mathbf e$ & Number of factors of order $p^2$ & Positions of these factors
\tabularnewline
\midrule
$(1,0,0,0,0,0)$ & $2$ & $1,7$ \\
$(6,0,2,6,0,1)$ & $3$ & $1,5,7$ \\
$(0,0,4,6,6,3)$ & $3$ & $1,3,7$ \\
$(2,2,0,4,1,6)$ & $4$ & $1,5,6,7$ \\
$(0,6,3,5,5,3)$ & $5$ & $1,3,5,6,7$ \\
$(0,1,2,3,4,5)$ & $6$ & $1,3,4,5,6,7$ \\
\bottomrule
\end{tabular}
\vspace{5pt}
\caption{Examples of lower central series behaviour for $p=7$.}
\end{table}
\end{remark}

\begin{remark}
\Cref{indices not FG-type} need not be true for non-periodic GGS-groups that are regular branch but not regular branch over $G'$.
For example, for $p=5$ the group $G$ with defining vector $\mathbf{e}=(0,1,1,0)$ satisfies that
$|\gamma_i(G_3):\gamma_{i+1}(G_3)|=5$ for $i=3,4$.
\end{remark}

\section{Proof of Theorem A}
\label{sec:proof theorem A}

In this section, we establish \Cref{main}, which describes the lower central series of a GGS-group $G$ of FG-type.
We split this section into three subsections.
We begin by stating the analogous result for the congruence quotients $G_n$, \Cref{main for quotients} below.
As already indicated in the introduction, Theorem A will easily follow from this case. 
The proof of \Cref{main for quotients} proceeds by induction on $n$, with the base case $n=3$ given in
\Cref{LCS G3}.
In the second subsection, we give a road map of our strategy for the induction step, which will help the reader
follow the technical intermediate results that are needed.
The final subsection proves these supporting results, culminating in the proof of \Cref{main for quotients}
and \Cref{main}.

\subsection{The version of Theorem A for congruence quotients}
In the remainder, we denote by $c(n)$ the nilpotency class of $G_n$.
From the previous section, we have $c(1)=1$, $c(2)=p$ and $c(3)=p^2-1$.
We also let $\ell(m)$, $r(m)$, $x(i)$, and $y(i)$ be as defined in the introduction. 
Note that, by \Cref{LCS G3}, we have
\[
\gamma_i(G_3) = \langle x(i) \rangle \, \gamma_{i+1}(G_3)
\quad
\text{for $2\le i\le c(3)$, $i\ne p$,}
\]
in which case $|\gamma_i(G_3):\gamma_{i+1}(G_3)|=p$, and 
\[
\gamma_p(G_3) = \langle x(p), y(p) \rangle \, \gamma_{p+1}(G_3),
\]
with $|\gamma_p(G_3):\gamma_{p+1}(G_3)|=p^2$.

\vspace{8pt}

The proof of Theorem A will easily follow from our next theorem, which generalises the
results about $G_3$ in the last paragraph to an arbitrary congruence quotient $G_n$.
More precisely, we determine the nilpotency class of $G_n$, and we show that the index
$|\gamma_i(G_n):\gamma_{i+1}(G_n)|$ is $p$ or $p^2$, and that the
quotient $\gamma_i(G_n)/\gamma_{i+1}(G_n)$ can be generated either by $x(i)$ alone or by
$x(i)$ together with $y(i)$, respectively.
Which of the two possibilities occurs is determined by the position of $i$
with respect to the sequences $\{\ell(m)\}_{m\ge 0}$ and $\{r(m)\}_{m\ge 0}$.

\begin{theorem}
\label{main for quotients}
Let $G$ be a GGS-group of FG-type, and let $n\in\N$, with $n\ge 3$.
Then
\[
c(n)=p^{n-1}-p^{n-3}-p^{n-4}-\cdots-p-1
\]
and the following hold:
\begin{enumerate}
\item 
If $\ell(m)\le i < r(m)$ for some $m\in \{0,\ldots,n-2\}$, then 
\[
|\gamma_i(G_n):\gamma_{i+1}(G_n)|=p^2,
\]
and $\gamma_i(G_n)=\langle x(i), y(i) \rangle \gamma_{i+1}(G_n)$.
\item 
If $r(m)\le i < \ell(m+1)$ for some $m\in \{0,\ldots,n-3\}$, or if $r(n-2)\le i \le c(n)$ then
\[
|\gamma_i(G_n):\gamma_{i+1}(G_n)|=p,
\]
and $\gamma_i(G_n)=\langle x(i)\rangle\gamma_{i+1}(G_n)$.
\end{enumerate}
%Furthermore, for $i\ge p+1$ we have 
%\[
%\psi(\gamma_p(G_n))\geq \gamma_{i+1}(W(G_{n-1}))\geq\psi(\gamma_{i}(G_n))\geq \gamma_{i+2}(W(G_{n-1}))
%\]
%and if $i=p+1$, $r(m)-1\leq i\leq \ell(m+1)-1$, or $r(n-2)-1\leq i$, then 
%\begin{align*}
%\gamma_{i+1}(W(G_{n-1}))=\psi(\gamma_{i}(G_n)).
%\end{align*}
\end{theorem}

\begin{remark}
\label{intervals}
According to \Cref{main for quotients}, for a given positive integer $i$, if $i$ lies in one of the
intervals $[\ell(m),r(m)-1]$ with $0\le m\le n-2$, then we have $|\gamma_i(G_n):\gamma_{i+1}(G_n)|=p^2$,
while for $i$ in the intervals $[r(m),\ell(m+1)-1]$ with $0\le m\le n-3$ or in the final interval $[r(n-2),c(n)]$,
we have $|\gamma_i(G_n):\gamma_{i+1}(G_n)|=p$.
We refer to the former intervals as \emph{$p^2$-intervals} of $G_n$, and to the latter as
\emph{$p$-intervals} of $G_n$.
Note that, for $m\ge 1$, the number of consecutive indices $i$ for which we have $|\gamma_i(G_n):\gamma_{i+1}(G_n)|=p^2$
in the $p^2$-interval $[\ell(m),r(m)-1]$ is
\[
r(m) - \ell(m) = p^{m-1}.
\]
On the other hand, if we compare the intervals of $G_n$ with those of $G_{n-1}$, we can see that $G_n$ has one
more interval of each type.
The extra $p^2$-interval for $G_n$ is $[\ell(n-2),r(n-2)-1]$, providing $p^{n-3}$ new sections with index $p^2$
with respect to $G_{n-1}$.
Since $r(n-2)-1\le c(n-1)$, these sections of index $p^2$ come all from the final $p$-interval of $G_{n-1}$:
passing from $G_{n-1}$ to $G_n$ has the effect of increasing the index of these sections from $p$ to $p^2$.
At the same time, all the quotients of the lower central series that were trivial in $G_{n-1}$, i.e.\ corresponding
to indices $i\ge c(n-1)+1$, integrate into the final $p$-interval of $G_n$.
This does not come as a surprise; we have $\gamma_i(G_n)\le \St_{G_n}(n-1)$ for those indices, and as we
know from \Cref{uniserial}, the action of $G_n$ on $\St_{G_n}(n-1)$ is uniserial, since $G$ is non-periodic.
\end{remark}

We can illustrate \Cref{main for quotients} with the following numerical example:
for $p=3$ and $n=4$, the $p^2$-intervals of $G_4$ are $\{1\}$, $\{3\}$ and $[5,7]$, and $c(4)=23$;
the total $23+5=28$ agrees with the fact that $|G_4|=3^{28}$ from \eqref{order Gn}.

\begin{remark}
As a consequence of \Cref{main for quotients}, we have
\begin{equation}
\label{almost all gens with aepsilonb}
\gamma_i(G_n) \equiv [\gamma_{i-1}(G_n),a^{\varepsilon}b] \mod \gamma_{i+1}(G_n)
\end{equation}
whenever $i$ is not of the form $\ell(m)$ for $0\le m\le n-2$ or, using the terminology of \Cref{intervals},
whenever $i$ is not the starting index of a $p^2$-interval for $G_n$.
Indeed, for those values of $i$, both $x(i)$ and $y(i)$ are $i$-fold commutators
whose last entry is $a^{\varepsilon}b$.
\end{remark}

\vspace{-5pt}

\subsection{Road map of the induction step}
\label{subsection road map}
Since the base of the induction is the case $n=3$, we may assume that $n\ge 4$.
Then \Cref{main for quotients} holds for $G_{n-1}$ by the induction hypothesis.

The key observation is that the map $\psi$ embeds $G_n$ into the wreath product $W(G_{n-1})$.
Then our main goal is to understand how the image of the lower central series of $G_n$ under $\psi$ relates to
the lower central series of $W(G_{n-1})$.
Observe that \Cref{LCS of W(G) detail} allows us to describe the lower central factors of $W(G_{n-1})$ in terms of
those of $G_{n-1}$.
More precisely, if $i=(j-1)p+k$ with $1\le j\le c(n-1)$ and $1\le k\le p$, we obtain that
\[
\gamma_i(W(G_{n-1}))/\gamma_{i+1}(W(G_{n-1}))
\cong
\gamma_j(G_{n-1})/\gamma_{j+1}(G_{n-1}),
\]
provided that $i\ge 2$.
As a consequence,
\begin{equation}
\label{index gammai WG = index gammaj G}
\begin{split}
|\gamma_i(W(G_{n-1})):\gamma_{i+1}(W(G_{n-1}))|
&=
|\gamma_j(G_{n-1}):\gamma_{j+1}(G_{n-1})|
\\
&=
\begin{cases}
p^2, & \text{if $\ell(m)\le j < r(m)$ with $0\le m\le n-3$,}
\\
p, & \text{otherwise.}
\end{cases}
\end{split}
\end{equation}

Let us assume that $i\ge p+2$, which is enough for our purpose.
Then the main steps that we need to take for the proof of \Cref{main for quotients} are the following:
\begin{enumerate}
\item[(S1)]
In \Cref{gamma WG vs gamma Gn 1}, we prove that $\psi$ always maps $\gamma_{i-1}(G_n)$ into the section\break
$\gamma_i(W(G_{n-1}))/\gamma_{i+1}(W(G_{n-1}))$ as a subgroup of order $p$, generated by the image of $x(i-1)$.
\item[(S2)]
In particular, if $|\gamma_i(W(G_{n-1})):\gamma_{i+1}(W(G_{n-1}))|=p$ then the image of $\gamma_{i-1}(G_n)$
under $\psi$ covers $\gamma_i(W(G_{n-1}))/\gamma_{i+1}(W(G_{n-1}))$.
\item[(S3)]
On the other hand, if $|\gamma_i(W(G_{n-1})):\gamma_{i+1}(W(G_{n-1}))|=p^2$ then \Cref{gamma WG vs gamma Gn 2}
shows that $\gamma_i(W(G_{n-1}))/\gamma_{i+1}(W(G_{n-1}))$ is covered by the image of the subgroup
$\langle y(i-2) \rangle \gamma_{i-1}(G_n)$ under $\psi$.
This property will be needed to prove the existence of a sandwich.
At this point, the condition that $\delta\ne 0$ plays again an essential role: it ensures that the commutator
$[x(\ell-1),a^{-\delta}]$ appearing in congruence \eqref{first generator} in the proof of \Cref{gamma WG vs gamma Gn 2}
is non-trivial and actually generates the same subgroup as $y(\ell)$ modulo $\langle x(\ell) \rangle \gamma_{\ell+1}(G_{n-1})$.
\item[(S4)]
Then we proceed to prove \Cref{main for quotients} by reverse induction on $i$.
The previous steps describe the image of the lower central series
of $G_n$ under $\psi$, but only in sections of the form $\gamma_i(W(G_{n-1}))/\gamma_{i+1}(W(G_{n-1}))$.
In order to make this information global rather than local, we first show that the last
non-trivial term of the lower central series of $\psi(G_n)$ coincides with that of $W(G_{n-1})$.
Relying on \eqref{index gammai WG = index gammaj G} and (S2), this equality can be extended to more terms of
these two series, which yields the final $p$-interval $[r(n-2),c(n)]$ of the series of $G_n$.
See \Cref{figure p-interval} below.
\item[(S5)]
Then we use \eqref{index gammai WG = index gammaj G} and (S3) to obtain the final $p^2$-interval
$[\ell(n-2),r(n-2)-1]$ of the series of $G_n$.
Its terms appear creating sandwiches with the terms of the series of $W(G_{n-1})$, with
$\psi(\gamma_{i-2}(G_n))=\langle \psi(x(i-2)) \rangle \gamma_i(W(G_{n-1}))$ and
$\gamma_i(W(G_{n-1}))=\langle \psi(y(i-2)) \rangle \psi(\gamma_{i-1}(G_n))$.
See \Cref{figure p2-interval} below.
\item[(S6)]
After proving that the final $p$-interval and the final $p^2$-interval of the lower central series of $G_n$
behave as indicated in \Cref{main for quotients}, it is easy to see that $\St_{G_n}(n-1)\le \gamma_{\ell(n-2)}(G_n)$,
and the result readily follows from the induction hypothesis on $G_{n-1}$.
\end{enumerate}

\vspace{8pt}

\begin{figure}[H]
 \centering
{\footnotesize\begin{tikzpicture}[scale=6]
 \draw[-, thick] (0,0) -- (0,0.84);
 \foreach \x/\xtext in {0/$1=B(\gamma_{c(n-1)+1}(G_{n-1}))=\gamma_{c(n)+2}(W(G_{n-1}))$,0.1/,0.39/$\gamma_{i}(W(G_{n-1}))$,0.25/$\gamma_{i+1}(W(G_{n-1}))$,0.54/,0.64/$B(\gamma_{r(n-3)}(G_{n-1}))=\gamma_{r(n-2)+2}(W(G_{n-1}))$,0.84/$\gamma_{r(n-2)+1}(W(G_{n-1}))$}
 \draw[thick] (0.5pt, \x) -- (-0.5pt,\x) node[left] {\xtext};
 \foreach \y/\ytext in {0/$\psi(\gamma_{c(n)+1}(G_{n}))$,0.1/,0.25/$\psi(\gamma_{i}(G_{n}))$,0.39/$\psi(\gamma_{i-1}(G_{n}))$,.64/$\psi(\gamma_{r(n-2)+1}(G_{n}))$,0.74/$\psi(\gamma_{r(n-2)}(G_{n}))$}
 \draw[thick] (0.5pt, \y) node[right] {\ytext} -- (-0.5pt,\y);
 \draw(0.01,0.05) node[right]{$p$};
 \draw(0.01,0.32) node[right]{$p$};
 \draw(0.01,0.59) node[right]{$p$};
 \draw(0.01,0.69) node[right]{$p$};
 \draw(0.3,0.32) node[right]{$r(n-2)+2\leq i\leq c(n)+1 $};
 \draw[decorate,decoration={brace,mirror}] (-0.05,0.79) -- (-0.05,0.69) node[midway,left=0.1em]{$p^2$};
 \filldraw[black] (0,0.32) circle (0.25pt) node[anchor=east]{$\psi(x(i-1))$};
\end{tikzpicture}}
 \caption{
 Position of the relevant subgroups in the interval $[r(n-2),c(n)]$.
 }
 \label{figure p-interval}
\end{figure}

\vspace{8pt}

\begin{figure}[H]
 \centering
 {\footnotesize\begin{tikzpicture}[scale=6]
 \draw[-, thick] (0,0) -- (0,1.18);
 \foreach \x/\xtext in {0/$B(\gamma_{r(n-3)}(G_{n-1}))=\gamma_{r(n-2)+2}(W(G_{n-1}))$,0.2/$\gamma_{r(n-2)+1}(W(G_{n-1}))$,0.4/$\gamma_{i+1}(W(G_{n-1}))$,0.64/$\gamma_{i}(W(G_{n-1}))$,1.08/$B(\gamma_{\ell(n-3)}(G_{n-1}))=\gamma_{\ell(n-2)+2}(W(G_{n-1}))$,1.18/$\gamma_{\ell(n-2)+1}(W(G_{n-1}))$}
 \draw[thick] (0.5pt, \x) -- (-0.5pt,\x) node[left] {\xtext};
 \foreach \y/\ytext in {0.0/$\psi(\gamma_{r(n-2)+1}(G_{n}))$,0.1/$\psi(\gamma_{r(n-2)}(G_{n}))$,0.5/$\psi(\gamma_{i-1}(G_{n}))$,0.78/$\psi(\gamma_{i-2}(G_{n}))$,0.98/$\psi(\gamma_{\ell(n-2)+1}(G_{n}))$,1.18/$\psi(\gamma_{\ell(n-2)}(G_{n}))$}
 \draw[thick] (0.5pt, \y) node[right] {\ytext} -- (-0.5pt,\y);
 \draw(0.01,0.05) node[right]{$p$};
 \draw(-0.01,1.13) node[left]{$p$};
 \draw(0.4,0.61) node[right]{$\ell(n-2)+2\leq i\leq r(n-2)+1 $};
 \draw[decorate,decoration={brace}] (-0.05,0.05) -- (-0.05,0.15) node[midway,left=0.1em]{$p^2$};
 \draw[decorate,decoration={brace}] (-0.35,0.41) -- (-0.35,0.62) node[midway,left=0.1em]{$p^2$};
 \draw[decorate,decoration={brace,mirror}] (0.28,0.52) -- (0.28,0.75) node[midway,right=0.1em]{$p^2$};
 \draw[decorate,decoration={brace,mirror}] (0.05,1.03) -- (0.05,1.13) node[midway,right=0.1em]{$p^2$};
 \filldraw[black] (0,0.57) circle (0.25pt) node[anchor=east]{$\psi(y(i-2))$};
 \filldraw[black] (0,0.71) circle (0.25pt) node[anchor=west]{$\psi(x(i-2))$};
 \end{tikzpicture}}
 \caption{
 Position of the relevant subgroups in the interval $[\ell(n-2),r(n-2)-1]$.
 }
 \label{figure p2-interval}
\end{figure}

At this point we can explain the definition of the two sequences $\ell(n)$ and $r(n)$ given in the introduction
of the paper.
The initial values correspond to the behaviour of the lower central series of $G_2$ and $G_3$.
Regarding the recursive relation, observe on the one hand that the length of the $p$-intervals and of the
$p^2$-intervals is multiplied by $p$ in passing from $G_{n-1}$ to $W(G_{n-1})$, as follows from
\eqref{index gammai WG = index gammaj G}.
Observe also that the last $p^2$-interval for $W(G_{n-1})$ is $[p(\ell(n-3)-1)+1,p(r(n-3)-1)]$.
Now (S5) shows that $G_n$ inherits from $W(G_{n-1})$ the same number of values in the last $p^2$-interval,
but for indices that are two units less, since the section $\psi(\gamma_{i-2}(G_n))/\psi(\gamma_{i-1}(G_n))$
makes a sandwich with $\gamma_i(W(G_{n-1}))/\gamma_{i+1}(W(G_{n-1}))$.
Thus the last $p^2$-interval for $G_n$ will be $[p(\ell(n-3)-1)-1,p(r(n-3)-1)-2]$.
This shows that $r(n-2)=p(r(n-3)-1)-1$ and $\ell(n-2)=p(\ell(n-3)-1)-1$, which is the recursive relation
that we used to define these two sequences.

\subsection{\texorpdfstring{Proof of \Cref{main for quotients}}
{Proof of Theorem A for congruence quotients}}

We start with a couple of results regarding the effect of taking commutators with $\psi(G_n)$ in some
special cases.
In the following proposition, we consider the commutator $[\mathbf{g},\psi(G_n)]$ with a single tuple
$\mathbf{g}\in\gamma_i(W(G_{n-1}))$, and then we apply this to determine commutators of the form
$[\gamma_i(W(G_{n-1})),\psi(G_n)]$.

\begin{proposition}
\label{comm g with psi gn}
Let $G$ be a non-periodic GGS-group and let $n\ge 3$.
Then for a given tuple $\mathbf{g}\in\gamma_i(W(G_{n-1}))$ we have
\begin{equation}
\label{comm psi gn = comm psi aeb}
[\mathbf{g},\psi(G_n)]
\equiv
\langle [\mathbf{g},\psi(a^{\varepsilon}b)] \rangle
\mod \gamma_{i+2}(W(G_{n-1})).
\end{equation}
Furthermore, if we write $i=(j-1)p+k$ with $j\ge 1$ and $1\le k\le p$,
then the following hold:
\begin{enumerate}
\item
If $1\le k\le p-1$ then
\[
[\mathbf{g},\psi(G_n)] \equiv \langle \Delta(\mathbf{g}) \rangle \mod \gamma_{i+2}(W(G_{n-1})).
\]
\item
If $k=p$ and we write $\mathbf{g}\equiv (g,\overset{p}{\ldots},g) \mod \gamma_{i+1}(W(G_{n-1}))$ with
$g\in\gamma_j(G_{n-1})$,
then
\[
[\mathbf{g},\psi(G_n)] \equiv \langle ([g,a^{\varepsilon}b],1,\overset{p-1}{\ldots},1) \rangle \mod \gamma_{i+2}(W(G_{n-1})).
\]
\end{enumerate}
\end{proposition}

\begin{proof}
To begin with, note that
\[
\psi(G_n) = \psi(G_n'\langle b \rangle \langle a \rangle) = \psi(G_n)' \langle \psi(b) \rangle \langle \sigma \rangle,
\]
and that
\[
[\mathbf{g},\psi(G_n)']
\le
[\gamma_i(W(G_{n-1})),\psi(G_n)']
\le
[\gamma_i(W(G_{n-1})),\psi(G_n),\psi(G_n)]
\le
\gamma_{i+2}(W(G_{n-1})),
\]
where the second inclusion follows from Hall's Three Subgroup Lemma.
Hence
\begin{equation}
\label{comm g psi gn as product}
[\mathbf{g},\psi(G_n)] \equiv \langle [\mathbf{g},\psi(b)] , [\mathbf{g},\sigma] \rangle
\mod \gamma_{i+2}(W(G_{n-1})),
\end{equation}
taking into account that commutators of elements of $\gamma_i(W(G_{n-1}))$ and elements of
$W(G_{n-1})$ behave bilinearly modulo $\gamma_{i+2}(W(G_{n-1}))$.

Suppose first that $1\le k\le p-1$.
Then
\begin{align*}
[\mathbf{g},\psi(b)]
\in
[B(\gamma_j(G_{n-1})),B(G_{n-1})]
&=
B(\gamma_{j+1}(G_{n-1}))
\\
&=
\gamma_{jp+1}(W(G_{n-1}))
\le
\gamma_{i+2}(W(G_{n-1})).
\end{align*}
Thus from \eqref{comm g psi gn as product} we have
\[
[\mathbf{g},\psi(G_n)]
\equiv
\langle [\mathbf{g},\sigma] \rangle
=
\langle \Delta(\mathbf{g}) \rangle
\mod \gamma_{i+2}(W(G_{n-1})).
\]
On the other hand,
\[
[\mathbf{g},\psi(a^{\varepsilon}b)]
=
[\mathbf{g},\sigma^{\varepsilon}\psi(b)]
\equiv
[\mathbf{g},\sigma]^{\varepsilon}
\mod
\gamma_{i+2}(W(G_{n-1})).
\]
This completes the proof of the proposition in the case $1\le k\le p-1$, since $\varepsilon\ne 0$ under the
standing assumption that $G$ is non-periodic.

Assume now that $k=p$.
Since $\sigma$ commutes with constant tuples, we have $[\mathbf{g},\sigma]\in \gamma_{i+2}(W(G_{n-1}))$.
Thus, by \eqref{comm g psi gn as product} we have
\[
[\mathbf{g},\psi(G_n)]\equiv \langle [\mathbf{g},\psi(b)] \rangle \mod \gamma_{i+2}(W(G_{n-1})).
\]
Now
\begin{align*}
[\mathbf{g},\psi(b)]
&\equiv
[(g,\overset{p}{\ldots},g),(a^{e_1},\ldots,a^{e_{p-1}},b)]
\\
&=
([g,a^{e_1}],\ldots,[g,a^{e_{p-1}}],[g,b])
\\
&\equiv
([g,a^{\varepsilon}b],1,\overset{p-1}{\ldots},1)
\mod \gamma_{i+2}(W(G_{n-1})),
\end{align*}
where the last congruence follows from \Cref{reduction to first components}.
Since
\[
[\mathbf{g},\psi(a^{\varepsilon}b)]
=
[\mathbf{g},\sigma^{\varepsilon}\psi(b)]
\equiv
[\mathbf{g},\psi(b)]
\mod \gamma_{i+2}(W(G_{n-1})),
\]
this concludes the proof.
\end{proof}

\begin{corollary}
Let $G$ be a non-periodic GGS-group and let $n\ge 3$.
Then for a given $i\ge 1$, we have
\begin{equation}
\label{comm gammai psi gn}
[\gamma_i(W(G_{n-1})),\psi(G_n)]
\equiv
[\gamma_i(W(G_{n-1})),\psi(a^{\varepsilon}b)]
\mod
\gamma_{i+2}(W(G_{n-1})).
\end{equation}
Furthermore, if we write $i=(j-1)p+k$ with $j\ge 1$ and $1\le k\le p$, then the following hold:
\begin{enumerate}
\item
If $1\le k\le p-1$ then
\begin{equation}
\label{comm gammai psi gn 1=k=p-1}
[\gamma_i(W(G_{n-1})),\psi(G_n)]
\equiv
\gamma_{i+1}(W(G_{n-1}))
\mod
\gamma_{i+2}(W(G_{n-1})).
\end{equation}
\item
If $k=p$ then
\begin{equation}
\label{comm gammai psi gn k=p}
[\gamma_i(W(G_{n-1})),\psi(G_n)]
\equiv
B([\gamma_j(G_{n-1}),a^{\varepsilon}b])
\mod
\gamma_{i+2}(W(G_{n-1})).
\end{equation}
\end{enumerate}
\end{corollary}

\begin{proof}
First of all, note that \eqref{comm gammai psi gn} follows immediately from \eqref{comm psi gn = comm psi aeb}.

(i)
If $1\le k\le p-1$ then
\begin{align*}
[\gamma_i(W(G_{n-1})),\psi(G_n)]
&\equiv
\langle \Delta(\mathbf{g}) \mid \mathbf{g}\in\gamma_i(W(G_{n-1})) \rangle
\\
&\equiv
\gamma_{i+1}(W(G_{n-1}))
\mod
\gamma_{i+2}(W(G_{n-1})),
\end{align*}
where the first congruence follows from (i) of \Cref{comm g with psi gn}, and the second one from (iv)
of \Cref{LCS of W(G) detail}.
This proves the result.

(ii)
If $k=p$ then by (iii) of \Cref{LCS of W(G) detail} and (ii) of \Cref{comm g with psi gn}, we have
\[
[\gamma_i(W(G_{n-1})),\psi(G_n)]
\equiv
\langle ([g,a^{\varepsilon}b],1,\overset{p-1}{\ldots},1) \mid g\in\gamma_j(G_{n-1}) \rangle
\mod
\gamma_{i+2}(W(G_{n-1})).
\]
By \Cref{reduction to first components}, this means that
\[
[\gamma_i(W(G_{n-1})),\psi(G_n)]
\equiv
B([\gamma_j(G_{n-1}),a^{\varepsilon}b])
\mod
\gamma_{i+2}(W(G_{n-1})),
\]
as desired.
\end{proof}

Next we study the relation between the image of the lower central series of $G_n$ under $\psi$
and the lower central series of $W(G_{n-1})$.
We will need the following straightforward lemma.

\begin{lemma}
\label{same index}
Let $G$ be a group and let $N$ be a normal subgroup of $G$.
If the subgroup $\gamma_i(G)$ has finite index in $G$ for some $i\ge 1$ then
we have $|G:\gamma_i(G)|=|G/N:\gamma_i(G/N)|$ if and only if $N\le \gamma_i(G)$.
\end{lemma}

The following proposition corresponds to steps (S1) and (S2) described in
\Cref{subsection road map}.

\begin{proposition}
\label{gamma WG vs gamma Gn 1}
Let $G$ be a GGS-group of FG-type, and let $n\ge 4$.
If \Cref{main for quotients} is true for $G_{n-1}$ then the following hold for $p+2\le i\le p\cdot c(n-1)$:
\begin{enumerate}
\item
If we write $i=(j-1)p+k$ with $1\le k\le p$ then
\begin{equation}
\label{psii-1 modulo}
\psi(\gamma_{i-1}(G_n))
\equiv
\langle \psi(x(i-1)) \rangle
\equiv
\langle \boldsymbol{\lambda}_k(x(j)) \rangle
\mod
\gamma_{i+1}(W(G_{n-1})).
\end{equation}
In particular,
\begin{equation}
\label{psigammai-1Gn in gammaiWGn-1}
\psi(\gamma_{i-1}(G_n))\le \gamma_i(W(G_{n-1})).
\end{equation}
\item
We have
\begin{equation}
\label{index in psii-1}
|\psi(\gamma_{i-1}(G_n)) \, \gamma_{i+1}(W(G_{n-1})):\gamma_{i+1}(W(G_{n-1}))|=p.
\end{equation}
Thus if $|\gamma_i(W(G_{n-1})):\gamma_{i+1}(W(G_{n-1}))|=p$ then
\begin{equation}
\label{psii-1 if index p}
\gamma_i(W(G_{n-1})) = \psi(\gamma_{i-1}(G_n)) \, \gamma_{i+1}(W(G_{n-1})).
\end{equation}
\end{enumerate}
\end{proposition}

\begin{proof}
(i)
We use induction on $i$.
Note that it suffices to prove \eqref{psii-1 modulo}.
Indeed, since $x(j)\in \gamma_j(G_{n-1})$, we have
$\boldsymbol{\lambda}_k(x(j))\in \gamma_i(W(G_{n-1}))$ by (iii) of \Cref{LCS of W(G) detail}.
Thus \eqref{psii-1 modulo} implies that $\psi(\gamma_{i-1}(G_n))\le \gamma_i(W(G_{n-1}))$.

We first deal with the case $i=p+2$, relying on the information about $G_3$ gathered in
\Cref{LCS G3}.
We know that
\begin{align*}
\gamma_{p+2}(W(G_2))
&=
\psi(\gamma_{p+1}(G_3))
=
\psi(\langle x(p+1) \rangle \, \gamma_{p+2}(G_3))
\\
&=
\langle \, \psi(x(p+1)) \, \rangle \, \gamma_{p+3}(W(G_2)).
\end{align*}
Since $\psi(\St_{G_n}(3))=B(\St_{G_{n-1}}(2))$ by \eqref{psi stG(n)}, it follows that
\begin{equation}
\label{gammap+2 and psi}
\begin{split}
\gamma_{p+2}(W(G_{n-1})) \, B(\St_{G_{n-1}}(2))
&=
\psi(\gamma_{p+1}(G_n)) \, B(\St_{G_{n-1}}(2))
\\
&=
\langle \, \psi(x(p+1)) \, \rangle \, \gamma_{p+3}(W(G_{n-1})) \, B(\St_{G_{n-1}}(2)).
\end{split}
\end{equation}
Now since \Cref{main for quotients} holds for $G_{n-1}$, we have
\[
|G_{n-1}:\gamma_p(G_{n-1})|
=
p^{\,p}
=
|G_2:\gamma_p(G_2)|,
\]
and then \Cref{same index} implies that $\St_{G_{n-1}}(2)\le \gamma_p(G_{n-1})$.
Consequently
\[
B(\St_{G_{n-1}}(2)) \le B(\gamma_p(G_{n-1})) = \gamma_{p^2-p+1}(W(G_{n-1})) \le \gamma_{p+3}(W(G_{n-1})),
\]
since $p$ is odd.
Hence we deduce from \eqref{gammap+2 and psi} that
\[
\gamma_{p+2}(W(G_{n-1}))
=
\psi(\gamma_{p+1}(G_n)) \, \gamma_{p+3}(W(G_{n-1}))
=
\langle \, \psi(x(p+1)) \, \rangle \, \gamma_{p+3}(W(G_{n-1})).
\]
On the other hand, since the image of $x(2)$ generates the quotient $\gamma_2(G)/\gamma_3(G)$,
we have
\[
\gamma_{p+2}(W(G_{n-1}))
=
\langle \boldsymbol{\lambda}_2(x(2)) \rangle \gamma_{p+3}(W(G_{n-1}))
\]
by (iii) of \Cref{LCS of W(G) detail}.
This completes the proof of \eqref{psii-1 modulo} in the case $i=p+2$.

Assume now that the result is true for $i$ and let us prove it for $i+1$.
By the induction hypothesis, \eqref{psii-1 modulo} holds true.
As a consequence,
\begin{equation}
\label{comm+induction hypothesis}
\begin{split}
\psi(\gamma_i(G_n))
=
[\psi(\gamma_{i-1}(G_n)),\psi(G_n)]
&\equiv
[\psi(x(i-1)),\psi(G_n)]
\\
&\equiv
[\boldsymbol{\lambda}_k(x(j)),\psi(G_n)]
\mod
\gamma_{i+2}(W(G_{n-1})).
\end{split}
\end{equation}
Now, since $\psi(\gamma_{i-1}(G_n))\le \gamma_i(W(G_{n-1}))$ by induction, we get in particular
$\psi(x(i-1))\in \gamma_i(W(G_{n-1}))$.
Then, by \eqref{comm psi gn = comm psi aeb}, we obtain
\[
\psi(\gamma_i(G_n))
\equiv
\langle \psi(x(i)) \rangle
\mod
\gamma_{i+2}(W(G_{n-1})),
\]
since $x(i)=[x(i-1),a^{\varepsilon}b]$.
This is one of the congruences in \eqref{psii-1 modulo}, written for $i+1$.
Let us prove the other congruence.
On the one hand, if $1\le k\le p-1$ then by \eqref{comm+induction hypothesis} and (i) of \Cref{comm g with psi gn},
\[
\psi(\gamma_i(G_n))
\equiv
\langle \Delta(\boldsymbol{\lambda}_k(x(j))) \rangle
=
\langle \boldsymbol{\lambda}_{k+1}(x(j)) \rangle
\mod
\gamma_{i+2}(W(G_{n-1})).
\]
On the other hand, if $k=p$ then note that, by \Cref{lambdai(g)} and \eqref{lambdapr mod p},
we have
\[
\boldsymbol{\lambda}_p(x(j))
\equiv
(x(j),\overset{p}{\ldots},x(j))
\mod
\gamma_{i+1}(W(G_{n-1})),
\]
since $\gamma_j(G_{n-1})/\gamma_{j+1}(G_{n-1})$ is of exponent $p$.
Then, by (ii) of \Cref{comm g with psi gn},
\[
\psi(\gamma_i(G_n))
\equiv
\langle ([x(j),a^{\varepsilon}b],1,\overset{p-1}{\ldots},1) \rangle
=
\langle \boldsymbol{\lambda}_1(x(j+1)) \rangle
\mod
\gamma_{i+2}(W(G_{n-1})),
\]
as desired.

(ii)
By (i), we only need to determine the order of $\boldsymbol{\lambda}_k(x(j))$ modulo
$\gamma_{i+1}(W(G_{n-1}))$.
According to (iii) of \Cref{LCS of W(G) detail}, this order coincides with the order of $x(j)$ modulo
$\gamma_{j+1}(G_{n-1})$.
Now the latter is equal to $p$, since \Cref{main for quotients} holds for $G_{n-1}$.
\end{proof}

Thus the image of $\gamma_{i-1}(G_n)$ under $\psi$ covers $\gamma_i(W(G_{n-1}))$ modulo $\gamma_{i+1}(W(G_{n-1}))$
provided that $|\gamma_i(W(G_{n-1})):\gamma_{i+1}(W(G_{n-1}))|=p$, but falls short if this index is $p^2$.
As we next see, in this latter case $\gamma_{i-2}(G_n)$ will do.
This corresponds to step (S3) in \Cref{subsection road map}.

\begin{proposition}
\label{gamma WG vs gamma Gn 2}
Let $G$ be a GGS-group of FG-type, and assume that \Cref{main for quotients} holds for $G_{n-1}$,
where $n\ge 4$.
If $i\ge p+1$ is such that $|\gamma_i(W(G_{n-1})):\gamma_{i+1}(W(G_{n-1}))|=p^2$ then the following hold:
\begin{enumerate}
\item
If we write $i=(j-1)p+k$ with $j\ge 1$ and $1\le k\le p$, then $j$ belongs to a $p^2$-interval
$[\ell,r-1]$ of $G_{n-1}$.
\item
We have
\begin{equation}
\label{gammai WG mod psii-1}
\gamma_i(W(G_{n-1}))
\equiv
\langle \psi(y(i-2)) \rangle
\mod \psi(\gamma_{i-1}(G_n)) \gamma_{i+1}(W(G_{n-1})).
\end{equation}
In particular,
\begin{equation}
\label{gammai WG contained in psi gamma i-2}
\gamma_i(W(G_{n-1})) \le \psi(\gamma_{i-2}(G_n)) \, \gamma_{i+1}(W(G_{n-1})).
\end{equation}
\end{enumerate}
\end{proposition}

\begin{proof}
(i)
This follows immediately from \eqref{index gammai WG = index gammaj G}.

(ii)
Clearly, it suffices to prove the first assertion.
By (i), we have $i\in [(\ell-1)p+1,rp]$.
We use induction on $i$ in this range of values.
Let us first prove the result for $i=(\ell-1)p+1$.
Note that the condition $i\ge p+1$ implies that $\ell>1$.
Since $[\ell,r-1]$ is a $p^2$-interval for $G_{n-1}$, we have $|\gamma_{\ell-1}(G_{n-1}):\gamma_{\ell}(G_{n-1})|=p$.
Also $\gamma_{\ell-1}(G_{n-1})=\langle x(\ell-1) \rangle \gamma_{\ell}(G_{n-1})$, since $G_{n-1}$
satisfies \Cref{main for quotients}.
From (iii) of \Cref{LCS of W(G) detail}, the quotient $\gamma_{i-2}(W(G_{n-1}))/\gamma_{i-1}(W(G_{n-1}))$
can be generated by the image of the element $\mathbf{g}=\boldsymbol{\lambda}_{p-1}(x(\ell-1))$.
Note that, by \Cref{lambdai(g)} and \eqref{binomial p-2},
\[
\mathbf{g}
\equiv
( x(\ell-1)^{-1}, x(\ell-1)^{-2}, \ldots, x(\ell-1)^{-(p-1)}, 1 )
\mod \gamma_i(W(G_{n-1})),
\]
since $\gamma_i(W(G_{n-1}))=B(\gamma_{\ell}(G_{n-1}))$ and $\gamma_{\ell-1}(G_{n-1})/\gamma_{\ell}(G_{n-1})$ is of
exponent $p$.
Then we have
\begin{equation}
\label{first generator}
\begin{split}
[\mathbf{g},\psi(b)]
&\equiv
( [x(\ell-1)^{-1},a^{e_1}], [x(\ell-1)^{-2},a^{e_2}], \ldots, [x(\ell-1)^{-(p-1)},a^{e_{p-1}}], 1 )
\\
&\equiv
( [x(\ell-1),a]^{-e_1}, [x(\ell-1),a]^{-2e_2}, \ldots, [x(\ell-1),a]^{-(p-1)e_{p-1}}, 1 )
\\
&\equiv
( [x(\ell-1),a^{-\delta}], 1, \overset{p-1}{\ldots}, 1)
\mod \gamma_{i+1}(W(G_{n-1})),
\end{split}
\end{equation}
where the last congruence follows from \Cref{reduction to first components}.
Note that $[\mathbf{g},\psi(b)]\in \gamma_i(W(G_{n-1}))$.

On the other hand, since $G_{n-1}$ satisfies \Cref{main for quotients}, we have
\[
\gamma_{\ell}(G_{n-1})=\langle x(\ell),y(\ell) \rangle \, \gamma_{\ell+1}(G_{n-1}).
\]
As a consequence,
\begin{equation}
\label{gammai WG generating tuples}
\begin{split}
\gamma_i(W(G_{n-1}))
&=
\langle (x(\ell),1,\overset{p-1}{\ldots},1),(y(\ell),1,\overset{p-1}{\ldots},1) \rangle
\gamma_{i+1}(W(G_{n-1}))
\\
&=
\langle (y(\ell),1,\overset{p-1}{\ldots},1) \rangle
\psi(\gamma_{i-1}(G_n)) \gamma_{i+1}(W(G_{n-1})),
\end{split}
\end{equation}
where the first equality follows from (iii) of \Cref{LCS of W(G) detail}, and the second, from
\eqref{psii-1 modulo}.
Now, since $x(\ell)=[x(\ell-1),a^{\varepsilon}b]$ and $y(\ell)=[x(\ell-1),b]$, we have
\[
[x(\ell-1),a]
\equiv
y(\ell)^{-\varepsilon^{-1}}
\mod \langle x(\ell) \rangle \gamma_{\ell+1}(G_{n-1}),
\]
where $\varepsilon^{-1}$ denotes the inverse of $\varepsilon$ in $\F_p^{\times}$.
Since $G$ is of FG-type, we have $\delta\ne 0$ in $\F_p$, and then it follows that
\[
\langle [x(\ell-1),a^{-\delta}] \rangle
\equiv
\langle y(\ell) \rangle
\mod \langle x(\ell) \rangle \gamma_{\ell+1}(G_{n-1}).
\]
This, together with \eqref{first generator} and \eqref{gammai WG generating tuples}, implies that
\[
\gamma_i(W(G_{n-1}))
=
\langle [\mathbf{g},\psi(b)] \rangle \,
\psi(\gamma_{i-1}(G_n)) \, \gamma_{i+1}(W(G_{n-1})).
\]
Finally, since $|\gamma_{i-2}(W(G_{n-1})):\gamma_{i-1}(W(G_{n-1}))|=|\gamma_{i-1}(W(G_{n-1})):\gamma_i(W(G_{n-1}))|=p$,
we have
\[
\gamma_{i-2}(W(G_{n-1})) = \langle \psi(x(i-3)) \rangle \, \gamma_{i-1}(W(G_{n-1}))
\]
and
\[
\gamma_{i-1}(W(G_{n-1})) = \psi(\gamma_{i-2}(G_n)) \, \gamma_i(W(G_{n-1})),
\]
by taking into account \Cref{gamma WG vs gamma Gn 1}.
From the latter equality, $[\gamma_{i-1}(W(G_{n-1})),\psi(b)]\le \psi(\gamma_{i-1}(G_n))\gamma_{i+1}(W(G_{n-1}))$, and
then the former equality yields
\[
\langle [\mathbf{g},\psi(b)] \rangle
\equiv
\langle [\psi(x(i-3)),\psi(b)] \rangle
\mod \psi(\gamma_{i-1}(G_n))\gamma_{i+1}(W(G_{n-1})).
\]
Since $[\psi(x(i-3)),\psi(b)]=\psi(y(i-2))$ and $i-2=(\ell-1)p-1$,
this concludes the proof for the base of the induction.

Now assume that the result is true for $i=(j-1)p+k$ and let us prove it for $i+1$.
Then, on the one hand, we have
\begin{align*}
[\gamma_i(W(G_{n-1})),\psi(G_n)]
&\equiv
[\gamma_i(W(G_{n-1})),\psi(a^{\varepsilon}b)]
\\
&\equiv
\langle \psi(y(i-1)) \rangle
\mod
\psi(\gamma_i(G_n)) \, \gamma_{i+2}(W(G_{n-1})),
\end{align*}
by using \eqref{comm gammai psi gn} and the induction hypothesis.
Thus it suffices to show that
\[
[\gamma_i(W(G_{n-1})),\psi(G_n)]
\equiv
\gamma_{i+1}(W(G_{n-1}))
\mod \gamma_{i+2}(W(G_{n-1})).
\]

Recall from \eqref{comm gammai psi gn 1=k=p-1} and \eqref{comm gammai psi gn k=p} that
\begin{align*}
[\gamma_i(W(G_{n-1}))
&,\psi(G_n)] \gamma_{i+2}(W(G_{n-1}))
\\
&=
\begin{cases}
\gamma_{i+1}(W(G_{n-1})),
&
\text{if $1\le k\le p-1$,}
\\[5pt]
B([\gamma_j(G_{n-1}),a^{\varepsilon}b]) \gamma_{i+2}(W(G_{n-1})),
&
\text{if $k=p$.}
\end{cases}
\end{align*}
Thus we are done if $1\le k\le p-1$.
Assume now that $k=p$.
Then since $[\ell,r-1]$ is a $p^2$-interval for $G_{n-1}$ and $\ell\le j\le r-1$, it follows that $j+1$ cannot be the starting index of a $p^2$-interval and then, by \eqref{almost all gens with aepsilonb}, we have
\[
B([\gamma_j(G_{n-1}),a^{\varepsilon}b])
\equiv
B(\gamma_{j+1}(G_{n-1}))
\mod
B(\gamma_{j+2}(G_{n-1})).
\]
Since 
\[
B(\gamma_{j+1}(G_{n-1}))=\gamma_{jp+1}(W(G_{n-1}))=\gamma_{i+1}(W(G_{n-1}))
\]
and $B(\gamma_{j+2}(G_{n-1}))$ is contained in $\gamma_{i+2}(W(G_{n-1}))$, it follows that
\[
B([\gamma_j(G_{n-1}),a^{\varepsilon}b]) \gamma_{i+2}(W(G_{n-1})) = \gamma_{i+1}(W(G_{n-1})),
\]
and the result holds also in this case.
\end{proof}

At this point, we can combine \Cref{gamma WG vs gamma Gn 1} and \Cref{gamma WG vs gamma Gn 2} in order to
prove \Cref{main for quotients}.

\begin{proof}[Proof of \Cref{main for quotients}]
As already mentioned, we use induction on $n$.
The theorem holds for $n=3$ by the results in \Cref{sec:n at most 3}.
Now assume that $n\ge 4$ and that \Cref{main for quotients} is true for all values smaller than $n$.

Since the result holds for $G_{n-1}$, we can apply \Cref{LCS of W(G)} and \Cref{LCS of W(G) detail}
to obtain information about the lower central series of $W(G_{n-1})$.
In particular, we deduce that the nilpotency class of $W(G_{n-1})$ is $p\cdot c(n-1)$ and that
\begin{equation}
\label{index p lower part}
|\gamma_i(W(G_{n-1})):\gamma_{i+1}(W(G_{n-1}))| = p
\end{equation}
for $p(r(n-3)-1)+1\le i\le p\cdot c(n-1)$.

Now we can use \Cref{gamma WG vs gamma Gn 1} to compare the lower central series of $W(G_{n-1})$
with the image of the lower central series of $G_n$ under $\psi$.
From \eqref{index p lower part}, we get
\begin{equation}
\label{relation lower part}
\gamma_i(W(G_{n-1})) = \psi(\gamma_{i-1}(G_n)) \, \gamma_{i+1}(W(G_{n-1}))
\end{equation}
for $p(r(n-3)-1)+1\le i\le p\cdot c(n-1)$.
Since $\gamma_{p\cdot c(n-1)+1}(W(G_{n-1}))=1$, it follows that
\begin{equation}
\label{last terms equal}
\gamma_{p\cdot c(n-1)}(W(G_{n-1})) = \psi(\gamma_{p\cdot c(n-1)-1}(G_n))
\end{equation}
is of order $p$, and consequently the nilpotency class of $G_n$ is
$p\cdot c(n-1)-1$, which by induction coincides with the expected value of $c(n)$, namely
$p^{n-1}-p^{n-3}-\cdots-p-1$.

On the other hand, starting from \eqref{last terms equal} and using repeatedly \eqref{relation lower part},
we obtain that
\begin{equation}
\label{coincidence lower part}
\gamma_i(W(G_{n-1})) = \psi(\gamma_{i-1}(G_n))
\end{equation}
for $r(n-2)+2\le i\le c(n)+1$.
(Note that $p(r(n-3)-1)+1=r(n-2)+2$ from the recursive definition of $r(n)$, since $n\ge 4$.)
By \eqref{index p lower part}, these subgroups form a chain with consecutive indices equal to $p$;
see \Cref{figure p-interval} in \Cref{subsection road map}.
From \eqref{coincidence lower part} and \eqref{index in psii-1}, we also get
$|\psi(\gamma_{i-1}(G_n)):\psi(\gamma_i(G_n))|=p$
for $i=r(n-2)+1$.
On the other hand, from \eqref{psii-1 modulo} we have
\[
\psi(\gamma_{i-1}(G_n)) \equiv \langle \psi(x(i-1)) \rangle \mod \gamma_{i+1}(W(G_{n-1}))
\]
for $i\ge r(n-2)+1$, and since $\gamma_{i+1}(W(G_{n-1}))=\psi(\gamma_i(G_n))$ by \eqref{coincidence lower part},
it follows that $\psi(\gamma_{i-1}(G_n))=\langle \psi(x(i-1)) \rangle \psi(\gamma_i(G_n))$.
Thus we conclude that
\[
|\gamma_i(G_n):\gamma_{i+1}(G_n)|=p
\]
and
\[
\gamma_i(G_n) = \langle x(i) \rangle \gamma_{i+1}(G_n)
\]
for $r(n-2)\le i\le c(n)$.
This proves in particular that $[r(n-2),c(n)]$ is a $p$-interval for $G_n$.

Now we examine the interval $[\ell(n-2),r(n-2)-1]$.
By \eqref{index gammai WG = index gammaj G} and the induction hypothesis, we have
\[
|\gamma_i(W(G_{n-1})):\gamma_{i+1}(W(G_{n-1}))| = p^2
\]
for $p(\ell(n-3)-1)+1\le i\le p(r(n-3)-1)$ or, what is the same, for
$\ell(n-2)+2\le i\le r(n-2)+1$.
Thus from \eqref{gammai WG contained in psi gamma i-2} we get
\begin{equation}
\label{inclusion in gammai-2 1}
\gamma_i(W(G_{n-1})) \le \psi(\gamma_{i-2}(G_n)) \, \gamma_{i+1}(W(G_{n-1}))
\end{equation}
for $i$ in this range of values.
On the other hand, from \eqref{coincidence lower part},
\[
\gamma_{r(n-2)+2}(W(G_{n-1}))
\le
\psi(\gamma_{r(n-2)}(G_n)).
\]
Combining this with \eqref{inclusion in gammai-2 1} and using backward induction on $i$,
it follows that
\begin{equation}
\label{inclusion in gammai-2 2}
\gamma_i(W(G_{n-1})) \le \psi(\gamma_{i-2}(G_n))
\end{equation}
for $\ell(n-2)+2\le i\le r(n-2)+1$.
This inclusion and \eqref{index in psii-1} imply that
\[
|\psi(\gamma_{i-2}(G_n)):\gamma_i(W(G_{n-1}))|
=
|\psi(\gamma_{i-1}(G_n)):\gamma_{i+1}(W(G_{n-1}))|
=
p
\]
and, since $\psi(\gamma_{i-1}(G_n))\le \gamma_i(W(G_{n-1}))$ and
$|\gamma_i(W(G_{n-1})):\gamma_{i+1}(W(G_{n-1}))|=p^2$, it follows that
\[
|\psi(\gamma_{i-2}(G_n)):\psi(\gamma_{i-1}(G_n))| = p^2.
\]
Also, by \eqref{gammai WG mod psii-1},
\[
\gamma_i(W(G_{n-1}))
\equiv
\langle \psi(y(i-2)) \rangle
\mod \psi(\gamma_{i-1}(G_n)).
\]
Since
\[
\psi(\gamma_{i-2}(G_n))
\equiv
\langle \psi(x(i-2)) \rangle
\mod
\gamma_i(W(G_{n-1}))
\]
by \eqref{psii-1 modulo}, it follows that
\[
\psi(\gamma_{i-2}(G_n))
=
\langle \psi(x(i-2)), \psi(y(i-2)) \rangle \, \psi(\gamma_{i-1}(G_n)).
\]
The position of the relevant subgroups in this case is depicted in \Cref{figure p2-interval} of
\Cref{subsection road map}.
We conclude that
\[
|\gamma_i(G_n):\gamma_{i+1}(G_n)|=p^2
\quad
\text{for $\ell(n-2)\le i\le r(n-2)-1$,}
\]
and that
\[
\gamma_i(G_n)
=
\langle x(i), y(i) \rangle \, \gamma_{i+1}(G_n).
\]
In particular, $[\ell(n-2),r(n-2)-1]$ is a $p^2$-interval for $G_n$.

After having proved \Cref{main for quotients} for the intervals $[r(n-2),c(n)]$ and $[\ell(n-2),r(n-2)-1]$,
we could use a similar argument to deal alternatively with the rest of the intervals that appear in
Theorem A by backwards induction on the intervals.
However, we can complete the proof of \Cref{main for quotients} in a cleaner way as follows,
by using only the standing induction on $n$ and relying on \Cref{same index}.

By what we have proved so far and taking into account \eqref{order Gn}, we have
\begin{align*}
\label{index gammal(n-2) Gn}
|G_n:\gamma_{\ell(n-2)}(G_n)|
=
\frac{p^{\,p^{n-1}+1}}{p^{\,2(r(n-2)-\ell(n-2))} \cdot p^{\,c(n)-r(n-2)+1}}
=
p^{\,p^{n-1}+2\ell(n-2)-r(n-2)-c(n)}.
\end{align*}
On the other hand, since \Cref{main for quotients} holds for $G_{n-1}$ and $r(n-3)\le \ell(n-2)\le c(n-1)+1$
(note that $n\ge 4$ for the last inequality), we have
\[
|G_{n-1}:\gamma_{\ell(n-2)}(G_{n-1})|
=
\frac{p^{\,p^{n-2}+1}}{p^{\,c(n-1)-\ell(n-2)+1}}
=
p^{\,p^{n-2}+\ell(n-2)-c(n-1)}.
\]
Since $n\ge 4$, we have $r(n-2)-\ell(n-2)=p^{n-3}$, and on the other hand, $c(n)=c(n-1)+p^{n-1}-p^{n-2}-p^{n-3}$.
It follows that
\[
|G_n:\gamma_{\ell(n-2)}(G_n)|=|G_{n-1}:\gamma_{\ell(n-2)}(G_{n-1})|.
\]
By \Cref{same index}, this means that $\St_{G_n}(n-1)\le \gamma_{\ell(n-2)}(G_n)$.
As a consequence, $|\gamma_i(G_n):\gamma_{i+1}(G_n)|=|\gamma_i(G_{n-1}):\gamma_{i+1}(G_{n-1})|$
for $1\le i\le \ell(n-2)-1$, and if $X\subseteq G$ is such that $\gamma_i(G_{n-1})=\langle X \rangle \, \gamma_{i+1}(G_{n-1})$,
then also $\gamma_i(G_n)=\langle X \rangle \, \gamma_{i+1}(G_n)$.
Since \Cref{main for quotients} holds for $G_{n-1}$, this completes the proof of the result for $G_n$.
\end{proof}

Now we can easily prove the main theorem of the paper.

\begin{proof}[Proof of Theorem A]
Let $i$ be an arbitrary positive integer.
As we know, the quotients of consecutive terms of the lower central series of $G$ have all exponent $p$.
Hence $G/\gamma_{i+1}(G)$ is a $2$-generator nilpotent group of exponent at most $p^{i+1}$,
and so is finite.
(Alternatively, note that $\gamma_{i+1}(G)$ is non-trivial by using \Cref{main for quotients}, or because
branch groups do not satisfy any laws, and take into account that $G$ is just infinite.) 
Now by Theorem A of \cite{Gustavo-Alejandra-Jone}, all GGS-groups with a non-constant defining vector have the congruence
subgroup property.
It follows that $\St_G(k)\le \gamma_{i+1}(G)$ for some $k$.
Consequently $|\gamma_i(G):\gamma_{i+1}(G)|=|\gamma_i(G_n):\gamma_{i+1}(G_n)|$ for every $n\ge k$.
Since the intervals of the form $[\ell(m),r(m))$ and $[r(m),\ell(m+1))$ form a disjoint cover of
$\N$, there is a unique $m\ge 0$ such that $i$ belongs to one of these intervals.
By \Cref{main for quotients}, $|\gamma_i(G):\gamma_{i+1}(G)|=p$ or $p^2$, and $\gamma_i(G)/\gamma_{i+1}(G)$
can be generated either by $x(i)$ or by $x(i)$ and $y(i)$, depending on the interval
in which $i$ lies.

On the other hand, let $X_i\subseteq G$ be the set of generators that we have determined for the quotient
$\gamma_i(G)/\gamma_{i+1}(G)$.
By \Cref{main for quotients}, we have
\[
\gamma_i(G) \St_G(n) = \langle X_j \mid i\le j\le c(n) \rangle \St_G(n) 
\]
for every $n\in\N$.
Set $N_i=\langle X_i \rangle^G$, which is a subgroup of $\gamma_i(G)$.
Then, by definition of the sets $X_j$, we have $X_j\subseteq N_i$ for all $j\ge i$.
Consequently
\begin{equation}
\label{gamma_i and N_i equal mod stG(n)}
\gamma_i(G) \St_G(n) = N_i \St_G(n).
\end{equation}
Now, since $N_i$ is a non-trivial normal subgroup of $G$ and $G$ is just infinite and satisfies the congruence
subgroup property, we have $\St_G(n)\le N_i$ for some $n\in\N$.
We conclude from \eqref{gamma_i and N_i equal mod stG(n)} that $\gamma_i(G)=N_i$, as desired.
\end{proof}

Finally, the proof of Corollary B is an immediate consequence of the following standard result in
profinite group theory.
We include its proof for the convenience of the reader.
Take into account that, by definition, we consider $\gamma_i(\hat G)$ to be the \emph{closed} subgroup
generated by all $i$-fold commutators of elements of $\hat G$.

\begin{proposition}
\label{index of gammai equal}
Let $G$ be a residually finite group and let $\hat G$ be its profinite completion.
If $|G:\gamma_i(G)|$ is finite for some $i\ge 1$ then $|\hat G:\gamma_i(\hat G)|=|G:\gamma_i(G)|$.
\end{proposition}

\begin{proof}
The group $\hat G$ is the inverse limit of the finite quotients of $G$.
Since $|G:\gamma_i(G)|$ is finite, it follows that $G/\gamma_i(G)$ is isomorphic to a quotient
of $\hat G$, and consequently $|\hat G:\gamma_i(\hat G)|\ge |G:\gamma_i(G)|$.
On the other hand, if we consider $G$ as a subgroup of $\hat G$ then the topological closure
$\overline{\gamma_i(G)}$ of $\gamma_i(G)$ in $\hat G$ satisfies that
$|\hat G:\overline{\gamma_i(G)}|=|G:\gamma_i(G)|$, by Proposition 3.2.2(d) of \cite{Ribes-Zalesskii}.
Since $\overline{\gamma_i(G)}\le \gamma_i(\hat G)$, the result follows.
\end{proof}

\section{On the Lie algebra associated to the lower central series of a GGS-group of FG-type}
\label{sec:lie algebra}

This section is devoted to the proof of Theorem C, where we describe the structure of the Lie algebra $L(G)$ associated to the lower central series of a GGS-group $G$ of FG-type with $\varepsilon=1$.
As mentioned in the introduction, this does not mean a restriction, since every non-periodic GGS-group can be defined via a vector $\mathbf{e}$ satisfying $\varepsilon(\mathbf{e})=1$.

\vspace{8pt}

This section is organised as follows.
All brackets with $ab$ are clear from the definition of the elements $x(i)$ and $y(i)$, except for those of the form
$[y(r(m)-1),ab]_L$, which are dealt with in \Cref{brackets at degree p} and \Cref{comms y(r(m)-1)}.
As for the brackets with $b$, we first consider $[x(p),b]_L$ and $[y(p),b]_L$ in \Cref{brackets at degree p}.
In the rest of the cases, most of the brackets are easily settled by the position of $\psi(x(i))$ and $\psi(y(i))$ in the lower central series of $W(G)$;
see \Cref{comm y(i) and b}, \Cref{comm x(i) and b p-interval}, and \Cref{comm x(i) and b p2-interval non-exceptional}.
The major difficulty lies in the brackets of the form $[x(i),b]_L$, where $i$ lies in a $p^2$-interval and $i\equiv -2,-1 \mod p$.
These are determined in \Cref{brackets x(i) with b}.

\begin{remark}
As we observed in the proof of Theorem A, due to the congruence subgroup property, every term
$\gamma_i(G)$ of the lower central series contains some level stabiliser $\St_G(n)$.
Also, recall that for $n\ge 3$, the image of $\St_G(n)$ under $\psi$ is $B(\St_G(n-1))$.
On the other hand, every subgroup of the form $\gamma_i(W(G))$ contains a subgroup of the
form $B(\gamma_j(G))$ for some $j$, and consequently also of the form $B(\St_G(m))$.
It readily follows that any inclusion or equality between subgroups of the sequences
$\{\psi(\gamma_i(G_n))\}_{i\ge 1}$ and $\{\gamma_i(W(G_{n-1}))\}_{i\ge 1}$ that holds for all big enough values
of $n$ will also hold when we replace $G_n$ and $G_{n-1}$ with $G$.
Thus, from the results in \Cref{sec:proof theorem A} we deduce that
\[
\psi(\gamma_i(G))\le \gamma_{i+1}(W(G)) \le \psi(\gamma_{i-1}(G))
\]
for every $i\ge p+1$.
Also, for every $m\ge 2$ we have
\[
\psi(\gamma_i(G))=\gamma_{i+1}(W(G))
\quad
\text{for $i\in [r(m-1)+1,\ell(m)]$,}
\]
while
\[
\psi(\gamma_i(G))\lneq \gamma_{i+1}(W(G)) \lneq \psi(\gamma_{i-1}(G))
\quad
\text{for $i\in [\ell(m)+1,r(m)]$.}
\]
Similarly, the indices between subgroups in $G_n$ or in $W(G_n)$ for $n$ large enough (once they stabilise) will be the same in $G$ or in $W(G)$.
\end{remark}

Let us start with a preliminary result, which determines $\psi(x(i))$ modulo $\gamma_{i+2}(W(G))$.
Given an integer $i=(j-1)p+k\ge p+1$ with $0\le k\le p-1$, part (i) of \Cref{gamma WG vs gamma Gn 1}
shows that $\psi(x(i))$ is a non-trivial power of $\Delta^k(x(j),1,\ldots,1)$ modulo
$\gamma_{i+2}(W(G))$. 
We make this power explicit in the next lemma.

\begin{lemma}
\label{psi(x(i)) modulo}
Let $G$ be a GGS-group of FG-type with $\varepsilon=1$, and let $i\ge p+1$.
Then, if we write $i=(j-1)p+k$ with $0\le k\le p-1$, we have
\[
\psi(x(i)) \equiv \Delta^k (x(j),1,\ldots,1)^{-\delta} \mod \gamma_{i+2}(W(G)).
\]
\end{lemma}

\begin{proof}
We use induction on $i\ge p+1$.
We first assume that $i=p+1$.
In this case, we rely on the calculations in the proof of \Cref{position commutators}.
We have $x(p-1)\equiv [b,a,\overset{p-2}{\ldots},a] \mod \gamma_p(G)$ and, from
\eqref{comm bap-1},
\[
\psi([x(p-1),a])
\equiv
(ba[b,a],ba[b,a]^2,\ldots,ba[b,a]^{p-1},ba)
\mod{\gamma_{p+2}(W(G))},
\]
since $\psi([\St_G(2),G])=\psi(\gamma_{p+1}(G))=[B(G'),W(G)]=\gamma_{p+2}(W(G))$.
By (i) of \Cref{reduction to first components},
\[
([b,a],[b,a]^2,\ldots,[b,a]^{p-1},1)
\equiv
([b,a]^{p(p-1)/2},1,\ldots,1)
\equiv
(1,\ldots,1)
\mod \gamma_{p+2}(W(G)),
\]
so that
\[
\psi([x(p-1),a])
\equiv
(ba,\ldots,ba)
\mod{\gamma_{p+2}(W(G))}.
\]
On the other hand, again from the proof of \Cref{position commutators},
\begin{equation}
\label{psi(y(p))}
\begin{split}
\psi([x(p-1),b])
&\equiv
([b,a]^{-2e_1},[b,a]^{-3e_2},\ldots,[b,a]^{-(p-1)e_{p-2}},1,[b,a]^{1-\delta})
\\
&\equiv
([b,a]^{-2\delta},1,\ldots,1)
\mod{\gamma_{p+2}(W(G))}.
\end{split}
\end{equation}

Then
\begin{equation}
\label{comm psi(x(p-1)) and a and a}
\psi([x(p-1),a,a])
\equiv
(1,\ldots,1)
\mod{\gamma_{p+3}(W(G))},
\end{equation}
and
\begin{equation}
\label{comm psi(x(p-1)) and a and b}
\begin{split}
\psi([x(p-1),a,b])
&\equiv
([ba,a^{e_1}],\ldots,[ba,a^{e_{p-1}}],[ba,b])
\\
&\equiv
([b,a]^{e_1},\ldots,[b,a]^{e_{p-1}},[a,b])
\\
&
\equiv
([b,a]^{-\delta},[b,a]^{\delta},1,\ldots,1)
\mod{\gamma_{p+3}(W(G))},
\end{split}
\end{equation}
where the last congruence follows from (ii) of \Cref{reduction to first components}.
On the other hand,
\begin{equation}
\label{comm psi(y(p)) and a}
\psi([x(p-1),b,a])
\equiv
([b,a]^{2\delta},[b,a]^{-2\delta},1,\ldots,1)
\mod{\gamma_{p+3}(W(G))}
\end{equation}
and, obviously, $\psi([x(p-1),b,b])\in \gamma_{p+3}(W(G))$.
We conclude that
\begin{equation}
\label{psi(x(p+1))}
\begin{split}
\psi(x(p+1))
&\equiv
\psi([x(p-1),a,a])
\psi([x(p-1),a,b])
\psi([x(p-1),b,a])
\psi([x(p-1),b,b])
\\
&\equiv
([b,a]^{\delta},[b,a]^{-\delta},1,\ldots,1)
\\
&\equiv
\Delta (x(2),1,\ldots,1)^{-\delta}
\mod \gamma_{p+3}(W(G)).
\end{split}
\end{equation}
This proves the result for $i=p+1$.

Now we assume that the result holds for $i=(j-1)p+k$ with $0\le k\le p-1$, and prove it for $i+1$.
If $k\le p-2$ then
\[
B(\gamma_{j+1}(G)) = \gamma_{jp+1}(W(G)) \le \gamma_{i+3}(W(G))
\]
and consequently taking the commutator of $\psi(x(i))\in B(\gamma_j(G))$ with
$\psi(b)\in B(G)$ is trivial modulo $\gamma_{i+3}(W(G))$.
It follows that
\[
\psi(x(i+1)) = \psi([x(i),ab]) \equiv \psi([x(i),a]) \equiv \Delta^{k+1}(x(j),1\ldots,1)^{-\delta}
\mod \gamma_{i+3}(W(G)),
\]
since taking the commutator with $\psi(a)=\sigma$ is the same as applying the operator $\Delta$.
Thus the result also holds for $i+1$ in this case.
Finally, if $k=p-1$ then we have
\[
\psi(x(i)) \equiv (x(j),\ldots,x(j))^{-\delta} \mod \gamma_{i+2}(W(G)),
\]
and
\begin{align*}
\psi(x(i+1))
&=
\psi([x(i),ab])
\equiv
\psi([x(i),b])
\\
&\equiv
([x(j),a^{e_1}],\ldots,[x(j),a^{e_{p-1}}],[x(j),b])^{-\delta}
\\
&\equiv
([x(j),ab],1,\ldots,1)^{-\delta}
\equiv
(x(j+1),1,\ldots,1)^{-\delta}
\mod \gamma_{i+3}(W(G)),
\end{align*}
where the second-to-last congruence follows once again from (i) of \Cref{reduction to first components},
taking into account that $\varepsilon=1$.
This completes the proof.
\end{proof}

Now we proceed to the proof of Theorem C.
By definition, we have $[x(i),ab]_L=x(i+1)$ for every $i\ge 1$, and
$[y(i),ab]_L=y(i+1)$ for $\ell(m)\le i\le r(m)-2$, for every $m\ge 2$.
We split the determination of the rest of Lie brackets into several propositions.
We start with the ones involving $x(p)$ and $y(p)$.

\begin{proposition}
\label{brackets at degree p}
Let $G$ be a GGS-group of FG-type with $\varepsilon=1$.
Then the following hold:
\begin{enumerate}
\item
$[y(p),b]_L=0$.
\item
$[y(p),ab]_L=2x(p+1)$.
\item
$[x(p),b]_L=-x(p+1)$.
\end{enumerate}
\end{proposition}

\begin{proof}
(i)
Since $y(p)=[x(p-1),b]$, from \eqref{psi(y(p))} we get
\[
\psi([y(p),b]) \in [B(G'),B(G)] = B(\gamma_3(G)) = \gamma_{2p+1}(W(G)) \le \psi(\gamma_{2p-1}(G)).
\]
Hence $[y(p),b]\in \gamma_{2p-1}(G)\leq \gamma_{p+2}(G)$, since $p\geq 3$, and therefore $[y(p),b]_L=0$.

(ii)
We have $[y(p),ab]_L=[y(p),a]_L+[y(p),b]_L=[y(p),a]_L$.
Then the result follows from \eqref{comm psi(y(p)) and a} and \eqref{psi(x(p+1))},
taking into account that $\gamma_{p+3}(W(G))=\psi(\gamma_{p+2}(G))$.

(iii)
We have
\[
[x(p),b]_L = [x(p-1),a,b]_L + [x(p-1),b,b]_L = [x(p-1),a,b]_L + [y(p),b]_L
= [x(p-1),a,b]_L,
\]
by (i).
Now the result immediately follows from \eqref{comm psi(x(p-1)) and a and b} and \eqref{psi(x(p+1))}.
\end{proof}

Next we determine the remaining values of the form $[y(i),ab]_L$, i.e.\ those with $i=r(m)-1$ and
$m\ge 2$.
At the same time, we determine $[y(r(m)-1),b]_L$.
Note that, according to \eqref{gammai WG mod psii-1} in \Cref{gamma WG vs gamma Gn 2},
we have
\begin{equation}
\label{where is psi(y(i))}
\psi(y(i)) \in \gamma_{i+2}(W(G))
\end{equation}
for every $m\geq 2$ and every $i\in [\ell(m),r(m)-1]$.

\begin{proposition}
\label{comms y(r(m)-1)}
Let $G$ be a GGS-group of FG-type with $\varepsilon=1$.
Then $[y(r(m)-1),ab]_L=[y(r(m)-1),b]_L=0$ for every $m\ge 2$.
\end{proposition}

\begin{proof}
By \eqref{where is psi(y(i))}, we have $\psi(y(r(m)-1))\in \gamma_{r(m)+1}(W(G))$.
Let $g=ab$ or $b$.
We have
\[
\psi([y(r(m)-1),g]) \in \gamma_{r(m)+2}(W(G)) = \psi(\gamma_{r(m)+1}(G)),
\]
where the last equality follows from $r(m)$ being the starting point of a $p$-interval with $m\ge 2$.
Hence $[y(r(m)-1),g] \in \gamma_{r(m)+1}(G)$ and $[y(r(m)-1),g] _L=0$.
\end{proof}

In our following result, we complete the determination of all Lie brackets involving elements of the
form $y(i)$.

\begin{proposition}
\label{comm y(i) and b}
Let $G$ be a GGS-group of FG-type with $\varepsilon=1$.
Then, for every $m\ge 2$ and $\ell(m)\le i\le r(m)-2$, we have
\[
[y(i),b]_L
=
\begin{cases}
y(i+1),
&
\text{if $i\equiv -2 \hspace{-6pt} \mod p$,}
\\
0,
&
\text{otherwise.}
\end{cases}
\]
\end{proposition}

\begin{proof}
From the recurrence relations defining $\ell(m)$ and $r(m)$ for $m\ge 2$, we can write
$i=(j-1)p+k$ with $\ell(m-1)\le j\le r(m-1)-1$ and $-1\le k\le p-2$.
Since $\psi(y(i)) \in \gamma_{i+2}(W(G)) \le B(\gamma_j(G))$, we get
\[
\psi([y(i),b])\in [B(\gamma_j(G)),B(G)] = B(\gamma_{j+1}(G)) = \gamma_{jp+1}(W(G)).
\]
If $k\ne p-2$ or, equivalently, if $i\not\equiv -2 \mod p$, we have
$jp+1\ge i+4$, and consequently
\[
\psi([y(i),b])\in \gamma_{i+4}(W(G)) \le \psi(\gamma_{i+2}(G)).
\]
This implies that $[y(i),b]_L=0$ in this case.

If $i\equiv -2 \mod p$ then $\psi(y(i))\in \gamma_{jp}(W(G))$ is a constant vector with entries in
$\gamma_j(G)$ modulo $B(\gamma_{j+1}(G))$.
It follows that
\[
\psi([y(i),a]) = [\psi(y(i)),\sigma] \in [B(\gamma_{j+1}(G)),W(G)] = \gamma_{jp+2}(W(G))
\le \psi(\gamma_{jp}(G)) = \psi(\gamma_{i+2}(G)).
\]
Hence $[y(i),a]_L=0$ and then
\[
[y(i),b]_L = [y(i),a]_L+[y(i),b]_L = [y(i),ab]_L = y(i+1),
\]
taking into account that the condition $i\le r(m)-2$ implies that $y(i+1)$ is defined.
\end{proof}

Thus we only need to calculate the Lie brackets $[x(i),b]_L$.
We first deal with the case where $i$ lies in a $p$-interval $[r(m),\ell(m+1)-1]$, with $m\ge 0$.

\begin{proposition}
\label{comm x(i) and b p-interval}
Let $G$ be a GGS-group of FG-type with $\varepsilon=1$.
If $r(m)\le i\le \ell(m+1)-1$ with $m\ge 0$, then
\[
[x(i),b]_L
=
\begin{cases}
x(i+1),
&
\text{if $i\equiv -1 \hspace{-8pt} \mod p$, $i\ne p-1$,}
\\
y(i+1),
&
\text{if $i=\ell(m+1)-1$,}
\\
0,
&
\text{otherwise.}
\end{cases}
\]
\end{proposition}

\begin{proof}
If $i=\ell(m+1)-1$ then we get the value $y(i+1)$ by definition, so we assume that
$r(m)\le i\le \ell(m+1)-2$.
If $m=0$ then $2\le i\le p-2$ and so
\[
[x(i),b] \in \St_G(1)' = \St_G(2) \le \gamma_p(G) \le \gamma_{i+2}(G).
\]
Hence $[x(i),b]_L=0$ in this case.

Now consider the case $m\ge 1$.
Write $i=(j-1)p+k$ with $0\le k\le p-1$.
We have
\[
\psi(x(i)) \in \psi(\gamma_i(G)) \le \gamma_{i+1}(W(G)) \le B(\gamma_j(G)).
\]
If $k\le p-2$, i.e.\ if $i\not\equiv -1 \mod p$ then
\[
\psi([x(i),b]) \in [B(\gamma_j(G)),B(G)] = B(\gamma_{j+1}(G)) = \gamma_{jp+1}(W(G))
\le \gamma_{i+3}(W(G)) = \psi(\gamma_{i+2}(G)),
\]
and $[x(i),b]_L=0$.

On the other hand, if $k=p-1$ then $\psi(x(i))\in \gamma_{jp}(W(G))$ is a constant vector modulo
$B(\gamma_{j+1}(G))$.
As in the proof of \Cref{comm y(i) and b}, this implies that $[x(i),a]_L=0$.
Thus
\[
[x(i),b]_L = [x(i),a]_L + [x(i),b]_L = [x(i),ab]_L = x(i+1)
\]
in this case.
\end{proof}

We are only left with the case of $[x(i),b]$ with $i$ in a $p^2$-interval $[\ell(m),r(m)-1]$ with $m\ge 2$.
The following subcase is particularly easy.
It can be proved with arguments similar to those given in some cases of \Cref{comm y(i) and b} and
\Cref{comm x(i) and b p-interval}, so we omit the proof.

\begin{proposition}
\label{comm x(i) and b p2-interval non-exceptional}
Let $G$ be a GGS-group of FG-type with $\varepsilon=1$.
If $\ell(m)\le i\le r(m)-1$ with $m\ge 2$, and $i\not\equiv -2,-1 \mod p$ then
$[x(i),b]_L=0$.
\end{proposition}

The remaining case, where $i\in [\ell(m),r(m)-1]$ with $m\ge 2$ and $i\equiv -2,-1\mod p$ is quite intricate and requires a very careful analysis.
We need several lemmas before proceeding.
The first lemma will help us connect the value of $\psi([g,ab,\overset{p}{\ldots},ab])$ with
$\psi(g)$ under some conditions.

\begin{lemma}
\label{p comms with ab}
Let $G$ be a GGS-group of FG-type with $\varepsilon=1$, and let $L\le K$ be normal subgroups
of $\psi(G)$ such that:
\begin{enumerate}
\item
$K/L$ is an elementary abelian $p$-group, i.e.\ $K'K^p\le L$.
\item
$[K,B(G')]\le L$.
\end{enumerate}
Then for every tuple $(g_1,\ldots,g_p)\in K$ we have
\[
[(g_1,\ldots,g_p),\psi(ab),\overset{p}{\ldots},\psi(ab)]
\equiv
([g_1,ab],\ldots,[g_p,ab])
\mod L.
\]
\end{lemma}

\begin{proof}
Set $V=K/L$, and note that $V$ can be seen as an $\F_p$-vector space by condition (i).
Conjugation by $\psi(ab)$ induces an automorphism $\varphi$ of $V$, and then we have
\[
[(g_1,\ldots,g_p),\psi(ab),\overset{p}{\ldots},\psi(ab)]
\equiv
(g_1,\ldots,g_p)^{(\varphi-1)^p}
\mod L.
\]
Since $\varphi-1\in\End(V)$ and $\End(V)$ is a ring of characteristic $p$, it follows that
$(\varphi-1)^p=\varphi^p-1$.
On the other hand, $\varphi^p$ corresponds to conjugation by $\psi((ab)^p)$, and so
\begin{equation}
\label{com with psi(ab) p times}
[(g_1,\ldots,g_p),\psi(ab),\overset{p}{\ldots},\psi(ab)]
\equiv
[(g_1,\ldots,g_p),\psi((ab)^p)]
\mod L.
\end{equation}

Now observe that
\[
(ab)^p = a^p b^{a^{p-1}} \cdots b^a b = b^{a^{p-1}} \cdots b^a b,
\]
and consequently
\[
\psi((ab)^p) = \psi(b^{a^{p-1}}) \cdots \psi(b^a) \psi(b)
\equiv (ab,\ldots,ab) \mod B(G').
\]
Hence the result follows from \eqref{com with psi(ab) p times} and condition (ii).
\end{proof}

The next results will be useful when we want to ensure the conditions needed to apply the previous lemma.

\begin{lemma}
\label{exponent p}
Let $G$ be a GGS-group of FG-type.
Then, for every $m\ge 1$, we have $\exp \gamma_{\ell(m)-1}(G)/\gamma_{r(m)}(G)=p$.
\end{lemma}

\begin{proof}
We argue by induction on $m$.
If $m=1$ we need to prove that $\gamma_{p-1}(G)^p\le \gamma_{p+1}(G)$.
We actually have the stronger condition $\St_G(1)^p\le \gamma_{p+1}(G)$.
Indeed, $\St_G(1)/\gamma_{p+1}(G)$ is a finite $p$-group of class $2$ with $p$ odd, so a regular $p$-group, that can be generated by the elements $b^{a^i}$ of order $p$, for $0\le i\le p-1$.
(See \cite{Gustavo-Brasilia} or \cite[Chapter 3]{Suzuki} for the relevant properties of regular finite $p$-groups.)

Now assume that $m\ge 2$ and that the result holds for $m-1$.
Since
\[
\ell(m)-1=p(\ell(m-1)-1)-2\ge p(\ell(m-1)-2),
\]
we have
\[
\psi(\gamma_{\ell(m)-1}(G)) \le \psi(\gamma_{p(\ell(m-1)-2)}(G))
\le \gamma_{p(\ell(m-1)-2)+1}(W(G))
= B(\gamma_{\ell(m-1)-1}(G)).
\]
From this and the induction hypothesis, we get
\begin{align*}
\psi(\gamma_{\ell(m)-1}(G)^p)
&\le
B(\gamma_{\ell(m-1)-1}(G)^p) \le B(\gamma_{r(m-1)}(G))
\\
&=
\gamma_{p(r(m-1)-1)+1}(W(G)) = \gamma_{r(m)+2}(W(G)) \le \psi(\gamma_{r(m)}(G)),
\end{align*}
and consequently
\[
\gamma_{\ell(m)-1}(G)^p \le \gamma_{r(m)}(G),
\]
as desired.
\end{proof}

\begin{lemma}
\label{good i for lemma}
Let $G$ be a GGS-group of FG-type, and let $i\ge p+1$.
Then $\gamma_i(G)'\le \gamma_{i+p+1}(G)$, and if $i\not\equiv -1\mod p$ then also
$[\psi(\gamma_i(G)),B(G')]\le \psi(\gamma_{i+p+1}(G))$.
\end{lemma}

\begin{proof}
The first inclusion is obvious, since $\gamma_i(G)'\le \gamma_{2i}(G)$.
For the second, observe that, if $i\not\equiv -1 \mod p$ and we write $i=(j-1)p+k$ with $0\le k\le p-2$, then
\[
\psi(\gamma_i(G)) \le \gamma_{i+1}(W(G)) \le \gamma_{(j-1)p+1}(W(G)) = B(\gamma_j(G))
\]
and consequently
\[
[\psi(\gamma_i(G)),B(G')] \le B(\gamma_{j+2}(G)) = \gamma_{(j+1)p+1}(W(G))
\le \psi(\gamma_{(j+1)p-1}(G)) \le \psi(\gamma_{i+p+1}(G)),
\]
by using that $k\le p-2$.
\end{proof}

\begin{remark}
\label{i not -1 is good}
From the last two lemmas, we conclude that for $i\in [\ell(m)-1,r(m)-p-1]$ with $m\ge 2$,
and $i\not\equiv -1\mod p$, \Cref{p comms with ab} can be applied with $K=\psi(\gamma_i(G))$
and $L=\psi(\gamma_{i+p+1}(G))$.
In some situations where we have $i\equiv -1\pmod p$, the next lemma can be useful to overcome this restriction and make sure that \Cref{p comms with ab} still works.
\end{remark}

\begin{lemma}
\label{comm with B(G')}
For every $j\ge 2$, we have
\[
B(\gamma_j(G))' \le [B(\gamma_j(G)),B(G')] \le \psi(\gamma_{jp}(G)).
\]
\end{lemma}

\begin{proof}
We have
\begin{align*}
[B(\gamma_j(G)),B(G')]
&\le
B(\gamma_{j+2}(G))
\\
&=
\gamma_{(j+1)p+1}(W(G))
\le
\psi(\gamma_{(j+1)p-1}(G))
\le
\psi(\gamma_{jp}(G)).
\end{align*}
\end{proof}

We can finally determine the remaining Lie brackets of the form $[x(i),b]_L$,
i.e.\ those where $i$ lies in a $p^2$-interval $[\ell(m),r(m)-1]$ with $m\ge 2$ and $i\equiv -2,-1 \mod p$.
Let us give a quick idea of the strategy of the proof.
We rely on \Cref{p comms with ab}, which can be applied in this setting, as explained in \Cref{i not -1 is good}.
Using this tool and the map $\psi$, we can transport information from a degree of the form $\ell(m-1)+\lambda-1$ to degree $\ell(m)+\lambda p-1$.
Following the notation of items (vii) and (viii) of \Cref{theoremC}, roughly speaking, the brackets at level $m$ are thus inherited from those at level $m-1$ with $s$ replaced by $s/p$ or $(s-1)/p$, respectively.
This is why the coefficient $\eta$ appearing in the expression for the bracket $[x(i),b]_L$ depends on the base-$p$ expansion of $s$.

\begin{proposition}
\label{brackets x(i) with b}
Let $G$ be a GGS-group of FG-type with $\varepsilon=1$, and let $i\in [\ell(m),r(m)-1]$ with $m\ge 2$.
Then the following hold:
\begin{enumerate}
\item
If $i=r(m)-1$ then $[x(i),b]_L=0$.
\item
Suppose $i\le r(m)-2$, with $i\equiv -2 \mod p$, and write
$i=\ell(m)-1+s$ with $s\equiv 0\mod p$.
Then
\[
[x(i),b]_L
=
\eta y(i+1),
\]
where $\eta=(-1)^t$ if $s=p^{i_1}+\cdots+p^{i_t}$, with $i_1<\cdots<i_t$, is a sum of $t$ different powers of $p$, and $\eta=0$ otherwise.
\item
Suppose $i\le r(m)-2$, with $i\equiv -1 \mod p$, and write
$i=\ell(m)-1+s$ with $s\equiv 1\mod p$.
Then
\[
[x(i),b]_L
=
x(i+1) + \eta y(i+1),
\]
where $\eta=(-1)^t$ if $s=p^{i_1}+\cdots+p^{i_t}$, with $0=i_1<\cdots<i_t$, is a sum of $t$ different powers of $p$, and $\eta=0$ otherwise.
\end{enumerate}
\end{proposition}

\begin{proof}
For simplicity, let us write $\ell=\ell(m)$, $r=r(m)$, $\ell'=\ell(m-1)$, and $r'=r(m-1)$.
It will be convenient to have in mind the following two diagrams, corresponding to the top part and to the bottom part of the $p^2$-interval $[\ell,r-1]$, respectively:
\begin{figure}[H]
\centering
{\footnotesize
\begin{tikzpicture}[scale=7]
% ticks + labels (unchanged)
\foreach \x/\xtext in {
0.22/$x(\ell+1)$,
0.38/$y(\ell)$,
0.52/$x(\ell)$
} \draw[thick] (-1.3pt,\x) -- (-2.4pt,\x) node[left]{\xtext};
\foreach \x/\xtext in {
0.7/$x(\ell-1)$
} \draw[thick] (0.02,\x) -- (-0.02,\x) node[left]{\xtext}; 

% \foreach \y/\ytext in {
% 0.6/$\psi(\gamma_{i-1}(G))$
% } \draw[thick] (0.5pt,\y) node[right]{\ytext} -- (-0.5pt,\y);
% black dots + text
\filldraw (0,0) circle (0.4pt);
\filldraw (-0.116,0.15) circle (0.4pt)
node[anchor=east]{$\gamma_{\ell+3}(W(G))=[B(\gamma_{\ell'}(G)),W(G)]$};
\filldraw (0,0.3) circle (0.4pt)
node[anchor=west]{$\psi(\gamma_{\ell+1}(G))$};
\filldraw (-0.116,0.45) circle (0.4pt)
node[anchor=east]{$\gamma_{\ell+2}(W(G))=B(\gamma_{\ell'}(G))$};
\filldraw (0,0.6) circle (0.4pt)
node[anchor=west]{$\psi(\gamma_{\ell}(G))=\gamma_{\ell+1}(W(G))$};
\filldraw (0,0.8) circle (0.4pt);
% -------- connected diamonds (touching at points) --------
\def\aw{0.12} % half width of diamonds; increase for wider diamonds
% diamonds between successive y's: (top -> bottom)
\DiamondBetween{0.6}{0.3}{\aw}
\DiamondBetween{0.3}{0.0}{\aw}
% -------- extra vertical line and tick on top --------
\draw[thick] (0,0.6) -- (0,0.8); % vertical segment

\end{tikzpicture}}
\caption{
Top part of the $p^2$-interval
}
\label{figure top p2-interval}
\end{figure}

\begin{figure}[H]
\centering
{\footnotesize
\begin{tikzpicture}[scale=7]
% ticks + labels (unchanged)
\foreach \x/\xtext in {
0.28/$y(r-1)$,
0.42/$x(r-1)$
} \draw[thick] (-1.3pt,\x) -- (-2.4pt,\x) node[left]{\xtext};
\foreach \x/\xtext in {
0.1/$x(r)$
} \draw[thick] (0.02,\x) -- (-0.02,\x) node[left]{\xtext}; 
% \foreach \y/\ytext in {
% 0.6/$\psi(\gamma_{i-1}(G))$
% } \draw[thick] (0.5pt,\y) node[right]{\ytext} -- (-0.5pt,\y);
% black dots + text
\filldraw (0,0) circle (0.4pt)
node[anchor=west]{$\psi(\gamma_{r+1}(G))=\gamma_{r+2}(W(G))=B(\gamma_{r'}(G))$};
\filldraw (0,0.2) circle (0.4pt)
node[anchor=west]{$\psi(\gamma_{r}(G))$};
\filldraw (-0.116,0.35) circle (0.4pt)
node[anchor=east]{$\gamma_{r+1}(W(G))$};
\filldraw (0,0.5) circle (0.4pt)
node[anchor=west]{$\psi(\gamma_{r-1}(G))$};
% -------- connected diamonds (touching at points) --------
\def\aw{0.12} % half width of diamonds; increase for wider diamonds
% diamonds between successive y's: (top -> bottom)
\DiamondBetween{0.5}{0.2}{\aw}
% -------- extra vertical line and tick on top --------
\draw[thick] (0,0) -- (0,0.2); % vertical segment
\end{tikzpicture}}
\caption{
Bottom part of the $p^2$-interval.
}
\label{figure bottom p2-interval}
\end{figure}

Since $\ell-1=(\ell'-2)p+p-2$, \Cref{psi(x(i)) modulo} yields
\[
\psi(x(\ell-1)) \equiv \Delta^{p-2} (x(\ell'-1),1,\ldots,1)^{-\delta} \mod \gamma_{\ell+1}(W(G)),
\]
that is,
\begin{equation}
\label{psi x(l-1)}
\psi(x(\ell-1)) \equiv (x(\ell'-1)^{-1},x(\ell'-1)^{-2},\ldots,x(\ell'-1)^{-(p-1)},1)^{-\delta}
\mod \psi(\gamma_{\ell}(G)).
\end{equation}
Now observe that since $\ell-1\equiv -2 \mod p$, \Cref{i not -1 is good} applies and we can use
\Cref{p comms with ab} to get $\psi(x(\ell+p-1))$ modulo $\psi(\gamma_{\ell+p}(G))$.
By iterating this process, we obtain that
\begin{multline}
\label{psi x(l+lambda p-1)}
\psi(x(\ell+\lambda p-1))
\equiv
(x(\ell'+\lambda-1)^{-1},x(\ell'+\lambda-1)^{-2},\ldots,x(\ell'+\lambda-1)^{-(p-1)},1)^{-\delta}\\
\mod \psi(\gamma_{\ell+\lambda p}(G))
\end{multline}
for $0\le \lambda \le p^{m-2}$.
Note that the upper bound for $\lambda$ comes from the need that $\ell+\lambda p\le r$.
By taking the commutator of \eqref{psi x(l+lambda p-1)} with $\psi(ab)$, we get
\begin{multline}
\label{psi x(l+lambda p)}
\begin{aligned}
\psi(x(\ell+\lambda p))
&\equiv
\psi([x(\ell+\lambda p-1),a]) \, \psi([x(\ell+\lambda p-1),b])
\\
&\equiv
(x(\ell'+\lambda-1),\ldots,x(\ell'+\lambda-1))^{-\delta} \, \psi([x(\ell+\lambda p-1),b])
\end{aligned}
\\
\mod \psi(\gamma_{\ell+\lambda p+1}(G)).
\end{multline}

Assume first that $\lambda=p^{m-2}$.
Then \eqref{psi x(l+lambda p)} reads
\begin{equation}
\label{psi x(r)}
\psi(x(r))
\equiv
(x(r'-1),\ldots,x(r'-1))^{-\delta} \, \psi([x(r-1),b])
\mod \psi(\gamma_{r+1}(G)).
\end{equation}
On the other hand, since $r=(r'-2)p+p-1$, \Cref{psi(x(i)) modulo} yields
\begin{equation}
\label{psi x(r) bis}
\psi(x(r))
\equiv
(x(r'-1),\ldots,x(r'-1))^{-\delta}
\mod \psi(\gamma_{r+1}(G)),
\end{equation}
since $\gamma_{r+2}(W(G))=\psi(\gamma_{r+1}(G))$, as shown in \Cref{figure bottom p2-interval}.
From \eqref{psi x(r)} and \eqref{psi x(r) bis}, we deduce that $\psi([x(r-1),b])\in \psi(\gamma_{r+1}(G))$,
and consequently $[x(r-1),b]_L=0$.
This proves (i).

Thus we assume that $0\le \lambda<p^{m-2}$ in the remainder.
Consider now the case $\lambda=0$.
From the definition of $y(\ell)$, \eqref{psi x(l+lambda p)} gives
\[
\psi(x(\ell) y(\ell)^{-1}) \equiv (x(\ell'-1),\ldots,x(\ell'-1))^{-\delta}
\mod
\psi(\gamma_{\ell+1}(G)).
\]
Taking the commutator with $\psi(ab)$, we get
\begin{align*}
\psi(x(\ell+1) y(\ell+1)^{-1})
&\equiv [(x(\ell'-1),\ldots,x(\ell'-1))^{-\delta},\psi(a)]
\psi([x(\ell) y(\ell)^{-1},b])
\\
&\equiv
\psi([x(\ell),b])
\,
\psi([y(\ell),b])^{-1}
\mod
\psi(\gamma_{\ell+2}(G)).
\end{align*}
Now from \Cref{comm y(i) and b} we get $[y(\ell),b]\in \gamma_{\ell+2}(G)$, and so
the last congruence reduces to
\[
\psi(x(\ell+1) y(\ell+1)^{-1})
\equiv
\psi([x(\ell),b])
\mod
\psi(\gamma_{\ell+2}(G)).
\]
In the Lie algebra, this means that $[x(\ell),b]_L=x(\ell+1)-y(\ell+1)$.
This is the case $s=1$ of (iii).

Let us digress for a moment to examine the value of $\psi(y(\ell+\lambda p))$ for $0\le \lambda < p^{m-2}$.
To start with, from \eqref{psi x(l-1)} we get
\begin{multline*}
\psi(y(\ell))
=
\psi([x(\ell-1),b])
\equiv
([x(\ell'-1),a]^{-e_1},[x(\ell'-1),a]^{-2e_2},\ldots,[x(\ell'-1),a]^{-(p-1)e_{p-1}},1)^{-\delta}
\\
\mod \psi(\gamma_{\ell+1}(G)).
\end{multline*}
The tuple in this congruence belongs to $B(\gamma_{\ell'}(G))=\gamma_{\ell+2}(W(G))$, as we can see in
\Cref{figure top p2-interval}.
By \Cref{reduction to first components}, this tuple is congruent to
\[
([x(\ell'-1),a]^{-e_1-2e_2-\cdots-(p-1)e_{p-1}},1,\ldots,1)^{-\delta}
=
([x(\ell'-1),a],1,\ldots,1)^{\delta^2}
\]
modulo $\gamma_{\ell+3}(W(G))$, which is contained in $\psi(\gamma_{\ell+1}(G))$.
Hence
\begin{align*}
\psi(y(\ell))
&\equiv
([x(\ell'-1),a],1,\ldots,1)^{\delta^2}
\\
&\equiv
(x(\ell')y(\ell')^{-1},1,\ldots,1)^{\delta^2}
\mod \psi(\gamma_{\ell+1}(G)).
\end{align*}
Now from \Cref{psi(x(i)) modulo} we have
\[
\psi(x(\ell+1))
\equiv
(x(\ell'),1,\ldots,1)^{-\delta}
\mod \gamma_{\ell+3}(W(G)).
\]
Since $\delta\ne 0$ in $\F_p$ and $\gamma_{\ell+3}(W(G))\le \psi(\gamma_{\ell+1}(G))$, it follows that
\[
(x(\ell'),1,\ldots,1)
\in
\psi(\gamma_{\ell+1}(G)).
\]
Thus we conclude that
\[
\psi(y(\ell))
\equiv
(y(\ell')^{-1},1,\ldots,1)^{\delta^2}
\mod \psi(\gamma_{\ell+1}(G)).
\]

From this value we can obtain $\psi(y(\ell+p))$ modulo $\psi(\gamma_{\ell+p+1}(G))$ by using
\Cref{p comms with ab}; take into account that the required conditions are fulfilled thanks to
\Cref{exponent p} and \Cref{comm with B(G')}, since $\psi(y(\ell))\in B(\gamma_{\ell'}(G))$.
We can then iterate this process to get
\begin{equation}
\label{psi y(l+lambda p)}
\psi(y(\ell+\lambda p))
\equiv
(y(\ell'+\lambda)^{-1},1,\ldots,1)^{\delta^2}
\mod \psi(\gamma_{\ell+\lambda p+1}(G))
\end{equation}
for $0\le \lambda < p^{m-2}$.

Now we prove parts (ii) and (iii) of the proposition by induction on $m\ge 2$.
If $m=2$ then the only value $i\in [\ell(2),r(2)-2]$ that is congruent to $-2$ or $-1$ modulo $p$ is
$i=\ell(2)$.
The result holds in this case, as already mentioned (case $\lambda=0$ above).

Assume now that $m\ge 3$ and that the result is true for $m-1$.
We first deal with (ii).
Write $i=\ell-1+s$ with $s=\lambda p$, where $1\le \lambda<p^{m-2}$.
From \eqref{psi x(l+lambda p-1)}, we have
\begin{multline*}
\psi([x(i),b])
\equiv
([x(\ell'+\lambda-1),a]^{-e_1},[x(\ell'+\lambda-1),a]^{-2e_2},\ldots,
[x(\ell'+\lambda-1),a]^{-(p-1)e_{p-1}},1)^{-\delta}
\\
\mod \psi(\gamma_{i+2}(G)).
\end{multline*}
With a similar argument to the one used for $\psi(y(\ell))$, it follows that
\begin{equation}
\label{psi [x(i),b]}
\psi([x(i),b])
\equiv
([x(\ell'+\lambda-1),a],1,\ldots,1)^{\delta^2} \mod \psi(\gamma_{i+2}(G)).
\end{equation}
Now we have
\[
[x(\ell'+\lambda-1),a] \equiv x(\ell'+\lambda) [x(\ell'+\lambda-1),b]^{-1}
\mod \gamma_{\ell'+\lambda+1}(G),
\]
and $\ell'\le \ell'+\lambda-1\le \ell'+p^{m-2}-2=r'-2$.
Hence, the induction hypothesis can be applied to obtain the Lie bracket
$[x(\ell'+\lambda-1),b]_L$.
It follows that either
\begin{equation}
\label{comm x(l'+lambda-1) and a case 1}
[x(\ell'+\lambda-1),a] \equiv x(\ell'+\lambda)^{q} y(\ell'+\lambda)^{(-1)^{t+1}}
\mod \gamma_{\ell'+\lambda+1}(G)
\end{equation}
for some integer $q$, if $\lambda$ is the sum of $t$ different powers of $p$, or
\begin{equation}
\label{comm x(l'+lambda-1) and a case 2}
[x(\ell'+\lambda-1),a] \equiv x(\ell'+\lambda)^q
\mod \gamma_{\ell'+\lambda+1}(G)
\end{equation}
for some integer $q$, otherwise.
Now, as in previous cases, we can prove that
\begin{equation}
\label{vanishing tuple}
(x(\ell'+\lambda),1,\ldots,1) \in \psi(\gamma_{i+2}(G)).
\end{equation}
By plugging into \eqref{psi [x(i),b]} the values of $[x(\ell'+\lambda-1),a]$ obtained in
\eqref{comm x(l'+lambda-1) and a case 1} and \eqref{comm x(l'+lambda-1) and a case 2}, and taking into
consideration \eqref {vanishing tuple}, it follows that
\[
\psi([x(i),b])
\equiv
\begin{cases}
( y(\ell'+\lambda)^{(-1)^{t+1}},1,\ldots,1)^{\delta^2},
&
\text{if $s=\lambda p$ is a sum of $t$ different powers of $p$,}
\\[5pt]
(1,\ldots,1), & \text{otherwise,}
\end{cases}
\]
modulo $\psi(\gamma_{i+2}(G))$.
This, together with \eqref{psi y(l+lambda p)}, immediately implies that (ii) holds for $m$.

Let us now deal with part (iii).
Write $i=\ell-1+s$, where $s=1+\lambda p$.
Since we already know the result for $i=\ell$, we may assume that $1\le \lambda<p^{m-2}$.
We have $[x(i),b]\equiv x(i+1) \, [x(i),a]^{-1} \mod \gamma_{i+2}(G)$,
and on the other hand, $[x(i),a]\equiv [x(i-1),b,a] \mod \gamma_{i+2}(G)$, taking \eqref{psi x(l+lambda p)} into account.
Hence
\begin{equation}
\label{last [x(i),b]}
[x(i),b]\equiv x(i+1) \, [x(i-1),b,a]^{-1} \mod \gamma_{i+2}(G).
\end{equation}
Now the Lie bracket $[x(i-1),b]_L$ has already been determined in part (ii), and as a consequence,
\[
[x(i-1),b]
\equiv
\begin{cases}
y(i)^{(-1)^{t-1}},
&
\text{if $s-1$ is a sum of $t-1$ different powers of $p$,}
\\[5pt]
1, & \text{otherwise,}
\end{cases}
\]
modulo $\gamma_{i+1}(G)$.
Now observe that $[y(i),a]\equiv y(i+1) [y(i),b]^{-1} \equiv y(i+1) \mod \gamma_{i+2}(G)$, since
$[y(i),b]_L=0$ by \Cref{comm y(i) and b}.
Hence
\[
[x(i-1),b,a] \equiv y(i+1)^{(-1)^{t-1}} \quad \text{or} \quad 1 \mod \gamma_{i+2}(G),
\]
with the same case distinction as above.
Taking this value to \eqref{last [x(i),b]}, the result follows.
\end{proof}

As indicated in Corollary D in the introduction, for a fixed prime $p$, all the Lie algebras corresponding
to GGS-groups of FG-type are isomorphic and have an infinite dimensional centre.
Another interesting consequence of Theorem C is that the subspace
$J=\oplus_{m\ge 2} \, \langle y(\ell(m)),\ldots,y(r(m)-1) \rangle$ is a hypercentral graded ideal of $L(G)$.
This readily follows from items (ii) and (x) of Theorem C: $J$ is an ideal since it is invariant under taking commutators with
$ab$ and $b$, which generate $L(G)$, and hypercentrality follows similarly (starting from the fact that $y(r(m)-1)$ is central for every $m\ge 2$).
On the other hand, since we are factoring out all elements of the form $y(i)$ other than $y(1)$ and $y(p)$, the quotient
$L(G)/J$ has only two diamonds (i.e.\ graded components of dimension $2$), in degrees $1$ and $p$.
Thus $L(G)/J$ is an infinite dimensional graded $\F_p$-Lie algebra of coclass $2$.
These facts clearly generalise to other similar ideals inside $J$, as we describe in our following result.
Since the only added difficulty with respect to the case of $J$ is purely notational and no new ideas are needed,
we have decided to omit the proof.

\begin{proposition}
\label{factor lie algebras}
Let $G$ be a GGS-group of FG-type, and let $\mathbf{s}=(s(m))_{m\ge 2}$ be a sequence of integers, with
$s(m)\in [\ell(m),r(m)-1]$ for every $m\ge 2$.
Consider the subspace $J(\mathbf{s}) = \oplus_{m\ge 2} \, \langle y(i) \mid s(m)\le i\le r(m)-1 \rangle$ of $L(G)$.
Then the following hold:
\begin{enumerate}
\item
$J(\mathbf{s})$ is a hypercentral graded ideal of $L(G)$.
\item
The quotient $L(G)/J(\mathbf{s})$ is an infinite dimensional graded $\F_p$-Lie algebra of (possibly infinite) coclass
equal to
\[
2 + \sum_{m\ge 2} (s(m)-\ell(m)).
\]
\end{enumerate}
\end{proposition}

This result allows us to construct graded $\F_p$-Lie algebras $\oplus_{i\ge 1} \, L_i$ with all graded components of dimension
at most $2$, and whose distribution of the number of consecutive diamonds is \textit{à la carte}, provided only that we have two
isolated diamonds at the beginning (namely $L_1$ and $L_p$), and that the following diamonds appear in chains of length between
$0$ and $r(m)-\ell(m)=p^{m-1}$, for $m\ge 2$.
This way we can get, for example, graded $\F_p$-Lie algebras of arbitrary finite coclass at least $2$.
By choosing $s(m)=\ell(m)+1$, we can also get an algebra $M$ of infinite coclass where all diamonds are isolated.
The $\F_p$-Lie algebra $N$ associated to the Nottingham group, for $p>2$, shares this property with $M$, but is not isomorphic to $M$.
Indeed, the diamonds of $N$ appear exactly at the degrees congruent to $1\hspace{-3pt}\mod{(p-1)}$, as proved by York in his Ph.D.\ thesis \cite{York}
(a more easily available reference is Caranti's paper \cite{Caranti}), while in $M$ this happens at degrees $\ell(m)$, for all $m\ge 0$.

\vspace{8pt}

On the other hand, the structure constants for $L(G)$ obtained in Theorem C allow us to depict the so-called
Cayley graph of $L(G)$ with respect to the basis given by the elements $x(i)$ and $y(i)$.
Recall that the Cayley graph of a graded Lie algebra $L=\oplus_{i=1}^{\infty} \, L_i$ that is generated
by a finite set $S$ of elements of degree $1$ was introduced by Bartholdi and Grigorchuk in
\cite[Subsection 6.1]{Bartholdi-Grigorchuk}.
This is a directed, labelled, and weighted graph that is constructed as follows.
If we choose a basis $\BB_i$ for every homogeneous component $L_i$, the set of vertices of the Cayley graph
is the union $\cup_{i=1}^{\infty} \, \BB_i$.
As for the edges, if for every $v\in\BB_i$ and $s\in S$ we write
\[
[v,s] = \sum_{w\in\BB_{i+1}} \, \lambda_w w
\]
then, for every $w\in \BB_{i+1}$ such that $\lambda_w\ne 0$, we draw a directed edge going from $v$ to $w$,
labelled $s$ and with weight $\lambda_w$.
Observe that all Lie brackets between basis elements can be recovered from the Cayley graph, by using
bilinearity and the Jacobi identity.
Thus the description of the Lie algebra $L$ is equivalent to that of its Cayley graph.

If $G$ is a GGS-group of FG-type and we consider the Lie algebra $L=L(G)$, then we can choose
$\BB_i=\{x(i)\}$ or $\{x(i),y(i)\}$, as needed according to the value of $i$, and $S=\BB_1=\{b,ab\}$.
In order to simplify the representation of the graph, we draw edges with label $ab$ as solid
edges, and edges with label $b$ as dashed edges.
The weights are then placed in the middle of the corresponding edges, and if the weight is $1$, it is omitted.
We split the representation of the graph into three figures: one for the initial interval of degrees
$[1,p+1]$, a second one for the stretch $[p+1,p^2-p-1]$, and a final one corresponding to
$[\ell(m),\ell(m+1)]$, for an arbitrary $m\ge 2$.
To begin with, since
\[
[x(i),ab]_L=x(i+1),
\text{ for $i\ge 1$,}
\quad
[y(1),ab]_L=0,
\quad
[y(p),ab]_L=2x(p+1),
\]
\[
[x(i),b]_L=0
\text{ for $1\le i\le p-2$,}
\quad
[x(p-1),b]_L=y(p),
\quad
[x(p),b]_L=-x(p+1),
\]
and
\[
[y(1),b]_L=-x(2),
\quad
[y(p),b]_L=0,
\]
the initial part of the Cayley graph of $L(G)$ looks as follows:

\begin{figure}[!ht]
	\centering
	\begin{tikzpicture}[
		>={Stealth[length=1.5mm]},
		bn/.style={font=\scriptsize, inner sep=1.5pt},
		gh/.style={font=\scriptsize, inner sep=1.5pt, text=black!50},
		cl/.style={font=\tiny, inner sep=1pt},
		sa/.style={->, semithick},
		da/.style={->, semithick, densely dashed}
		]
		
		\node[bn] (x2) {$x(2)$};
		
		\node[bn, left=5mm of x2]
		(x1) {$x(1)$};
		
		\node[bn, below=7.5mm of x1]
		(y1) {$y(1)$};

		\node[gh, right=5mm of x2]
		(dots) {$\cdots$};
		
		\node[bn, right=5mm of dots]
		(xpm) {$x(p{-}1)$};

		\node[bn, right=5mm of xpm]
		(xp) {$x(p)$};
		
		\node[bn, below=7.5mm of xp]
		(yp) {$y(p)$};
		
		\node[bn, right=5mm of xp]
		(xpp) {$x(p{+}1)$};

		\draw[sa] (x1) -- (x2);
		
		\draw[sa] (x2) -- (dots);
		
		\draw[sa] (dots) -- (xpm);
		
		\draw[sa] (xpm) -- (xp);
		
		\draw[sa] (xp) -- (xpp);
		
		\draw[sa] (yp) --
		node[cl, pos=0.55, below right=1pt] {$2$}
		(xpp);
		
		\draw[da] (y1) --
		node[cl, pos=0.50, below right=1pt] {$-1$}
		(x2);
		
		\draw[da] (xpm) -- (yp);

		\draw[da]
 (xp) to[out=70,in=110]
 node[cl, midway, above=1pt] {$-1$}
 (xpp);
		
	\end{tikzpicture}
\caption{The initial part of the Cayley graph of $L(G)$, along the stretch
$[1,p+1]$ of degrees, for a GGS-group $G$ of FG-type with $\varepsilon=1$.}
\label{figure cayley second part}
\end{figure}

We similarly obtain the following subgraph for the interval $[p+1,p^2-p-1]$; note that $\ell(2)=p^2-p-1$.
When $p=3$, this interval is just $[4,5]$ and the graph has to be interpreted by placing $x(p+1)$ in the place of $x(p^2-p-2)$.
Thus the edges in the middle of the graph, and in particular those corresponding to degrees congruent to $1$ modulo $p$,
do not exist in this case.
All congruences in the next two figures are taken modulo $p$.

\begin{figure}[!ht]
\centering
\begin{tikzpicture}[
	>={Stealth[length=1.5mm]},
	bn/.style={font=\scriptsize, inner sep=1.5pt},
	gh/.style={font=\scriptsize, inner sep=1.5pt, text=black!50},
	cl/.style={font=\tiny, inner sep=1pt},
	sa/.style={->, semithick},
	da/.style={->, semithick, densely dashed},
	node distance=5mm
	]
	
	\node[bn] (xp1) {$x(p{+}1)$};
	
	\node[gh, right=of xp1] (dots1) {$\cdots$};
	
	\node[bn, right=of dots1] (x2pm1) {$x(2p{-}1)$};
	
	\node[bn, right=of x2pm1] (x2p) {$x(2p)$};
	
	\node[gh, right=of x2p] (dots2) {$\cdots$};
	
	\node[bn, right=of dots2] (xi) {$x(i)$};
	
	\node[bn, right=of xi] (xip1) {$x(i{+}1)$};
	
	\node[gh, right=of xip1] (dots3) {$\cdots$};
	
	\node[bn, right=of dots3]
 (xlast2) {$x(p^2{-}p{-}2)$};
 
 \node[bn, right=of xlast2]
 (xlast) {$x(p^2{-}p{-}1)$};
 
 \node[bn, below=7.5mm of xlast]
 (ylast) {$y(p^2{-}p{-}1)$};
 
 \node[bn, below=7.5mm of xlast] (ylast) {$y(p^2{-}p{-}1)$};
	
	\draw[sa] (xp1) -- (dots1);
	\draw[sa] (dots1) -- (x2pm1);
	\draw[sa] (x2pm1) -- (x2p);
	\draw[sa] (x2p) -- (dots2);
	\draw[sa] (dots2) -- (xi);
	\draw[sa] (xi) -- (xip1);
	\draw[sa] (xip1) -- (dots3);
	\draw[sa] (dots3) -- (xlast2);
	\draw[sa] (xlast2) -- (xlast);
	
	\draw[da]
	(x2pm1) to[out=70,in=110]
	(x2p);
	
	\draw[da]
	(xi) to[out=70,in=110]
	node[cl, midway, above=8pt] {$i\equiv -1$}
	(xip1);

	\draw[da] (xlast2) -- (ylast);

\end{tikzpicture}
\caption{The second part of the Cayley graph of $L(G)$, along the stretch
$[p+1,p^2-p-1]$ of degrees, for a GGS-group $G$ of FG-type with $\varepsilon=1$.}
\label{figure cayley initial part}
\end{figure}

Finally, \Cref{figure cayley generic} below shows the general pattern of the Cayley
graph along the stretch of degrees $[\ell(m),\ell(m+1)]$, for every $m\ge 2$:

\begin{figure}[!ht]
\centering
{\footnotesize
\begin{tikzpicture}[
	>={Stealth[length=1.5mm]},
	bn/.style={font=\scriptsize, inner sep=1.5pt},
	gh/.style={font=\scriptsize, inner sep=1.5pt},
	cl/.style={font=\tiny, inner sep=1pt},
	res/.style={font=\tiny, inner sep=1pt},
	sa/.style={->, semithick},
	da/.style={->, semithick, densely dashed}
	]
	
	\node[bn]
	(xA1) {$x(\ell)$};
	
	\node[bn, right=5mm of xA1]
	(xA2) {$x(\ell{+}1)$};
	
	\node[gh, right=5mm of xA2]
	(dA3) {$\cdots$};
	
	\node[bn, right=5mm of dA3]
	(xA4) {$x(i)$};
	
	\node[bn, right=5mm of xA4]
	(xA5) {$x(i{+}1)$};
	
	\node[bn, right=5mm of xA5]
	(xA6) {$x(i{+}2)$};
	
	\node[gh, right=5mm of xA6]
	(dA7) {$\cdots$};
	
	\node[bn, right=5mm of dA7]
	(xA8) {$x(r{-}1)$};
	
	\node[bn, below=7.5mm of xA1]
	(yA1) {$y(\ell)$};
	
	\node[bn, right=5mm of yA1]
	(yA2) {$y(\ell{+}1)$};
	
	\node[gh, right=5mm of yA2]
	(eA3) {$\cdots$};
	
	\node[bn, right=5mm of eA3]
	(yA4) {$y(i)$};
	
	\node[bn, right=5mm of yA4]
	(yA5) {$y(i{+}1)$};
	
	\node[bn, right=5mm of yA5]
	(yA6) {$y(i{+}2)$};
	
	\node[gh, right=5mm of yA6]
	(eA7) {$\cdots$};
	
	\node[bn, right=5mm of eA7]
	(yA8) {$y(r{-}1)$};
	
	\node[res, above=4mm of xA1]
	{$\ell\equiv -1$};
	
	\node[res, above=4mm of xA4]
	{$i\equiv -2$};
	
	\node[res, above=4mm of xA5]
	{$i{+}1\equiv -1$};
	
	\draw[sa] (xA1) -- (xA2);
	\draw[sa] (xA2) -- (dA3);
	\draw[sa] (dA3) -- (xA4);
	\draw[sa] (xA4) -- (xA5);
	\draw[sa] (xA5) -- (xA6);
	\draw[sa] (xA6) -- (dA7);
	\draw[sa] (dA7) -- (xA8);
	
	\draw[sa] (yA1) -- (yA2);
	\draw[sa] (yA2) -- (eA3);
	\draw[sa] (eA3) -- (yA4);
	\draw[sa] (yA4) -- (yA5);
	\draw[sa] (yA5) -- (yA6);
	\draw[sa] (yA6) -- (eA7);
	\draw[sa] (eA7) -- (yA8);
	
	\draw[da]
	(xA1) to[bend left=45] (xA2);
	
	\draw[da]
	(xA1) --
	node[cl, pos=0.50, below left=2pt] {$-1$}
	(yA2);
	
	\draw[da]
	(xA4) --
	node[cl, pos=0.50, below left=2pt] {$\eta$}
	(yA5);
	
	\draw[da]
	(yA4) to[bend right=45] (yA5);
	
	\draw[da]
	(xA5) to[bend left=45] (xA6);
	
	\draw[da]
	(xA5) --
	node[cl, pos=0.50, below left=2pt] {$\eta'$}
	(yA6);
	
	\draw[
	decorate,
	decoration={brace, mirror, amplitude=4pt}
	]
	([yshift=-6mm]yA1.south west)
	--
	([yshift=-6mm]yA8.south east)
	node[midway, below=6pt, font=\scriptsize]
	{$|\gamma_i(G):\gamma_{i+1}(G)|=p^2$
		\quad ($p^{m-1}$ values of $i$)};
	
\end{tikzpicture}

\vspace{16pt}

\begin{tikzpicture}[
	>={Stealth[length=1.5mm]},
	bn/.style={font=\scriptsize, inner sep=1.5pt},
	gh/.style={font=\scriptsize, inner sep=1.5pt},
	cl/.style={font=\tiny, inner sep=1pt},
	res/.style={font=\tiny, inner sep=1pt},
	sa/.style={->, semithick},
	da/.style={->, semithick, densely dashed}
	]

	\node[bn] (xB0) {$x(r{-}1)$};
	\node[bn, below=7.5mm of xB0] (yB0) {$y(r{-}1)$};

	\node[bn, right=5mm of xB0] (xB1) {$x(r)$};
	\node[bn, right=5mm of xB1] (xB2) {$x(r{+}1)$};
	\node[gh, right=5mm of xB2] (dB3) {$\cdots$};
	
	\node[bn, right=5mm of dB3] (xI) {$x(i)$};
	\node[bn, right=5mm of xI] (xIp) {$x(i{+}1)$};
	\node[gh, right=5mm of xIp] (dB4) {$\cdots$};
	
	\node[bn, right=5mm of dB4] (xL1) {$x(\ell'{-}1)$};
	\node[bn, right=5mm of xL1] (xL2) {$x(\ell')$};
	
	%\node[bn, below=7.5mm of xIp] (yIp) {$y(i{+}1)$};
	\node[bn, below=7.5mm of xL2] (yL2) {$y(\ell')$};
		
	\node[res, above=4mm of xB1]
	{$r\equiv -1$};
	
	\node[res, above=4mm of xI]
	{$i\equiv -1$};
	
	\node[res, above=4mm of xL1]
	{$i=\ell'{-}1$};

	\draw[sa] (xB0) -- (xB1);
	\draw[sa] (xB1) -- (xB2);
	\draw[sa] (xB2) -- (dB3);
	\draw[sa] (dB3) -- (xI);
	\draw[sa] (xI) -- (xIp);
	\draw[sa] (xIp) -- (dB4);
	\draw[sa] (dB4) -- (xL1);
	\draw[sa] (xL1) -- (xL2);

	\draw[da]
	(xB1) to[bend left=45] (xB2);
	
	\draw[da]
	(xI) to[bend left=45] (xIp);
	
	% \draw[da]
	% (xI) --
	% node[cl, pos=0.50, below left=2pt] {$\ell'{-}1$}
	% (yIp);
	
	\draw[da]
	(xL1) -- (yL2);

	\draw[
	decorate,
	decoration={brace, mirror, amplitude=4pt}
	]
	([yshift=-6mm]xB1.south west)
	--
	([yshift=-6mm]xL1.south east)
	node[midway, below=7pt, font=\scriptsize]
	{$|\gamma_i(G):\gamma_{i+1}(G)|=p$};
	
\end{tikzpicture}}
\caption{The Cayley graph of $L(G)$ along a generic stretch
$[\ell(m),\ell(m+1)]$ of degrees, with $m\ge 2$, for a GGS-group $G$ of
FG-type with $\varepsilon=1$.
We write $\ell=\ell(m)$, $r=r(m)$ and $\ell'=\ell(m+1)$.
Also $\eta=\eta(s)$ and $\eta'=\eta(s+1)$, where $s=i-\ell+1$ and, for
a positive integer $s$, we set $\eta(s)=(-1)^t$ if $s$ is the sum of
$t$ different powers of $p$, and $\eta(s)=0$ otherwise.
An edge labelled $\eta$ or $\eta'$ is to be omitted whenever these weights are
equal to $0$.}
\label{figure cayley generic}
\end{figure}

For a specific example, we next draw the first twenty degrees of the Cayley graph in the case
$p=3$.
Note that this corresponds to the original Fabrykowski-Gupta group.
The Cayley graph of this group was described in \cite[Theorem 3.13]{Bartholdi}; compare the following diagram
with Figure 4 in that paper.

\begin{figure}[!ht]
\centering
{\footnotesize
\begin{tikzpicture}[x=1.3cm, y=1cm,
 >={Stealth[length=1.7mm]},
 bn/.style={font=\scriptsize, inner sep=2pt},
 gh/.style={font=\scriptsize, inner sep=2pt, text=black!50},
 cl/.style={font=\tiny, inner sep=1pt},
 sa/.style={->, semithick},
 da/.style={->, semithick, densely dashed}
]

\node[bn] (x1) at (0,0) {$x(1)$};
\node[bn] (x2) at (1,0) {$x(2)$};
\node[bn] (x3) at (2,0) {$x(3)$};
\node[bn] (x4) at (3,0) {$x(4)$};
\node[bn] (x5) at (4,0) {$x(5)$};
\node[bn] (x6) at (5,0) {$x(6)$};
\node[bn] (x7) at (6,0) {$x(7)$};
\node[bn] (x8) at (7,0) {$x(8)$};
\node[bn] (x9) at (8,0) {$x(9)$};
\node[bn] (x10) at (9,0) {$x(10)$};

\node[bn] (y1) at (0,-1.1) {$y(1)$};
\node[bn] (y3) at (2,-1.1) {$y(3)$};
\node[bn] (y5) at (4,-1.1) {$y(5)$};
\node[bn] (y6) at (5,-1.1) {$y(6)$};
\node[bn] (y7) at (6,-1.1) {$y(7)$};

\draw[sa] (x1) -- (x2);
\draw[sa] (x2) -- (x3);
\draw[sa] (x3) -- (x4);
\draw[sa] (x4) -- (x5);
\draw[sa] (x5) -- (x6);
\draw[sa] (x6) -- (x7);
\draw[sa] (x7) -- (x8);
\draw[sa] (x8) -- (x9);
\draw[sa] (x9) -- (x10);

\draw[sa]
 (y3) --
 node[cl, pos=0.50, below right=2pt] {$-1$}
 (x4);

\draw[sa] (y5) -- (y6);
\draw[sa] (y6) -- (y7);

\draw[da]
 (y1) --
 node[cl, pos=0.50, below right=2pt] {$-1$}
 (x2);

\draw[da] (x2) -- (y3);

\draw[da]
 (x3) to[bend left=48]
 node[cl, midway, above=0.5pt] {$-1$}
 (x4);

\draw[da] (x4) -- (y5);

\draw[da]
 (x5) to[bend left=48] (x6);

\draw[da]
 (x5) --
 node[cl, pos=0.50, below left=2pt] {$-1$}
 (y6);

\draw[da]
 (x8) to[bend left=48] (x9);

\end{tikzpicture}

\vspace{16pt}

\begin{tikzpicture}[
	>={Stealth[length=1.5mm]},
	bn/.style={font=\scriptsize, inner sep=1.5pt},
	gh/.style={font=\scriptsize, inner sep=1.5pt},
	cl/.style={font=\tiny, inner sep=1pt},
	sa/.style={->, semithick},
	da/.style={->, semithick, densely dashed}
	]
	
	\node[gh] (x10g) {$x(10)$};
	
	\node[bn, right=5mm of x10g]
	(x11) {$x(11)$};
	
	\node[bn, right=5mm of x11]
	(x12) {$x(12)$};
	
	\node[bn, right=5mm of x12]
	(x13) {$x(13)$};
	
	\node[bn, right=5mm of x13]
	(x14) {$x(14)$};
	
	\node[bn, right=5mm of x14]
	(x15) {$x(15)$};
	
	\node[bn, right=5mm of x15]
	(x16) {$x(16)$};
	
	\node[bn, right=5mm of x16]
	(x17) {$x(17)$};
	
	\node[bn, right=5mm of x17]
	(x18) {$x(18)$};
	
	\node[bn, right=5mm of x18]
	(x19) {$x(19)$};
	
	\node[bn, right=5mm of x19]
	(x20) {$x(20)$};
	
	\node[bn, below=7.5mm of x11]
	(y11) {$y(11)$};
	
	\node[bn, right=5mm of y11]
	(y12) {$y(12)$};
	
	\node[bn, right=5mm of y12]
	(y13) {$y(13)$};
	
	\node[bn, right=5mm of y13]
	(y14) {$y(14)$};
	
	\node[bn, right=5mm of y14]
	(y15) {$y(15)$};
	
	\node[bn, right=5mm of y15]
	(y16) {$y(16)$};
	
	\node[bn, right=5mm of y16]
	(y17) {$y(17)$};
	
	\node[bn, right=5mm of y17]
	(y18) {$y(18)$};
	
	\node[bn, right=5mm of y18]
	(y19) {$y(19)$};

	\draw[sa] (x10g) -- (x11);
	\draw[sa] (x11) -- (x12);
	\draw[sa] (x12) -- (x13);
	\draw[sa] (x13) -- (x14);
	\draw[sa] (x14) -- (x15);
	\draw[sa] (x15) -- (x16);
	\draw[sa] (x16) -- (x17);
	\draw[sa] (x17) -- (x18);
	\draw[sa] (x18) -- (x19);
	\draw[sa] (x19) -- (x20);
	
	\draw[sa] (y11) -- (y12);
	\draw[sa] (y12) -- (y13);
	\draw[sa] (y13) -- (y14);
	\draw[sa] (y14) -- (y15);
	\draw[sa] (y15) -- (y16);
	\draw[sa] (y16) -- (y17);
	\draw[sa] (y17) -- (y18);
	\draw[sa] (y18) -- (y19);
	
	\draw[da]
	(x10g) -- (y11);
	
	\draw[da]
	(x11) to[bend left=48] (x12);
	
	\draw[da]
	(x11) --
	node[cl, pos=0.50, below left=2pt] {$-1$}
	(y12);
	
	\draw[da]
	(x13) --
	node[cl, pos=0.50, below left=2pt] {$-1$}
	(y14);
	
	\draw[da]
	(x14) to[bend left=48] (x15);
	
	\draw[da]
	(x14) -- (y15);
	
	\draw[da]
	(x17) to[bend left=48] (x18);
	
	\draw[da]
	(y13) to[bend right=48] (y14);
	
	\draw[da]
	(y16) to[bend right=48] (y17);
	
\end{tikzpicture}}
\caption{The Cayley graph of $L(G)$ in degrees $1\le i\le 20$ for the
Fabrykowski-Gupta group $(p=3)$.}

\label{figure cayley p3}
\end{figure}
\FloatBarrier

For the construction of the graph in \Cref{figure cayley p3}, note first that $\ell(1)=3$, $r(1)=4$,
$\ell(2)=5$, $r(2)=8$, $\ell(3)=11$, and $r(3)=20$.
Then take into account that the weights $\eta(s)$ follow from the equalities
$[x(13),b]_L=-y(14)$ and $[x(14),b]_L=x(15)+y(15)$, corresponding to
$s=3$ and $s=4=1+3$, and from $[x(16),b]_L=0$ and $[x(17),b]_L=x(18)$,
which hold since $s=6$ and $s=7$ are not sums of different powers of $3$.
Finally, note that the edge $y(p)\overset{2}{\longrightarrow} x(p+1)$ of \Cref{figure cayley second part}
transforms into $y(3)\overset{-1\ }{\longrightarrow} x(4)$, since we are working over $\F_3$.

\bibliographystyle{unsrt}

\begin{thebibliography}{1}

\bibitem{Abert05}
M.\ Ab\'ert,
Group laws and free subgroups in topological groups,
\textit{Bull.\ London Math.\ Soc.\/} \textbf{37} (2005), 525--534.

\bibitem{Abert06}
M.\ Ab\'ert,
Representing graphs by the non-commuting relation,
\textit{Publ.\ Math.\ Debrecen\/} \textbf{69} (2006), 261--269.

\bibitem{Bartholdi}
L.\ Bartholdi,
Lie algebras and growth in branch groups,
\textit{Pacific J.\ Math.\/} \textbf{218} (2005), 241--282.

\bibitem{Bartholdi2}
L.\ Bartholdi,
Self similar Lie algebras,
\textit{J.\ Eur.\ Math.\ Soc.\/} \textbf{17} (2015), 3113--3151.

\bibitem{Bartholdi-Eick-Hartung}
L.\ Bartholdi, B.\ Eick, R.\ Hartung,
A nilpotent quotient algorithm for certain infinitely presented groups and its applications,
\textit{Internat.\ J.\ Algebra Comp.\/} \textbf{18} (2008), 1--23.

\bibitem{Bartholdi-Grigorchuk}
L.\ Bartholdi, R.I.\ Grigorchuk,
Lie methods in growth of groups and groups of finite width,
in \textit{Computational and Geometric Aspects of Modern Algebra\/},
Cambridge University Press, 2000, pp.\ 1--27.

\bibitem{Bartholdi-Pochon}
L.\ Bartholdi, F.\ Pochon,
On growth and torsion of groups,
\textit{Groups Geom.\ Dyn.\/} \textbf{3} (2009), 525--539.

\bibitem{Baumslag}
G.\ Baumslag,
Wreath products and $p$-groups,
\textit{Proc.\ Cambridge Philos.\ Soc.\/} \textbf{55} (1959), 224--231.

\bibitem{Buckley}
J.T.\ Buckley,
Polynomial functions and wreath products,
\textit{Illinois J.\ Math.\/} \textbf{14} (1970), 274--282.

\bibitem{Caranti}
A.\ Caranti,
Presenting the graded Lie algebra associated to the Nottingham group,
\textit{J.\  Algebra\/} \textbf{198} (1997), 266--289.

\bibitem{Eberhard}
S.\ Eberhard,
Is there anything known about the lower central series of a group $G\wr C_p$?,
\textit{MathOverflow\/}, URL (version: 2023-03-10): \texttt{https://mathoverflow.net/questions/442171}

\bibitem{Fabrykowski-Gupta85}
J.\ Fabrykowski, N.\ Gupta,
On groups with sub-exponential growth functions,
\textit{J.\ Indian Math.\ Soc.\/} \textbf{49} (1985), 249--256.

%\bibitem{Fabrykowski-Gupta91}
%J.\ Fabrykowski, N.\ Gupta,
%On groups with sub-exponential growth functions, II,
%\textit{J.\ Indian Math.\ Soc.\/} \textbf{56} (1991), 217--228.

\bibitem{Gustavo-Brasilia}
G.A.\ Fernández-Alcober,
An introduction to finite $p$-groups: regular $p$-groups and groups of maximal class,
\textit{Mat.\ Contemp.\/} \textbf{20} (2001), 155--226.

\bibitem{Gustavo-Alejandra-Jone}
G.A.\ Fernández-Alcober, A.\ Garrido, J. Uria-Albizuri,
On the congruence subgroup property for GGS-groups,
\textit{Proc.\ Amer.\ Math.\ Soc.\/} \textbf{145} (2017), 3311--3322.

\bibitem{Gustavo}
G.A.\ Fernández-Alcober, A.\ Zugadi-Reizabal,
GGS-groups: order of congruence quotients and Hausdorff dimension,
\textit{Trans.\ Amer.\ Math.\ Soc.\/} \textbf{366} (2014), 1993--2017.

\bibitem{Francoeur-Thillaisundaram}
D.\ Francoeur, A.\ Thillaisundaram,
Maximal subgroups of nontorsion Grigorchuk-Gupta-Sidki groups,
\textit{Canadian Math. Bull.\/} \textbf{65} (2022), 825--844.

\bibitem{GAP}
The GAP Group, GAP -- Groups, Algorithms, and Programming, Version 4.15.0, 2025.
(\texttt{https://www.gap-system.org})

\bibitem{Grigorchuk80}
R.I.\ Grigorchuk,
On Burnside’s problem on periodic groups,
\textit{Functional Anal.\ Appl.\/} \textbf{14} (1980), 41--43.

\bibitem{Grigorchuk85}
R.I.\ Grigorchuk,
Degrees of growth of finitely generated groups and the theory of invariant means,
\textit{Math.\ USSR-Izv.\/} \textbf{25} (1985), 259--300.

\bibitem{Grigorchuk-New Horizons}
R.I.\ Grigorchuk,
Just infinite branch groups,
in \textit{New Horizons in Pro-$p$ Groups\/},
Progr.\ Math., vol.\ 184, Birkh\"auser, 2000, pp. 121--179.

\bibitem{Gul-Uria}
\c{S}.\ G\"ul, J.\ Uria-Albizuri,
Grigorchuk-Gupta-Sidki groups as a source for Beauville surfaces,
\textit{Groups Geom.\ Dyn.\/} \textbf{14} (2020), 689--704.

\bibitem{Gupta-Sidki}
N.D.\ Gupta, S.N.\ Sidki,
On the Burnside problem for periodic groups,
\textit{Math.\ Z.\/} \textbf{182} (1983), 385--388.

\bibitem{Kaloujnine}
L.\ Kaloujnine,
La structure des $p$-groupes de Sylow des groupes symétriques finis,
\textit{Ann.\ Sci.\ Ec.\ Norm.\ Supér.\ (3)\/} \textbf{65} (1948), 239--276.

\bibitem{Klaas-LeedhamGreen-Plesken}
G.\ Klaas, C.R.\ Leedham-Green, W.\ Plesken,
\textit{Linear Pro-$p$ Groups of Finite Width\/},
Springer, 1997.

\bibitem{Klopsch-Thillaisundaram}
B.\ Klopsch, A.\ Thillaisundaram,
Normal subgroups of non-torsion multi-EGS groups,
arXiv:2509.08576 [math.GR].

\bibitem{LeedhamGreen}
C.R.\ Leedham-Green,
The structure of finite $p$-groups,
\textit{J.\ London Math.\ Soc.\/} \textbf{50} (1994), 49--67.

\bibitem{LeedhamGreen-McKay}
C.R.\ Leedham-Green, S.\ McKay,
\textit{The Structure of Groups of Prime Power Order\/},
Oxford University Press, 2002.

\bibitem{Liebeck}
H.\ Liebeck,
Concerning nilpotent wreath products,
\textit{Proc.\ Cambridge Philos.\ Soc.\/} \textbf{58} (1962), 443--451.

\bibitem{automgrp}
Y.\ Muntyan, D.\ Savchuk, AutomGrp, Automata groups GAP package, Version 1.3.3, 2025.
(\texttt{https://gap-packages.github.io/automgrp})

\bibitem{Pervova}
E.L.\ Pervova,
Maximal subgroups of some non locally finite $p$-groups,
\textit{Internat.\ J.\ Algebra Comput.\/} \textbf{15} (2005), 1129--1150.

\bibitem{Petschick}
J.M.\ Petschick,
On conjugacy of GGS-groups,
\textit{J.\ Group Theory\/} \textbf{22} (2019), 347--358.

\bibitem{Ribes-Zalesskii}
L.\ Ribes, P.\ Zalesskii,
\textit{Profinite Groups},
Springer, 2000.

\bibitem{Robinson-Finiteness1}
D.J.S.\ Robinson,
\textit{Finiteness Conditions and Generalized Soluble Groups, Part 1\/},
Springer, 1972.

\bibitem{Rozhkov}
A.V.\ Rozhkov,
Lower central series of a group of tree automorphisms,
\textit{Math.\ Notes\/} \textbf{60} (1996), 165--174.

\bibitem{Rozhkov-2}
A.V.\ Rozhkov,
\textit{Conditions of finiteness in groups of automorphisms of trees\/},
Habilitation thesis, Chelyabinsk, 1996.

\bibitem{Shalev}
A.\ Shalev,
The structure of finite $p$-groups: effective proof of the coclass conjectures,
\textit{Invent.\ Math.\/} \textbf{115} (1994), 315--345.

\bibitem{Shield}
D.\ Shield,
The class of a nilpotent wreath product,
\textit{Bull.\ Austral.\ Math.\ Soc.\/} \textbf{17} (1977), 53--89.

\bibitem{Suzuki}
M.\ Suzuki,
\textit{Group Theory II\/},
Springer, 1986.

\bibitem{Vieira}
A.C.\ Vieira,
On the lower central series and the derived series of the Gupta-Sidki $3$-group,
\textit{Comm.\ Algebra\/} \textbf{26} (1998), 1319--1333.

\bibitem{Vovkivsky}
T.\ Vovkivsky,
Infinite torsion groups arising as generalizations of the second Grigorchuk group,
in \textit{Algebra (Moscow, 1998)\/}, de Gruyter, 2000, pp.\ 357--377.

\bibitem{York}
I.O. York, 
\textit{The Group of Formal Power Series under Substitution\/},
Ph.D.\ thesis, Nottingham, 1990.

\bibitem{Zelmanov}
E.I.\ Zelmanov,
Talk at the ESF conference on algebra and discrete mathematics 
``Group Theory: from Finite to Infinite'', Castelvecchio Pascoli,
13--18 July 1996.

\end{thebibliography}

\end{document}